\documentclass[10pt,oneside]{amsart}
\usepackage[ansinew]{inputenc}
\usepackage{graphics}
\usepackage{amssymb}
\usepackage{latexsym}
\usepackage{graphics}
\usepackage{fancybox}
\usepackage{amsmath,amsfonts}
\usepackage{tikz}
\usepackage{mathrsfs}
\usepackage{geometry}
\usepackage{float}
\usepackage{lineno}
\usepackage[colorlinks = true,
linkcolor = blue, citecolor=red]{hyperref}
\def\2010mathclass#1{\par%
	\insert\footins{\footnotesize{\it\textup{2010} Mathematics
			Subject Classification:} #1}}

\numberwithin{equation}{section}
\newtheorem{Theorem}{Theorem}[section]
\newtheorem{Example}{Example}[section]
\newtheorem{Definition}{Definition}[section]
\newtheorem{Proposition}{Proposition}[section]
\newtheorem{Lemma}{Lemma}[section]
\newtheorem{Corollary}{Corollary}[section]

\newtheorem{Remark}{Remark}[section]
\newtheorem{assum}{Assumption} [section]

\def\supp{\operatorname{{supp}}}

\newcommand{\fonte}[1]{\LARGE #1}

\begin{document}
	
	\title[wave equation]{Decay rate estimates for the wave equation with subcritical semilinearities and locally distributed nonlinear dissipation}
	
	\author[Cavalcanti]{ M. M. Cavalcanti}
	\address{ Department of Mathematics, State University of
		Maring\'a, 87020-900, Maring\'a, PR, Brazil}
	\email{mmcavalcanti@uem.br}
	\thanks{Research of Marcelo M. Cavalcanti is partially supported by the CNPq Grant 300631/2003-0}
	
	\author[Domingos Cavalcanti]{ V. N. Domingos Cavalcanti}
	\address{ Department of Mathematics, State University of
		Maring\'a, 87020-900, Maring\'a, PR, Brazil}
	\email{vndcavalcanti@uem.br}
	\thanks{Research of Val\'eria N.Domingos Cavalcanti is partially supported by the CNPq Grant 304895/2003-2}
	
	\author[Gonzalez Martinez]{V. H. Gonzalez Martinez}
	\address{Department of Mathematics, State University of
		Maring\'a, 87020-900, Maring\'a, PR, Brazil}
	\email{victor.hugo.gonzalez.martinez@gmail.com}
	\thanks{Research of Victor Hugo Gonzalez Martinez is partially supported by CAPES Grant 88882.449176/2019-01.}
	
	\author[\"{O}zsari]{T. \"{O}zsar{\i}}
	\address{Department of Mathematics, Bilkent University,
		Ankara, 06800 Turkey}
	\email{turker.ozsari@bilkent.edu.tr}
	\thanks{Research of Türker \"{O}zsar{\i} is supported by the Science Academy's Young Scientist Award in Mathematics (BAGEP 2020)}
	\2010mathclass{35L05, 35L20, 35B40.}%

	\begin{abstract}
		We study the stabilization and the wellposedness of solutions of the wave equation with subcritical semilinearities and locally distributed nonlinear dissipation.  The novelty of this paper is that we deal with the difficulty that the main equation does not have good nonlinear structure amenable to a direct proof of a priori bounds and a desirable observability inequality. It is well known that observability inequalities play a critical role in characterizing the long time behaviour of solutions of evolution equations, which is the main goal of this study.   In order to address this, we truncate the nonlinearities, and thereby construct approximate solutions for which it is possible to obtain a priori bounds and prove the essential observability inequality.  The treatment of these approximate solutions is still a challenging task and requires the use of Strichartz estimates and some microlocal analysis tools such as microlocal defect measures.  We include an appendix on the latter topic here to make the article self contained and supplement details to proofs of some of the theorems which can be already be found in the lecture notes of \cite{Burq-Gerard}. Once we establish essential observability properties for the approximate solutions, it is not difficult to prove that the solution of the original problem also possesses a similar feature via a delicate passage to limit. In the last part of the paper, we establish various decay rate estimates for different growth conditions on the nonlinear dissipative effect.  We in particular generalize the known results on the subject to a considerably larger class of dissipative effects.
	\end{abstract}
	
	\maketitle
	
	

	\bigskip
	
	
	\medskip

	
	%
	\tableofcontents
	\newpage
	\section{Introduction}\label{section1}
	
	In this article, our purpose is to study the wellposedness and decay properties of solutions of the wave equation with subcritical semilinearities and locally distributed nonlinear dissipation.  More precisely, we consider the initial boundary value problem given by
	\begin{equation}\label{mainproblem}
	\begin{cases}
	\partial_t^2 u -\Delta u + f(u) + a(x) g(\partial_t u) = 0 & ~ \hbox{ in } ~ \Omega \times (0, +\infty),\\
	u=0 & ~ \hbox{ on } ~ \partial \Omega \times (0,+\infty ),\\
	u(x,0)=u_0(x), ~ u_t(x,0)=u_1(x) & ~ \hbox{ in } ~ \Omega,
	\end{cases}
	\end{equation}
	where $\Omega\subset \mathbb{R}^n$ ($n\ge 1$) is a bounded domain with a smooth boundary $\partial \Omega$ and $f$ and $g$ are real valued functions satisfying Assumption \ref{assmpflabel} and Assumption \ref{assmpglabel} below, respectively:
	
	\begin{assum}\label{assmpflabel}
		$f\in C^2(\mathbb{R})$ with
		\begin{eqnarray}\label{main ass on f}
		f(0)=0, \ \ \ |f^{(j)}(s)| \leq k_{0}(1 + |s|)^{p-j}, j=1,2, \forall s \in \mathbb{R},
		\end{eqnarray} for some constant $k_0>0$, where \begin{equation}\label{ass on f'}
		1\leq p < \frac{n+2}{n-2} \hbox{ if } n\geq 3  \hbox{ and } p\geq 1 \hbox{ if } n=1,2.
		\end{equation}
		In addition,
		\begin{eqnarray}\label{ass on F}
		0 \leq F(s) \leq f(s) s, \forall s\in \mathbb{R},
		\end{eqnarray}
		where $F(\lambda):= \int_0^\lambda f(s)\,ds$.
	\end{assum}
	
	\begin{Remark}
		\eqref{main ass on f} in particular gives
		\begin{eqnarray}\label{ass on f}
		|f(r) - f(s)| \leq c \left(1 + |s|^{p-1} + |r|^{p-1} \right)|r-s|, \forall s,r\in \mathbb{R}
		\end{eqnarray}
		for some $c>0$.
	\end{Remark}
	
	\begin{assum}\label{assmpglabel}
		The feedback function $g:\mathbb{R}\rightarrow \mathbb{R}$  is continuous, monotone increasing and also satisfies
		the following:
		\begin{eqnarray}\label{ass on g}
		&& g( s)s>0 \hbox{ for }s\neq 0, \hbox{ and } \tilde{m} s^{r+1}\leq g\left(s\right)s \leq \tilde{M} s^{r+1} \hbox{ for }\left\vert s\right\vert >1,
		\end{eqnarray}
		where $\tilde{m}$ and $\tilde{M}$ are positive constants. In addition, we assume there is a continuous, concave and strictly increasing function $h_0:\mathbb{R}\to \mathbb{R}$
		satisfying
		\begin{equation}\label{ass on h}
		h_0(g(s)s)\geq s^2+g(s)^2,~ |s|\leq1.
		\end{equation}
	\end{assum}
	
	It is not difficult to construct such a function by using (\ref{ass on g}).
	
	In what follows, we will write $a(s) \lesssim b(s)$ in the sense $a(s) \leq C b(s)$, where $C>0$ is a constant. Moreover, $a(s)\sim b(s)$ is used whenever $a(s)\lesssim b(s)$ {and} $b(s)\lesssim a(s)$. We will classify the nonlinear feedbacks of various growth rates by using the notion of the polynomial order at infinity:
	\begin{Definition}\label{def:order}
		A monotone increasing map $g:\mathbb{R}\to\mathbb{R}$, $g(0)=0$, is of  the order $r:= \mathcal{O} (g)\geq 0$ at infinity if
		\begin{equation}\label{i:order}
		\begin{split}
		|s|^{r+1} \sim  g(s)s&\quad \text{for}\quad |s|\geq 1.
		\end{split}
		\end{equation}
		If the order $r$ is bigger than, less than, or equal to $1$ we say the map $f$ is superlinear, sublinear, or linearly bounded at infinity, respectively.
	\end{Definition}

	We in addition assume that the damping coefficient satisfies the following:
	
	\begin{assum}
		$a(x)$ is nonnegative and belongs to $C^0(\overline{\Omega})$. Moreover, $a(x) \geq a_0>0$ on a neighborhood $\omega$ of the entire boundary $\partial \Omega$.
	\end{assum}
	
	We associate the following energy functional with problem \eqref{mainproblem}:
	\begin{eqnarray}\label{Ident energ orig prob}
	E_{u}(t):= \frac12\int_\Omega |\partial_t u(x,t)|^2 + |\nabla u(x,t)|^2\,dx + \int_\Omega F(u(x,t))\,dx,
	\end{eqnarray} where $F$ is the antiderivative of $f$ vanishing at zero.
	
	In order to obtain the stabilization of the nonlinear energy $E_u(t)$ associated with problem (\ref{mainproblem}), two ingredients are essential: (A):~the energy identity:
	\begin{eqnarray}\label{ei}\qquad
	E_{u}(t_2) +  \int_{t_1}^{t_2}\int_\Omega a(x) g(\partial_t u(x,t)) \partial_t u(x,t)\,dxdt = E_{u}(t_1)
	\end{eqnarray}
	for all $0\leq t_1<t_2<+\infty$ and
	(B):~ the observability inequality: there exists some constant $C=C(\|\{u_0,u_1\}\|_{H_{0}^{1}(\Omega)\times L^2(\Omega)})>0$ so that
	\begin{eqnarray}\label{oi}
	E_{u}(0) \leq C \int_{0}^{T}\int_\Omega  a(x) \left(|\partial_t u|^2 + |g(\partial_t u)|^2\right) \,dxdt
	\end{eqnarray}
	holds true for all $T \geq T_0$ (where $T_0$ is a certain time for which the so called geometric control condition holds). Moreover, for the given geometry we take the following granted.
	
	\begin{assum}\label{assumption1.3}
		Given any {$T \geq T_0$}, the unique solution $$u \in C(]0,T[; L^2(\Omega)) \cap C^1(]0,T[,H^{-1}(\Omega))$$ that satisfies
		\begin{equation}\label{UCP}
			\begin{cases}
				u_{tt} - \Delta u + V(x,t)u=0  &~\hbox{ in }~ \Omega \times (0,T), \\
				u=0 &~\hbox{ on }~ \omega,\\
			\end{cases}
		\end{equation}
		where  $V \in L^\infty( ]0,T[,L^{\infty}(\Omega))$, is the trivial one.
	\end{assum}
	
	The approach based on the ingredients (A) and (B) is a well known strategy in stabilization theory; see for example \cite{Astudillo2}, \cite{Dehman0}, \cite{Dehman2}, \cite{Joly}, \cite{Zuazua} and references therein. These papers work directly on the original problem (\ref{mainproblem}) and use a \emph{unique continuation argument} as a crucial ingredient for proving the observability inequality. However, the direct approach has a challenge that was explained in \cite{Dehman2} as follows:
	
	At first one argues by contradiction that the observability fails and thereby obtains a sequence of (suitably normalized) solutions  violating the desired inequality. Then, via passage to the limit, one gets a function $u$ in the energy space, which solves
	\begin{equation}\label{A}
	\begin{aligned}
	&\partial_t^2 u -\Delta u + f(u)=0 ~\hbox{in}~\Omega \times (0,T)
	\end{aligned}
	\end{equation}
	with $\partial_t u|_\omega \equiv 0$. The goal is to show that (\ref{A}) has a unique solution and it is zero because this would mean that the only solution of (\ref{mainproblem}) without damping is the zero solution, which is absurd. To this end, one can differentiate in time, set $w = \partial_t u$ and consider the problem
	\begin{equation}\label{B}
	\begin{aligned}
	&\partial_t^2 w -\Delta w + f'(u)w=0 ~\hbox{in}~\Omega \times (0,T)
	\end{aligned}
	\end{equation}
	with $w|_{\omega}\equiv 0$. The next step is then to apply a unique continuation argument to the above problem that has a lower order potential to deduce that $w \equiv 0$. This would imply that $u$ is time independent, and therefore $u$ becomes the solution of a stationary semilinear elliptic problem whose solution is certainly zero thanks to right sign of the nonlinearity. Some unique continuation results assuming rather strong integrability conditions for the potential are available in the literature \cite{DZZ}, \cite{Ruiz}, \cite{Tataru}, \cite{Tataru-UCP}, \cite{Xu Zhang}. Part of these results are valid only locally while others work globally. However, none of them directly applies to the given class of nonlinearities because of the less restrictive assumptions on the power index.
	
	The difficulty explained above with the direct approach was treated in \cite{Dehman2} by unraveling hidden regularity properties associated with the potential $f'(u)$ in the case the domain is $\Omega=\mathbb{R}^3$. To handle the same difficulty on a general domain $\Omega\subset \mathbb{R}^n$, we first introduce a sequence of truncated problems rather than directly working on the original problem \eqref{mainproblem}:
	\begin{equation}\label{TP}
	\begin{cases}
	\partial_t^2 u_{k} -\Delta u_{k} + f_k(u_k) + a(x)  g(\partial_t u_k) = 0 & ~ \hbox{ in } ~ \Omega \times (0, +\infty),\\
	u_k=0 & ~ \hbox{ on } ~ \partial \Omega \times (0,+\infty ),\\
	u_k(x,0)=u_{0,k}(x), ~ \partial_t u_k(x,0)=u_{1,k}(x)  & ~ \hbox{ in } ~ \Omega,
	\end{cases}
	\end{equation}
	where $ f_k$ is defined in (\ref{trunc func}) and $\{u_{0,k},u_{1,k}\}$ is a sequence of initial regular data that approaches $\{u_0,u_1\}$ in the phase space. Such truncations were previously used within the context of stabilization of nonlinear wave equations \cite{Lasiecka-Tataru} or the nonlinear Schrödinger equation \cite{Ozs2011}, \cite{CavOzs2020}.   Once we show that for each $k\in \mathbb{N}$, $u_k$ possesses a strong regularity, then we can exploit additional regularity properties of the nonlinear term $f_k(u)$ by utilizing Strichartz estimates that are valid for solutions of the subcritical wave equation. Next, the ellipticity
	on the domain $\omega \times (0, T)$ and the property of propagation of singularities provide a desirable regularity for $u_k$ and integrability for $f_k'(u_k)$. Only then, one can utilize the unique continuation results cited above. Since the sequence of problems (\ref{TP}) converges to the original problem (\ref{mainproblem}), it is sufficient to obtain the energy identity and the observability inequality for each truncated problem because the same properties as well as the decay remain valid for the original problem via a delicate passage to the limit.  An advantage of the present approach is that we do not need to consider the well known {\em } due to G\'erard \cite{Gerard-JFA},  which, roughly speaking, guarantees that the semilinear (subcritical) solution is asymptotically close to the solution of the undamped linear wave equation. This work uses the above approach to deal with a larger class of dissipative effects, thus substantially generalizing the recent results of \cite{Astudillo2}.
	
	The observability of the truncated problem is the novel ingredient here for which we trigger some microlocal analysis tools such as microlocal defect measures. Indeed, H\"ormander \cite{Hormander2} and Duistermaat and H\"ormander \cite{Hormander3} described the propagation of singularities to a partial differential equation $Pu=0$ in $\mathcal{D}'$ in a domain without boundary. In particular, they defined the wave front set $WF(u)$ of a distribution $u\in \mathcal{D}'$, the subset of the cotangent bundle in which singularities may move, and showed that the wave front set is invariant under the Hamiltonian flow $H_p$ of the principal symbol $p$ of $P$. Later, H\"ormander  \cite{Hormander} incorporated his own work and that of others into his magnum opus \emph{The Analysis of Linear Partial Differential Operators}. Chapter 24 in Volume III of that work remains a standard reference on the subject. It is well-known that if $P$ is a proper, classical pseudo-differential operator and $u \in \mathcal{D}'$ is such that $Pu \in C^\infty$, then $WF(u)$ is a union of integral curves of the Hamiltonian of $p$. Inspired by the works mentioned above, \cite{Burq-Gerard-CR}, \cite{Burq-Gerard}, and \cite{Gerad} proved an analogous result related to a microlocal defect measure $\mu$ (a Radon measure) associated with a bounded sequence in $ H^1_{loc} $ which weakly converges to zero. Roughly speaking, the $\supp (\mu)$ is also invariant under the Hamiltonian flow $H_p$ of the principal symbol $p$ of a differential operator $P$ which satisfies certain properties. The main result reads as follows: {\em Let $P$ be a self-adjoint differential operator of order $m$ on an open set $\Omega$ and $(u_n)_{n\in \mathbb{N}}$ be a bounded sequence in $L_{loc}^2(\Omega)$ weakly converging to $0$ with a microlocal defect measure $\mu$. Suppose that $P u_n$ converges to $0$ in $H_{loc}^{-(m-1)}$. Then the support of $\mu$, $\supp(\mu)$, is a union of curves like $s\in I \mapsto \left(x(s), \frac{\xi(s)}{|\xi(s)|}\right)$, where $s\in I \mapsto (x(s),\xi(s))$ is a null-bicharacteristic of $p,$ where $p$ is the principal symbol of $P$ (see Theorem \ref{Theorem 4.60} below)}. This result is crucial to establish the control and stabilization of waves, which we will discuss later. The idea for proving the above result is found in the unpublished lecture notes of Burq and G\'erard \cite{Burq-Gerard}. To make this article self-contained, we supplement details to the proof of this remarkable result and give other related preliminaries in the appendix.
	
	\begin{Remark}\label{Remark1}
		The results of this paper extend to the case of nonconstant coefficients.
	\end{Remark}
	
	\begin{Remark}
		For $V \equiv 0$,  Assumption \ref{assumption1.3} is satisfied by the results in the lecture notes of Burq and G\'erard (see equations $(6.28)$ and $(6.29)$ on page 75 of \cite{Burq-Gerard}).  One can also give an example where the unique continuation principle holds globally and also applies to the Laplacian with coefficients, see for instance Section 2.1.2 and Theorem 2.2 in \cite{DZZ}.
	\end{Remark}
	
	Analyzing the behaviour of solutions of the wave equation subject to a localized frictional damping has been a major topic of interest for several decades. \cite{martinez} made improvements from the point of the geometrical conditions on the localization of the damping and eliminated the polynomial growth assumption of the damping at zero.  \cite{caval2018}, \cite{Cavalcanti2} and \cite{Cavalcanti3} studied the wave equation with localized nonlinear damping, the first one dealing with the case of compact surfaces while the last two focus on the case of compact manifolds. \cite{alabau10} obtained sharp or quasi-optimal upper energy decay
	rates for a range of nonlinear feedbacks characterized by their behavior at the origin.  Some other notable references on the subject are \cite{Bortot2013}, \cite{dao}, \cite{Lebeau}, and \cite{RT}.  The case of the \emph{semilinear} wave equation as in the present article also attracted the attention of many researchers.  There are some results in the literature where damping is either linear (see e.g., \cite{Dehman0}, \cite{Dehman2}, \cite{Joly}, \cite{laurent}, \cite{Nakao}, and \cite{Zuazua}) or linearly bounded \cite{Astudillo2}. The case when the dissipative mechanism is not necessarily linear or linearly bounded (with $n=3$) has been extensively investigated, considering a larger class of dissipative effects, as we can see in \cite{Cavalcanti0}, \cite{TAMS}, \cite{Moez}, \cite{las-tou:06}, \cite{Toundykov2} and \cite{Toundykov} where the source terms were considered with subcubic growth. Moreover, following the ideas introduced by \cite{Cavalcanti0}, \cite{AMO}, \cite{TAMS}, \cite{Moez}, \cite{las-tou:06}, \cite{Toundykov2} and \cite{Toundykov} we give effective decay rate estimates depending on the growth rate of the feedback at infinity, extending some previous results (see Corollary \ref{uniformdecay} and \ref{uniformdecay2}) to subquintic sources. Our work in particular extends those results in the literature (see e.g., \cite{Joly}) who put rather strong analyticity assumptions on the nonlinear source.
	
	Our paper is organized as follows. Section \ref{section4} studies the wellposedness of the original problem and shows that the solutions of the truncated problem converge to the solution of the original problem. In Section \ref{section5}, we prove the observability inequality for the truncated problem and then deduce an analogous inequality for the original problem. Finally, in Sections \ref{section6} and \ref{section7} we apply Lasicka-Tataru's method to derive explicit decay rates. In Appendix \ref{section2} we review some microlocal analysis tools and supplement details to proofs of some major results on microlocal defect measures.
	
	\section{Wellposedness}\label{section4}
	\setcounter{equation}{0}

	\subsection{Existence of solutions in the energy space}

	This section is devoted to the proof of the wellposedness  of solutions to problem
	\begin{equation}\label{eq:*}
	\begin{cases}
	\partial_t^2 u -\Delta u + f(u) + a(x) g(\partial_t u) = 0 & ~ \hbox{ in } ~ \Omega \times (0, +\infty),\\
	u=0 & ~ \hbox{ on } ~ \partial \Omega \times (0,+\infty ),\\
	u(x,0)=u_0(x) , ~ u_t(x,0)=u_1(x)  & ~ \hbox{ in } ~ \Omega.
	\end{cases}
	\end{equation}
	
	We shall first consider the following sequence of approximate problems with truncated nonlinearities:
	\begin{equation}\label{eq:AP}
	\begin{cases}
	\partial_t^2 u_{k} -\Delta u_{k} + f_k(u_k) + a(x)  g(\partial_t u_k) = 0 & ~ \hbox{ in } ~ \Omega \times (0, +\infty),\\
	u_k=0 & ~ \hbox{ on } ~ \partial \Omega \times (0,+\infty ),\\
	u_k(x,0)=u_{0,k}(x), ~ \partial_t u_k(x,0)=u_{1,k}(x)  & ~ \hbox{ in } ~ \Omega,
	\end{cases}
	\end{equation}
	where, $\{u_{0,k},u_{1,k}\}$ is a sequence that will be defined later, and for each $k\in \mathbb{N}$, $f_k: \mathbb{R} \longrightarrow \mathbb{R}$ is defined by
	\begin{equation}\label{trunc func}
	f_{k}(s):=\begin{cases}
	f(s), & |s|\leq k,\\
	f(k),& s> k,\\
	f(-k),& s < -k.
	\end{cases}
	\end{equation}

	We shall start this section with some auxiliary results whose proofs are easy.
	\begin{Lemma}\label{Lemma1}
		The distributional derivative $f_k'$ of the function defined in (\ref{trunc func}) is the essentially bounded function $g_k:\mathbb{R} \rightarrow \mathbb{R}$ given by
		\begin{equation}\label{drrivative}
		g_k(s):=
		\begin{cases}
		f'(s), & |s|\leq k,\\
		0, & s> k,\\
		0, & s < -k.
		\end{cases}
		\end{equation}
	\end{Lemma}
	
	In the proof of the next result we use the following:
	\begin{Theorem}[\cite{Brezis}]\label{brezis}
		Let $u \in W^{1,p}(I)$ with $1 \leq p \leq \infty$, where $I$ is a bounded interval of $\mathbb{R}$. Then, there exists $\widetilde{u} \in C(\bar{I})$ such that
		\begin{equation*}
		u=\widetilde{u} \hbox{ a.e. in } I
		\end{equation*}
		and
		\begin{equation*}
		\widetilde{u}(x)-\widetilde{u}(y)=\int_{y}^{x} u'(t)dt \hbox{ for all } x, y \in \bar{I}.
		\end{equation*}
	\end{Theorem}

	\begin{Lemma}\label{Lema2}
		For each $k\in \mathbb{N}$, there exists a positive constant $C_k$ for which
		\begin{equation*}
		|f_k(r)-f_k(s)|\leq C_k|r-s|, \forall r,s \in  \mathbb{R},
		\end{equation*}
		where $f_k$ is the function defined in \eqref{trunc func}.
	\end{Lemma}

	\medskip

	Let $A: \mathscr{D}(A) \subset H \longrightarrow  H$ be a maximal monotone operator on a Hilbert space $H$ (for more details see \cite{Barbu}). We denote the inner product of $H$ by $(\cdot, \cdot)$ and the corresponding norm by $\| \cdot \|$. Consider the following initial value problem.
	\begin{equation}\label{acp}
	u_t+Au+Bu \ni f \hbox{ and } u(0)=u_0 \in H.
	\end{equation}
	where $B:H \longrightarrow H$ is a locally Lipschitz operator, that is,
	\begin{equation*}
	\|Bu-Bv\|\leq L(K) \|u-v\|
	\end{equation*}
	provided $\|u\|$, $\|v\| \leq K$.
	
	Then, one has the following result:
	\begin{Theorem}[\cite{Lasiecka}]\label{cpde}
		Suppose that $A$ is a maximal monotone mapping and that $A0 \ni 0$. Assuming $u_0 \in \mathscr{D}(A)$, $f \in W^{1,1}(0,t;H)$ for all $t>0$ and a locally Lipschitz mapping $B:H \longrightarrow H$, there exists $t_{\max}\leq \infty$ such that problem \eqref{acp} has a strong solution $u$ on the interval $[0,t_{\max})$, i.e., $u \in W^{1,\infty}(0,t_{\max};H)$ and $u(t) \in \mathscr{D}(A)$ for all $t>0$. Furthermore, if we assume $u_0 \in \overline{\mathscr{D}(A)}$ we obtain a unique generalized solution $u \in C([0,t_{\max};H)$ to problem \eqref{acp}. In both cases one has $\displaystyle\lim_{t \rightarrow t_{\max}}\|u(t)\|=\infty$ provided $t_{\max}<\infty$.
	\end{Theorem}

	\subsection{Abstract framework and energy functionals}
	
	We consider the weak phase space
	\begin{equation*}
	\mathcal{H} = H_{0}^{1}(\Omega) \times  L^{2}(\Omega),
	\end{equation*}
	which is endowed with the inner product
	\begin{equation*}
	\langle (u_{1}, u_{2}), (v_{1}, v_{2}) \rangle_{\mathcal{H}} = \int_{\Omega} \nabla u_{1} \cdot \nabla v_{1}+u_{2}v_{2}\,dx.
	\end{equation*}
	To study Hadamard's wellposedness for problem \eqref{eq:AP}, we write it in an abstract operator theoretic form. Setting $W_k(t)=(u_k,\partial_t u_k)$, the system \eqref{eq:AP} can be reformulated as the following Cauchy problem in $\mathcal{H}$:
	\begin{equation}\label{3.1}
	\begin{cases}
	\displaystyle \frac{\partial W_k}{\partial t}(t)+\mathbb{A}W_k(t) + \mathbb{B}_k(W_k(t)) = {}0  \\
	W(0)={}  (u_{0,k},  u_{1,k}),
	\end{cases}
	\end{equation}
	where the unbounded operator $\mathbb{A}: \mathscr{D}(\mathbb{A}) \subset \mathcal{H} \rightarrow \mathcal{H}$ is given by
	\begin{equation}\label{3.2}
	\mathbb{A}=\left(
	\begin{array}{cc}
	0 & -I \\
	-\Delta & a(x)g(\cdot)
	\end{array}
	\right),\end{equation}
	that is,
	\begin{equation}\label{3.3}
	\mathbb{A}(u,v) = \left(-v,-\Delta u +a(x)g(v)\right),
	\end{equation}
	with domain
	\begin{equation}\label{3.4}
	\mathscr{D}(\mathbb{A})  = \left\{(u,v) \in \mathcal{H}:
	\begin{array}{c}
	v \in H_{0}^{1}(\Omega),\\
	-\Delta u+a(x)g(v) \in L^2(\Omega),\\
	a(x)g(v) \in L^1(\Omega) \cap H^{-1}(\Omega)
	\end{array}
	\right\},
	\end{equation}
	and $\mathbb{B}_k: \mathcal{H} \rightarrow \mathcal{H}$ is the nonlinear operator
	\begin{equation}\label{3.5}
	\mathbb{B}_k(u,v) = \left(0, f_k(u) \right).
	\end{equation}
	Employing standard arguments of the nonlinear semigroup
	theory and Theorem \ref{cpde} we have the following result:
	
	\begin{Theorem}[Wellposedness]\label{thm:well-posedness}	If $g:\mathbb{R}\to \mathbb{R}$ is continuous monotone increasing and zero at the origin and $a(x)\in L^\infty(\Omega)$, then
		\begin{enumerate}
			\item [i)] If $\mathcal{O}(g)=1$, problem \eqref{eq:AP} generates a nonlinear semigroup on the Hilbert space $	\mathcal{H}$ so that for any initial data $( u_{0,k}, u_{1,k}) \in \mathcal{H}$, there exists a unique weak solution
			$u_k \in C(\mathbb{R}_{+};H_{0}^{1}(\Omega)) \cap C^{1}(\mathbb{R}_{+}; L^2(\Omega)).  \label{3.9}$
			\item [ii)] If $\mathcal{O}(g)=1$, $( u_{0,k}, u_{1,k}) \in \mathscr{D}(\mathbb{A})$ then
			$(u_k, \partial_tu_k) \in W^{1,\infty}(\mathbb{R}_+;\mathcal{H}).$
			\item [iii)] If $\mathcal{O}(g)\neq 1$, $1 \leq p <\frac{n}{n-2}$ and $( u_{0}, u_{1}) \in \mathscr{D}(\mathbb{A})$ then the problem \eqref{eq:*} admits a unique strong solution $(u,u_t) \in W^{1,\infty}(\mathbb{R}_+;\mathcal{H})$. Moreover, the energy identity \eqref{ei} is satisfied.
		\end{enumerate}
	\end{Theorem}
	
	\begin{Remark}\label{regularitystrong1}
		Assume that $\mathcal{O}(g)=1$. Since $\mathbb{B}_k$ is globally Lipschitz and $\mathbb{B}_k(0)=0$ we have
		\begin{equation}\label{estimateB}
		\|\mathbb{B}_k(W_k(t))\|_{\mathcal{H}} \lesssim \|W_k(t)\|_{\mathcal{H}}\leq \|W_k\|_{L^\infty(\mathbb{R}_+;\mathcal{H})}< \infty.
		\end{equation}
		From \eqref{estimateB} it follows that $\mathbb{B}_k(W_k) \in L^\infty(\mathbb{R}_+;\mathcal{H})$. Employing item ii) of Theorem \ref{thm:well-posedness} we deduce that $\partial_t W_k \in L^\infty(\mathbb{R}_+;\mathcal{H})$. Therefore, from $\eqref{3.1}_1$ we obtain $\mathbb{A}W_k \in L^\infty(\mathbb{R}_+;\mathcal{H})$. Moreover,
		\begin{eqnarray}\label{uniformg1}
		\|a(\cdot)g(\partial_t u_k(\cdot, t))\|_{L^2(\Omega)}^{2}&=&\int_{\Omega} |a(x)g(\partial_t u_k(x,t))|^2 \, dx \nonumber \\
		&\leq& \|a\|_{L^\infty(\Omega)}^2\int_{\{x:|\partial_t u_k(x,t)| \leq 1\}} |g(\partial_t u_k(x,t))|^2 \, dx \nonumber  \\
		&&+\|a\|_{L^\infty(\Omega)}^2\int_{\{ x:|\partial_t u_k(x,t)| > 1\}}|g(\partial_t u_k(x,t))|^2 \, dx\nonumber  \\
		&\leq& \|a\|_{L^\infty(\Omega)}^2M_g\operatorname{meas}(\Omega)+\|a\|_{L^\infty(\Omega)}^2\int_{\Omega} |\partial_{t} u_k(x,t)|^2 \, dx \nonumber  \\
		&\lesssim & \|a\|_{L^\infty(\Omega)}^2M_g\operatorname{meas}(\Omega)+\|a\|_{L^\infty(\Omega)}^2 \|\partial_{t} u_k(\cdot,t)\|_{L^2(\Omega)}^2 \nonumber \\
		&\lesssim& \|a\|_{L^\infty(\Omega)}^2M_g\operatorname{meas}(\Omega)+\|a\|_{L^\infty(\Omega)}^2 \|\partial_{t} u_k\|_{L^\infty(\mathbb{R}_+;L^2(\Omega))}^2,
		\end{eqnarray}
		where $M_g=\max_{s \in [-1,1]}|g(s)|$. The estimates in \eqref{uniformg1} imply that
		\begin{equation}\label{uniformg}
		a(\cdot)g(\partial _tu_k) \in L^\infty(\mathbb{R}_+;L^2(\Omega))
		\end{equation}
		and from \eqref{3.3} we obtain
		\begin{equation}\label{regularitystrong2}
		u_k \in L^\infty(\mathbb{R}_+;H_{0}^{1}(\Omega) \cap H^2(\Omega)) \cap W^{1,\infty}(\mathbb{R}_+;H_{0}^{1}(\Omega)) \cap W^{2,\infty}(\mathbb{R}_+;L^2(\Omega)).
		\end{equation}
		
	\end{Remark}
	
	Take $\{u_0,u_1\} \in H_0^1(\Omega)\times L^2(\Omega)$. Since $\mathscr{D}(\mathbb{A})$ is dense in $\mathcal{H}$, there exists a sequence $\{u_{0,k}, u_{1,k}\}\in  \mathscr{D}(\mathbb{A})$ such that
	\begin{eqnarray}\label{conv init data}
	\{u_{0,k}, u_{1,k}\} \rightarrow \{u_0,u_1\} \hbox{ strongly in } H_0^1(\Omega) \times L^2(\Omega).
	\end{eqnarray}
	
	From Theorem \ref{thm:well-posedness} and Remark \ref{regularitystrong1}, for each $k\in \mathbb{N}$, the problem \eqref{eq:AP} possesses a unique regular solution $u_k$ in the class
	\begin{eqnarray*}
		u_k \in L^\infty(\mathbb{R}_+;H_{0}^{1}(\Omega) \cap H^2(\Omega)) \cap W^{1,\infty}(\mathbb{R}_+;H_{0}^{1}(\Omega)) \cap W^{2,\infty}(\mathbb{R}_+;L^2(\Omega)).
	\end{eqnarray*}

	Multiplying the main equation in (\ref{eq:AP}) by $\partial_t u_k$ and performing integration by parts, it follows that
	\begin{eqnarray}\label{est1}
	&&\frac12 \frac{d}{dt} \|\partial_t u_k(t)\|_{L^2(\Omega)}^2 + \frac12 \frac{d}{dt} \|\nabla u_k(t)\|_{L^2(\Omega)}^2 + \frac{d}{dt}\int_\Omega F_k(u_k(x,t))\,dx\\
	&&+ \int_\Omega a(x)g(\partial_t u_k(x,t)) \partial_t u_k(x,t)\,dx=0, \hbox{ for all } t\in [0,\infty),\nonumber
	\end{eqnarray}
	where $F_k(\lambda) = \int_0^\lambda f_k(s)\,ds.$
	
	Thus, taking (\ref{est1}) into account, we obtain the energy identity for the truncated problem \eqref{eq:AP}, that is,
	\begin{eqnarray}\label{est2}\qquad
	E_{u_k}(t_2) +  \int_{t_1}^{t_2}\int_\Omega a(x) g(\partial_t u_k(x,t)) \partial_t u_k(x,t)\,dxdt = E_{u_k}(t_1)
	\end{eqnarray}
	for all $0\leq t_1<t_2<+\infty$, where
	\begin{eqnarray}
	E_{u_k}(t):= \frac12 \int_\Omega  |\partial_t u_k(x,t)|^2 + |\nabla u_k(x,t)|^2\,dx + \int_\Omega F_k(u_k(x,t))\,dx,
	\end{eqnarray}
	is the energy associated with problem (\ref{eq:AP}).
	We observe that from (\ref{trunc func}) one has
	\begin{equation}\label{primitive trunc func}
	F_{k}(s):=\begin{cases}
	\displaystyle \int_0^sf(\xi)\,d\xi, & |s|\leq k,\\
	\displaystyle \int_0^k f(\xi) \,d\xi + f(k)[s-k], & s > k,\\
	\displaystyle f(-k)[s+k] + \int_0^{-k}f(\xi)\,d\xi, & s < -k.
	\end{cases}
	\end{equation}

	Since $f$ satisfies the sign condition, it follows that $F_k(s) \geq 0$ for all $s\in \mathbb{R}$ and for all $k\in \mathbb{N}$. In addition, from \eqref{ass on f} and \eqref{ass on F}, we obtain $|f(s)|\leq c [|s| + |s|^p]$ and $0 \leq F(s) \leq f(s) \,s$ for all $s\in \mathbb{R}$, respectively. Then, we infer that
	\begin{eqnarray}\label{bound on Fk}
	|F_k(s)| \leq c [|s|^2 + |s|^{p+1}] \hbox{ for all }s\in \mathbb{R} \hbox{ and for all }~k\in\mathbb{N}.
	\end{eqnarray}
	Consequently,
	\begin{eqnarray}\label{bound data L1}
	\int_\Omega |F_k(u_{0,k})|\,dx &\leq& c\int_\Omega \left[|u_{0,k}|^2 + |u_{0,k}|^{p+1} \right]\,dx\\
	&\lesssim&  \|u_{0,k}\|_{H_0^1(\Omega)}.\nonumber
	\end{eqnarray}
	
	Assuming that $p$ satisfies \eqref{ass on f'}, we have for every dimension $n\geq 1$ that $H_0^1(\Omega) \hookrightarrow L^{p+1}(\Omega)$, which implies that the RHS of  (\ref{bound data L1}) is bounded.
	So, identity (\ref{est2}), convergence property \eqref{conv init data} and estimate \eqref{bound data L1} yield a subsequence of $\{u_k\}$, still denoted $\{u_k\}$ such that
	\begin{eqnarray}
	&&u_k \rightharpoonup u \hbox{ weak-star in } L^\infty(0,\infty; H_0^1(\Omega)),\label{conver1}\\
	&&\partial_t u_k \rightharpoonup \partial_t u \hbox{ weak-star in } L^\infty(0,\infty; L^2(\Omega))\label{conver2},\\
	&& a(x)\,g(\partial_t u_{k}) \rightharpoonup \chi \hbox{ weakly in }L^{2}(0,T; L^{2}(\Omega)), \hbox{ for all } T>0. \label{eq:3.56}
	\end{eqnarray}
	Now, from standard compactness arguments (see Simon \cite{Simon}) we deduce
	\begin{eqnarray}\label{conver3}
	u_k \rightarrow u \hbox{ strongly in } L^\infty (0,T; L^{2^{\ast}-\varepsilon}(\Omega)); \hbox{ for all } T>0,
	\end{eqnarray}
	where $2^{\ast}:= \frac{2n}{n-2}$ and $\varepsilon>0$ is small enough. In addition, from (\ref{conver3}), we obtain
	\begin{eqnarray}\label{conver4}
	u_k \rightarrow u \hbox{ a.e. in  } \Omega \times (0,T), \hbox{ for all } T>0.
	\end{eqnarray}
	
	Then, from \eqref{ass on f}, \eqref{ass on f'},
	\begin{eqnarray}\label{chain}
	H_0^1(\Omega)\hookrightarrow L^{\frac{2n}{n-2}}(\Omega)\hookrightarrow L^{\frac{n+2}{n-2}}(\Omega) \hookrightarrow L^p(\Omega) \hookrightarrow L^1(\Omega),
	\end{eqnarray}
	(\ref{conv init data}), (\ref{est2}) and (\ref{bound data L1}), the following estimate holds:
	\begin{eqnarray}\label{estIII}\quad
	\|f_k(u_k)\|_{L^{\frac{p+1}{p}}(\Omega \times (0,T)) }^{\frac{p+1}{p}} &=& \int_0^{T} \int_\Omega |f_k(u_k(x,t))|^{\frac{p+1}{p}}\,dxdt \nonumber\\
	&\lesssim& \int_0^T\int_\Omega |u_k|^{\frac{p+1}{p}}\,dxdt + \int_{0}^{T}\int_\Omega |u_k|^{p+1}\,dx dt\nonumber\\
	&=& \int_0^T \|u_k\|_{L^{\frac{p+1}{p}}(\Omega) }^{\frac{p+1}{p}}\,dt+\int_0^T \|u_k\|_{L^{p+1}(\Omega)}^{p+1}\, dt\nonumber\\
	&\lesssim& \int_0^T \|u_k\|_{H_0^1(\Omega)}^{\frac{p+1}{p}}\,dt+\int_0^T \|u_k\|_{H_0^1(\Omega)}^{p+1}\, dt\nonumber\\
	&\lesssim& \|u_k\|_{L^\infty(0,T;H_0^1(\Omega))}^{\frac{p+1}{p}}+\|u_k\|_{L^\infty(0,T;H_0^1(\Omega))}^{p+1}\nonumber\\
	& \leq& c <+\infty, \hbox{ for all } t\geq 0,
	\end{eqnarray}
	provided that $1\leq p < \frac{n+2}{n-2}$ (we note that $p+1 < \frac{2n}{n-2}=2^{\ast}\Leftrightarrow 1\leq p < \frac{n+2}{n-2}$).
	
	It is easy to see that
	\begin{equation}\label{estIII.1}
	f(u) \in L^\infty(0,\infty;L^{\frac{p+1}{p}}(\Omega)).
	\end{equation}
	Indeed,
	\begin{eqnarray}\label{estIII.2}
	\int_\Omega |f(u(x,t))|^{\frac{p+1}{p}}\,dx &\lesssim& \int_\Omega |u(x,t)|^{\frac{p+1}{p}}\,dx + \int_\Omega |u(x,t)|^{p+1}\,dx \nonumber\\
	& \lesssim& \|u(\cdot,t)\|_{H_{0}^{1}(\Omega)}^{\frac{p+1}{p}}+ \|u(\cdot,t)\|_{H_{0}^{1}(\Omega)}^{p+1}\nonumber\\
	& \lesssim& \|u\|_{L^\infty(0,T;H_{0}^{1}(\Omega))}^{\frac{p+1}{p}} +\|u\|_{L^\infty(0,T;H_{0}^{1}(\Omega))}^{p+1}<+\infty, \hbox{ for all } t\geq 0.
	\end{eqnarray}
	From \eqref{estIII.2} and the definition of essential supremum we obtain \eqref{estIII.1}.
	
	Now, (\ref{conver4}) yields
	\begin{eqnarray}\label{conver5'}
	f_k(u_k) \rightarrow f(u) \hbox{ a.e. in  } \Omega \times (0,T), \hbox{ for all } T>0.
	\end{eqnarray}
	Indeed, the convergence \eqref{conver4} guarantees the existence of a set $Z_T \subset \Omega \times (0,T)$ with $\operatorname{meas}(Z_T)=0$  such that $u_k(x,t) \rightarrow u(x,t)$ for all $(x,t) \in \Omega \times (0,T) \setminus Z_T$ when $k \rightarrow \infty$. Therefore, for all $(x,t) \in \Omega \times (0,T) \setminus Z_T$ there exists a positive constant $L=L(x,t)>0$ such that $|u_k(x,t)|<L$ for all $k \in \mathbb{N}$.  Then, using the definition of $f_k$, we obtain that
	\begin{equation}\label{boundedness}
	\hbox{ if } |u_k(x,t)|<L, \hbox{ for all } k \in \mathbb{N}, \hbox{ then }~ f_k(u_k(x,t))=f(u_k(x,t)), \hbox{ for all } k\geq L,
	\end{equation}
	that is,
	\begin{equation}\label{boundednes1}
	f_k(u_k(x,t))-f(u_k(x,t)) \rightarrow 0 \hbox{ when } k \rightarrow \infty \hbox{ for all } (x,t) \in \Omega \times (0,T) \setminus Z_T.
	\end{equation}
	On the other hand, employing the continuity of $f$ it follows that
	\begin{equation}\label{boundednes2}
	f(u_k(x,t))-f(u(x,t)) \rightarrow 0 \hbox{ when } k \rightarrow \infty \hbox{ for all } (x,t) \in \Omega \times (0,T) \setminus Z_T.
	\end{equation}
	From \eqref{boundednes1} and \eqref{boundednes2} the convergence \eqref{conver5'} holds.
	
	\medskip
	
	\begin{Lemma}[Strauss]~Let $\mathcal{O}$ be an open and bounded subset of $\mathbb{R}^N$, $N\geq 1$, $1<q<+\infty$ and $\{u_n\}_{n\in \mathbb{N}}$ be a sequence which is bounded in $L^q(\mathcal{O})$. If $u_n \rightarrow u$ a.e. in $\mathcal{O}$, then $u\in L^q(\mathcal{O})$ and $u_n \rightharpoonup u$ weakly in $L^q(\mathcal{O})$. In addition, if $1\leq r <q$, we also have $u_n \rightarrow u$ strongly in $L^r(\mathcal{O})$.
	\end{Lemma}
	\begin{proof}
		See  Exercise 4.16 in \cite{Brezis} or \cite{Strauss}.
	\end{proof}

	Combining \eqref{estIII}, \eqref{estIII.2}, and \eqref{conver5'} and employing Lions' Lemma we deduce that
	\begin{eqnarray}\label{weak conv fk}
	f_k(u_k) \rightharpoonup f(u) \hbox{ weakly in } L^{\frac{p+1}{p}}(\Omega \times (0,T)).
	\end{eqnarray}
	From Strauss' Lemma, it follows that
	\begin{eqnarray}\label{conver6}
	f_k(u_k) \rightarrow f(u) \hbox{ strongly in }L^r(\Omega \times (0,T)), \hbox{ for all }1\leq r < \frac{p+1}{p}.
	\end{eqnarray}

	Take $\varphi \in C_0^\infty(\Omega)$ and $\theta \in C_0^\infty(0,T)$. Multiplying the first equation of problem (\ref{eq:AP}) and integrating by parts, we get
	\begin{eqnarray}\label{form var}
	&&-\int_0^T \theta'(t) \int_\Omega \partial_t u_{k}(x,t)\, \varphi(x) \,dx dt + \int_0^T \theta(t) \int_\Omega \nabla u_{k}(x,t)\cdot \nabla \varphi(x) \,dxdt\\
	&& +\int_0^T \theta(t) \int_\Omega f_k(u_k(x,t))\,\varphi(x) \,dxdt +\int_0^T \theta(t) \int_\Omega a(x) g(\partial_t u_k(x,t))\, \varphi(x)\,dxdt\nonumber = 0.
	\end{eqnarray}
	
	Passing to the limit in (\ref{form var}) taking (\ref{conver1}), (\ref{conver2}), (\ref{eq:3.56}) and (\ref{weak conv fk}) there exists
	$u\in L^\infty(0,T; H_0^1(\Omega))$, $\partial_t u\in L^\infty(0,T; L^2(\Omega)$ such that
	\begin{eqnarray}\label{limit}
	&&-\int_0^T \theta'(t) \int_\Omega \partial_t u(x,t)\, \varphi(x) \,dx dt + \int_0^T \theta(t) \int_\Omega \nabla u(x,t)\cdot \nabla \varphi(x) \,dxdt\\
	&& +\int_0^T \theta(t) \int_\Omega f(u(x,t))\,\varphi(x) \,dxdt +\int_0^T \theta(t) \int_\Omega \chi\, \varphi(x)\,dxdt\nonumber = 0
	\end{eqnarray}
	for all $\varphi \in C_0^\infty(\Omega)$ and $\theta \in C_0^\infty(0,T)$, from which we conclude that
	\begin{eqnarray}\label{dist sol}
	\partial_t^2 u - \Delta u + f(u) + \chi = 0 \hbox{ in }\mathcal{D}'(\Omega \times (0,T)).
	\end{eqnarray}
	Assuming $n=3$ and observing
	$$\chi \in L^2(0,T; L^2(\Omega)), ~ \Delta u \in L^\infty(0,T; H^{-1}(\Omega)) \hbox{ and } f(u) \in L^\infty(0,T; L^{\frac{p+1}{p}}(\Omega)),$$
	we deduce that
	\begin{eqnarray}\label{weak solution'}
	\partial_t^2 u = \Delta u - f(u) -\chi \in  L^{2}(0,T;  H^{-1}(\Omega)).
	\end{eqnarray}
	
	Applying Lemma 8.1 of Lions-Magenes \cite{Lions-Magenes} and Lemma 1.2 from \cite{Lions1}, we deduce that
	\begin{equation}\label{weak solution}
	u\in C_w(0,T;H_0^1(\Omega)) \hbox{ and } \partial_t u \in C_w(0,T; L^2(\Omega)),
	\end{equation}
	where $C_w(0,T;Y)$ is the space of $f\in L^\infty(0,T;Y)$ for which the mapping $[0,T] \mapsto Y$ is weakly continuous. That is, $t \mapsto \langle y', f(t) \rangle_{Y',Y}$ is continuous in $[0,T]$ for all $y' \in Y'$, dual of $Y$.
	
	\medskip
	
	\subsection{Recovering the regularity in time for the range \texorpdfstring{$1\leq p < \frac{n+2}{n-2}$}.}
	
	Throughout this subsection, assume $\mathcal{O}(g)=1$. The main goal of this subsection is to prove that if $1\leq p < \frac{n+2}{n-2}$ and
	\begin{eqnarray}\label{chi}
	\chi= a(x) g(\partial_t u),
	\end{eqnarray}
	then solutions satisfy
	{\small
		\begin{eqnarray}\label{regularity}\quad
		u\in C^0([0,T];H_0^1(\Omega)),~\partial_t u \in C^0([0,T]; L^2(\Omega)),~\partial_t^2 u \in L^2([0,T]; H^{-1}(\Omega)).
		\end{eqnarray}}
	In addition, one has
	\begin{eqnarray}\label{unif conv}
	\{u_k, \partial_t u_k\} \rightarrow \{u,\partial_tu\} \hbox{ in } C^0([0,T];H_0^1(\Omega)) \times C^0([0,T]; L^2(\Omega)).
	\end{eqnarray}
	
	
	In order to deal with the case $\frac{n}{n-2} \leq p < \frac{n+2}{n-2}$, we first recall some basic results.
	\begin{Theorem}[Strichartz estimates, \cite{Sogge}]\label{St-Est-3d} Let $(\Omega, \textbf{g})$ be a compact three-dimensional Riemannian manifold with boundary, $T > 0$ and $(q, r)$
		satisfying
		\begin{eqnarray}\label{Strichartz}
		\frac{1}{q} + \frac{3}{r} = \frac{1}{2}, \quad q \in \left[\frac{7}{2}, \infty \right].
		\end{eqnarray}
		There exists $C = C(T, q) > 0$ such that for every $G \in L^{1}(0, T; L^{2}(\Omega))$ and every
		$(u_{0}, u_{1}) \in H_{0}^{1}(\Omega) \times L^{2}(\Omega)$, the solution $u$ of
		\begin{eqnarray*}
			&& u_{tt} - \Delta_{\textbf{g}}u = G(t) \\
			&& (u, \partial_{t}u)(0) = (u_{0}, u_{1})
		\end{eqnarray*}
		satisfies the estimate
		\begin{eqnarray}
		\|u\|_{L^{q}(0, T; L^{r}(\Omega))} \leq C\left(\|u_{0}\|_{H^{1}_{0}(\Omega)} + \|u_{1}\|_{L^{2}(\Omega)} + \|G\|_{L^{1}(0, T; L^{2}(\Omega))} \right).
		\end{eqnarray}
	\end{Theorem}
	

		\begin{Remark}\label{rem-St-1}
			It follows from Theorem \ref{St-Est-3d} and a fixed point argument that, given $T, E_0, M_1>0$, $F \in L^{1}(0, T; L^{2}(\Omega))$ fixed, and $(u_{0}, u_{1}) \in H_{0}^{1}(\Omega) \times L^{2}(\Omega)$ such that $$\|u_0\|_{H_0^1(\Omega)}+\|u_1\|_{L^2(\Omega)}+\|F\|_{L^1(0,T;L^2(\Omega))}<E_0+M_1,$$ there exists $C(T,E_0, M_1,\Omega)$, such that
			every solution u of the system
			\begin{align}\label{Linear-inhomog}
				\begin{cases}
					u_{tt} - \Delta_{\mathbf{G}}u + f(u) = F & \hbox{ in } \Omega \times (0,\infty),\\
					u = 0 & \hbox{ on }~\partial \Omega \times (0, \infty),\\
					u(x,0)=u_{0}(x),~~ u_{t}(x,0)=u_{1}(x)& x\in \Omega
				\end{cases}
			\end{align}
			satisfies
			\begin{eqnarray}
				\|u\|_{L^{q}(0, T; L^{r}(\Omega))} \leq C \left( \|u_{0}\|_{H^{1}_{0}} + \|u_{1}\|_{L^{2}} + \|F\|_{L^{1}(0, T; L^{2}(\Omega))} \right)
			\end{eqnarray}
			for all $(q, r)$ such that $q \geq \frac{7}{2}$ and $\frac{1}{q} + \frac{3}{r} = \frac{1}{2}$, provided that
			\begin{equation}\label{fstrichartz}
				|f(r) - f(s)| \leq c \left(1 + |s|^{p-1} + |r|^{p-1} \right)|r-s|, \forall s,r\in \mathbb{R},
			\end{equation}
			and $1 \leq p <5$.
	\end{Remark}
	
	
	First we observe that \eqref{estimateB} and \eqref{est2} yield
	\begin{equation}\label{uniformfk}
		{f_k(u_k) \in L^\infty(0,T;L^2(\Omega)) \hbox{ for all } T>0}
	\end{equation}
	Moreover, \eqref{est2} ensures that
	\begin{eqnarray}\label{trunc-limfraco-g}
		\|a(\cdot)g(\partial_t u_k)\|_{L^2(0,T;L^2(\Omega))}^2&=& \int_0^T \| a(\cdot) g(\partial_t u_k(t)) \|_{L^2(\Omega)}^2\,dt\\
		&\leq& \|a\|_{L^\infty(\Omega)} \int_0^T \int_\Omega a(x) |g(\partial_t u_k(x,t))|^2 \,dxdt\nonumber\\
		&\leq&    \|a\|_{L^\infty(\Omega)} E_{u_k}(0).\nonumber
	\end{eqnarray}
	{Thus, combining Remark \ref{rem-St-1}, \eqref{uniformfk} and \eqref{trunc-limfraco-g}, we can make use of Strichartz estimates, as
		in the Remark \ref{rem-St-1}, to get the bound}
	{
		\begin{eqnarray}\label{Stri-uk}\quad
		\|u_k\|_{L^5(0,T;L^{10}(\Omega))}&\lesssim& \left[\|u_{0,k}\|_{H_0^1(\Omega)} + \|u_{1,k}\|_{L^2(\Omega)} +  \|a(\cdot) g(\partial_t u_k) \|_{L^1L^2} \right]\nonumber \\
		&\lesssim& (E_{u_k}(0)+ T^{1/2}\|a\|_{L^\infty}^{1/2}\, [E_{u_k}(0)]^{1/2}) \leq C=C(\|a\|_{L^\infty},T)<+\infty
		\end{eqnarray}
		for all $k\in \mathbb{N}$, where $L^1L^2$ denotes the space $L^1(0,T;L^2(\Omega))$. Then, we conclude that $u_k \rightharpoonup u$ weakly in $L^5(0,T;L^{10}(\Omega))$ where $u$ is the solution associated with problem (\ref{eq:*}). Analogously, we also have
		\begin{eqnarray}\label{Stri-uk'}
		\|u_k\|_{L^4(0,T;L^{12}(\Omega))} \leq C=C(\|a\|_{L^\infty},T)<+\infty.
		\end{eqnarray}
		
		We observe that once we assume $1\leq p \leq \frac{n+2}{n-2}\leq 5$ for $n>2$, then $2p \leq 10$ and consequently, from (\ref{Stri-uk}), one has
		\begin{eqnarray}\label{bound f(u_k)}
		&&\|f(u_k)\|_{L^1(0,T;L^2(\Omega))}=\int_0^T \left[\int_\Omega |f(u_k(x,t))|^2\,dx\right]^{\frac{1}{2}}dt\\
		&&\lesssim  \int_0^T \left[\int_\Omega \left( |u_k(x,t)|^2 + |u_k(x,t)|^{2p}\right)\,dx\right]^{\frac{1}{2}}dt\nonumber\\
		&&\lesssim   T^{4/5}|\Omega|^\frac{2}{5}\,\|u_k\|_{L^5(0,T; L^{10}(\Omega))}+\int_{0}^T\left[\int_\Omega |u_k(x,t)|^{2p}dx\right]^\frac{1}{2}dt\nonumber\\
		&&\lesssim  T^{4/5}|\Omega|^\frac{2}{5}\,\|u_k\|_{L^5(0,T; L^{10}(\Omega))}+\int_{0}^T\left[\int_\Omega |u_k(x,t)|^{10}dx\right]^\frac{p}{10}\left[\int_\Omega dx\right]^\frac{5-p}{10}dt\nonumber\\
		&&\lesssim   T^{4/5}|\Omega|^\frac{2}{5}\,\|u_k\|_{L^5(0,T; L^{10}(\Omega))}+T^{(5-p)/5}|\Omega|^{\frac{5-p}{10}}\,\|u_k\|_{L^5(0,T; L^{10}(\Omega))}^p\leq C_T\nonumber
		\end{eqnarray}
		for all $k\in \mathbb{N}$.
		Now, observing that $|f_k(s)|\leq c [|s| + |s|^p]$ for all $s\in \mathbb{R}$ and $k\in \mathbb{N}$, analogously, we infer
		\begin{eqnarray}\label{bound f_k(u_k)}
		\|f_k(u_k)\|_{L^1(0,T;L^2(\Omega))} \leq C_T \hbox{ for all } k\in \mathbb{N}.
		\end{eqnarray}
		
		\medskip
		
		To prove (\ref{regularity}) and (\ref{unif conv}), we have to show that
		\begin{eqnarray}\label{main conv}
		f_k(u_k) \rightarrow f(u) ~\hbox{ strongly in }L^1(0,T; L^2(\Omega)).
		\end{eqnarray}

		In order to prove (\ref{main conv}), first we observe that
		\begin{eqnarray}\label{I}
		&&\|f_k(u_k) - f(u)\|_{L^1(0,T;L^2(\Omega))}\\
		&& \lesssim  \|f_k(u_k) - f(u_k)\|_{L^1(0,T;L^2(\Omega))} +  \|f(u_k) - f(u)\|_{L^1(0,T;L^2(\Omega))}.\nonumber
		\end{eqnarray}

		In view of (\ref{ass on f}) one has
		\begin{eqnarray*}
			|f(u_{k}) - f(u)| &\leq& c (1 + |u_{k}|^{p-1} + |u|^{p-1})|u_{k} - u|\\
			&=& c \left( |u_{k} - u| + |u_{k}|^{p-1}\, |u_{k} - u| + |u|^{p-1}\,|u_{k} - u|\right).
		\end{eqnarray*}

		
		Using H\"older and interpolation inequalities,
		\begin{eqnarray}\label{previous}
		\| |u_{k}|^{p-1}|u_{k} - u|\|_{L^{1}(0, T; L^{2}(\Omega))}  & \leq & \int_{0}^{T}{\|u_{k}\|_{L^{2p}}^{p-1} \|u_{k}-u\|_{L^{2p}}}dt \\
		& \leq & C \int_{0}^{T}{\|u_{k}\|_{L^{2p}}^{p-1}\|u_{k} - u\|_{L^{2}}^{\theta}\|u_{k} - u\|_{L^{10}}^{1-\theta} }dt\nonumber
		\end{eqnarray}
		with $\theta = \frac{5-p}{4p}$. Hence, using H\"older again with $q_{1} = \frac{5}{p-1}$, $q_{2} = \frac{8p}{5-p}$ and $q_{3} = \frac{4p}{p-1}$, we have
		\begin{eqnarray*}
			&&\| |u_{k}|^{p-1}|u_{k}- u|\|_{L^{1}(0, T; L^{2}(\Omega))}\\
			&& \leq  C(T) \|u_{k}\|_{L^{5}(0, T; L^{10}(\Omega))}^{p-1}\|u_{k} - u\|_{L^{2}(0, T; L^{2}(\Omega))}^{\theta} \|u_{k} - u\|_{L^{5}(0, T; L^{10}(\Omega))}^{1- \theta} \\
			&& \rightarrow  0 \quad \hbox{ as } k\to \infty.
		\end{eqnarray*}
		Note that in the above limit, due to $(\ref{Stri-uk})$ and $\|u_{k}\|_{L^{5}(0, T; L^{10}(\Omega))}, \|u^{k} - u\|_{L^{5}(0, T; L^{10}(\Omega))} \leq C$ for every fixed $T> 0$, we have $\|u_{k} - u\|_{L^{2}(0, T; L^{2}(\Omega))} \rightarrow 0$ by the Aubin-Lions-Simon lemma \cite{Simon}. Treating the other terms similarly, we conclude that
		\begin{eqnarray}\label{II}
		\| f(u_{k}) - f(u)\|_{L^{1}(0, T; L^{2}(\Omega))} \rightarrow 0.
		\end{eqnarray}
		From (\ref{I}) and (\ref{II}) it remains to prove that
		\begin{eqnarray}\label{III}
		\| f_k(u_{k}) - f(u_k)\|_{L^{1}(0, T; L^{2}(\Omega))}  \rightarrow 0 \hbox{ as } k \rightarrow \infty.
		\end{eqnarray}

		In fact, for each $t\in [0,T]$ we can set $$\Omega_k^t:=\{x\in \Omega:~ |u_k(x,t)| >k\}.$$
		
		Now we note from the definition of $f_k$ in (\ref{trunc func}) that
		\begin{eqnarray}\label{droping |u_k|<1}
		f_k(u_k) - f(u_k) =0, \hbox{ if } |u_k(x,t)|\leq k.
		\end{eqnarray}
		
		Consequently,  we can write
		\begin{eqnarray*}
			\| f_k(u_{k}) - f(u_k)\|_{L^{1}(0, T; L^{2}(\Omega))} = \| f_k(u_{k}) - f(u_k)\|_{L^{1}(0, T; L^{2}(\Omega_k^t))}.
		\end{eqnarray*}
		
		Thus, in view of (\ref{ass on f}) and $1\leq k < |u_k(x,t)|$ for all $x\in \Omega_k^t$ and $t\in [0,T]$, we have the following estimate for $1\leq p<5$:
		\begin{eqnarray}
		|f_k(u_k) - f(u_k)|^2 &\lesssim&  \left[  |f(u_k)|^2 +  |f(-k)|^2 +|f(k)|^2\right]\label{IV}\\
		&\lesssim& \left[ [|u_k|^{2} + |u_k|^{2p}] + [|k|^2 + |k|^{2p}]\right]\nonumber\\
		&\lesssim&\left[ |u_k|^{2p}+ |k|^{2p}\right],\nonumber\\
		&\lesssim& |u_k|^{2p}.\nonumber
		\end{eqnarray}
		It follows from (\ref{droping |u_k|<1}), (\ref{IV}) and $1\leq p <5$ that it is sufficient to prove
		\begin{eqnarray}
		&&\|\, |u_k|^{p} \, \|_{L^1(0,T;L^2(\Omega_k^t))} \rightarrow 0 \hbox{ as } k\rightarrow +\infty.\label{step1}
		\end{eqnarray}
		
		Before estimating the term in (\ref{step1})
		we observe that
		\begin{eqnarray*}
			\left( \int_{\Omega_k^t} k^{10}\,dx\right)^{1/2} \leq \left( \int_{\Omega_k^t}|u_k|^{10} \,dx\right)^{1/2},
		\end{eqnarray*}
		which implies $\left(\operatorname{meas}(\Omega_k^t)\right)^{1/2} \leq k^{-5} \|u_k(t)\|_{L^{10}(\Omega_k^t)}^5,$
		from which we conclude, taking (\ref{Stri-uk}) into account, that
		\begin{eqnarray}\label{meas(Omega_k)}
		\int_0^T\left(\operatorname{meas}(\Omega_k^t)\right)^{1/2}\,dt &\leq&  k^{-5} \|u_k\|_{L^5(0,T;L^{10}(\Omega))}^5
		\leq c(T)  k^{-5}.
		\end{eqnarray}
		
		Proceeding as in (\ref{previous}), we deduce
		\begin{eqnarray}\label{VIII'}
		\||u_k|^{p} \|_{L^1(0,T;L^2(\Omega_k^t))}&=&\int_0^T\left(\int_{\Omega_k^t} |u_k|^{2p}\,dx\right)^{1/2} \,dt\\
		&=& \int_0^T \|u_k(t)\|_{L^{2p}(\Omega_k^t)}^p\,dt\nonumber\\
		& \lesssim& \int_0^T \|u_k(t)\|_{L^{2p}(\Omega_k^t)}^{p-1}\|u_k(t)\|_{L^{2p}(\Omega_k)} \, dt\nonumber\\
		& \lesssim& \|u_k\|_{L^5(0,T;L^{10}(\Omega_k^t))}^{p-1}\,\|u_k\|_{L^2(0,T;L^{2}(\Omega_k^t))}^{\theta}\,\|u_k\|_{L^5(0,T;L^{10}(\Omega_k^t))}^{1-\theta}.\nonumber
		\end{eqnarray}
		
		In addition, (\ref{meas(Omega_k)}) yields
		\begin{eqnarray}\label{IX''}
		\|u_k\|_{L^2(0,T;L^{2}(\Omega_k^t))}^{\theta} &\leq& \left(\int_0^T\left(\operatorname{meas}(\Omega_k^t)^{1/2}\right)dt\right)^{\frac{\theta}{2}} \left(\int_0^T \|u_k(t)\|_{L^4(\Omega)}^2 \,dt\right)^{\frac{\theta}{2}}\\
		&\lesssim& k^{\frac{-5}{2}\theta} \rightarrow 0  \hbox{ as } k\rightarrow +\infty\nonumber
		\end{eqnarray}
		because  $H_0^1(\Omega) \hookrightarrow L^4(\Omega)$. Therefore, $\|u_k(t)\|_{L^4(\Omega)} \lesssim \|u_k(t)\|_{H_0^1(\Omega)} \lesssim E_{u_k}(0)$.
		Combining (\ref{Stri-uk}),  (\ref{VIII'}) and (\ref{IX''}), we deduce (\ref{step1}) and consequently (\ref{III}), as desired.


		
		The cases $n=1,2$ are not difficult to prove for $p\geq 1$ and consequently will be omitted.

		\medskip
		
		Defining $z_{\mu,\sigma}=u_{\mu}-u_{\sigma}$, $\mu,\sigma\in
		\mathbb{N}$, from (\ref{eq:AP}) it holds that
		\begin{eqnarray}\label{eq:3.59}
		&&\frac12 \frac{d}{dt} \left\{\|\partial_t z_{\mu,\sigma}(t)\|_{L^2(\Omega)}^2 +\|\nabla
		z_{\mu,\sigma}(t)\|_{L^2(\Omega)}^2 \right\}\\
		&&+ \int_{\Omega}a(x)\left(g(\partial_t u_{\mu})-
		g(\partial_t u_{\sigma})\right)(\partial_t u_{\mu}-\partial_t u_{\sigma}) \,dx \nonumber \\
		&&=  -\int_{\Omega}\left(f_{\mu}(u_{\mu})-
		f_{\sigma}(u_{\sigma})\right)(\partial_t u_{\mu}-\partial_t u_{\sigma}) \,dx.
		\nonumber
		\end{eqnarray}
		
		Integrating (\ref{eq:3.59}) over $(0,t)$ with $t\in [0,T]$, we obtain
		\begin{eqnarray}\label{eq:3.60}
		&&\frac12  \left\{\|\partial_t z_{\mu,\sigma}(t)\|_{L^2(\Omega)}^2 +\|\nabla
		z_{\mu,\sigma}(t)\|_{L^2(\Omega)}^2 \right\}\\
		&&+\int_0^t \int_{\Omega}a(x)\left(g(\partial_t u_{\mu})-
		g(\partial_t u_{\sigma})\right)(\partial_t u_{\mu}-\partial_t u_{\sigma}) \,dx ds\nonumber\\
		&&\leq  \frac12\left\{\|u_{1,\mu} - u_{1,\sigma}\|_{L^2(\Omega)}^2 + \|\nabla u_{0,\mu}
		- \nabla u_{0,\sigma}\|_{L^2(\Omega)}^2\right\}\nonumber\\
		&&+ \left|\int_0^t \int_{\Omega}\left(f_{\mu}(u_{\mu})-
		f_{\sigma}(u_{\sigma})\right)(\partial_t u_{\mu}-\partial_t u_{\sigma}) \,dxds\right|.
		\nonumber
		\end{eqnarray}

		We observe that the convergences (\ref{conv init data}), (\ref{conver2}) and (\ref{main conv}) (the latter being valid when $1\leq p < \frac{n+2}{n-2}$) imply the convergence to zero
		(as $\mu, \sigma \rightarrow +\infty$) of the terms on the RHS
		of (\ref{eq:3.60}). Thus, from (\ref{eq:3.60}) we deduce that
		\begin{eqnarray}\label{Cauchy conv}
		&&u_{\mu} \rightarrow u \hbox { in }
		C^0([0,T];H_{0}^1(\Omega)) \cap C^1 ([0,T]; L^2(\Omega)),\\
		&&\lim_{\mu,\sigma \rightarrow +\infty}
		\int_0^T \int_{\Omega}a(x)\left(g(\partial_t u_{\mu})-
		g(\partial_t u_{\sigma})\right)(\partial_t u_{\mu}-\partial_t u_{\sigma}) \,dx ds = 0\label{conv damp1}
		\end{eqnarray}
		for all $T>0$, which proves (\ref{regularity}).
		
		It remains to prove that $\chi = a(x) g(\partial_t u)$. From \eqref{Cauchy conv} it follows that $\partial_t u_k \rightarrow \partial_{t} u$ a.e. in $\Omega \times (0,T)$ as $k \rightarrow \infty$. The continuity of the function $g(\cdot)$ ensures that $a(x)g(\partial_tu_{k})\rightarrow a(x)g(\partial_{t} u)$ a.e. in $\Omega \times (0,T)$ as $k \rightarrow \infty$. Employing Strauss' Lemma we deduce that $\chi=a(x)g(\partial_t u)$ as desired.

		
		%
		
		The uniqueness of solution for the range $1\leq p < \frac{n+2}{n-2}$ follows the same idea in \cite{Astudillo2} and consequently will be omitted.
		
		\medskip
		Our first result reads as follows:
		\begin{Theorem}[Wellposedness]\label{theo 2}
			Assume that $a\in L^\infty(\Omega)$ is nonnegative, $f$ satisfies the conditions given in Assumption \ref{assmpflabel} with $n=3$ and $\mathcal{O}(g)=1$. Then, given $\{u_0,u_1\}\in H_0^1(\Omega) \times L^2(\Omega)$, the problem (\ref{eq:*}) possesses a unique global solution  $u$ which satisfies
			{\small
				\begin{eqnarray*}
					&u\in C^0([0,T];H_0^1(\Omega)),~\partial_t u \in C^1(0,T; L^2(\Omega)),~\partial_t^2 u \in C^0([0,T]; H^{-1}(\Omega)),&\\
					&u \in L^5(0,T,L^{10}(\Omega))\cap L^4(0,T;L^{12}(\Omega)).&
			\end{eqnarray*}}
		\end{Theorem}

		\subsection{Estimating \texorpdfstring{$F_k(u_k)$}. }
		
		\medskip
		
		Inequality (\ref{bound on Fk}) gives $|F_k(s)| \leq c [|s|^2 + |s|^{p+1}]$ for all $s\in \mathbb{R}$ and $k\in \mathbb{N}$.
		
		Now, assuming that $1\leq p < \frac{n+2}{n-2}$ for $n>2$, then $2\leq p+1 < \frac{2n}{n-2}=2^{\ast}$. Thus, there exists $\varepsilon >0$ such that $p+1+\varepsilon=2^{\ast}$, and, consequently, $H_0^1(\Omega) \hookrightarrow L^{p+1+\varepsilon}(\Omega)$. Then,
		\begin{eqnarray}\label{bound data Lp'}
		\int_\Omega |F_k(u_{0,k})|^\frac{p+1+\varepsilon}{p+1}\,dx &\leq& c\int_\Omega |u_{0,k}|^{\frac{2(p+1+\varepsilon)}{p+1}} + |u_{0,k}|^{p+1+\varepsilon} \,dx\\
		&\lesssim&  \|u_{0,k}\|_{H_0^1(\Omega)}^{p+1+\varepsilon} \leq C,\nonumber
		\end{eqnarray}
		from which we conclude that $E_{u_k}(0)$ is bounded.
		Analogously,
		\begin{eqnarray}\label{bound on Fk'}
		\int_\Omega |F_k(u_{k}(x,t_0))|^\frac{p+1+\varepsilon}{p+1}\,dx
		\lesssim \|u_{k}(\cdot,t_0)\|_{H_0^1(\Omega)}^{p+1+\varepsilon} \leq C E_{u_k}(0)^{p+1+\varepsilon}
		\end{eqnarray}
		for all $t_0\in [0,T]$.
		The boundedness of $E_{u_k}(0)$ implies that there exists $\Xi \in L^{\frac{2^{\ast}}{p+1}}(\Omega )$ so that
		\begin{eqnarray}\label{weak conv of Fk}
		F_k(u_k(\cdot,t_0))\rightharpoonup \Xi \hbox{ weakly in }L^{\frac{2^{\ast}}{p+1}}(\Omega ), \hbox{ as } k \rightarrow +\infty.
		\end{eqnarray}
		

		In what follows we are going to prove that $\Xi=F(u(\cdot,t_0))$. Indeed, from \eqref{Cauchy conv} we obtain $u_k(\cdot,t_0) \rightarrow u(\cdot,t_0)$ strongly in $L^2(\Omega)$.  Thus,
		\begin{equation}\label{victor1}
		u_k(x,t_0) \rightarrow u(x,t_0) \hbox{ a.e. in } \Omega.
		\end{equation}
		
		Note that
		\begin{eqnarray}\label{victor3}
		&&|F_k(u_k(x,t_0))-F(u(x,t_0))| \\
		&&\leq |F_k(u_k(x,t_0))-F(u_k(x,t_0))|+|F(u_k(x,t_0))-F(u(x,t_0))|.\nonumber
		\end{eqnarray}
		
		The convergence (\ref{victor1}) and the continuity of $F$ imply
		\begin{eqnarray}\label{victor4}
		F(u_k(x,t_0)) \rightarrow F(u(x,t_0)) \hbox{ a.e. in } \Omega.
		\end{eqnarray}
		
		In light of inequality \eqref{victor3}, to prove that
		\begin{eqnarray}\label{conv a.e. F_k}
		F_k(u_k(x,t_0)) \rightarrow F(u(x,t_0)) \hbox{ a.e. in } \Omega,
		\end{eqnarray}
		it is now enough to show that
		\begin{eqnarray*}
			F_k(u_k(x,t_0)) \rightarrow F(u_k(x,t_0)) \hbox{ a.e. in } \Omega.
		\end{eqnarray*}
		From (\ref{boundedness}), there exists a positive constant $L=L(x,t)>0$ such that \begin{eqnarray}\label{victor5}
		|F_k(u_k(x,t_0))-F(u_k(x,t_0))| &=& \left|\int_{0}^{u_k(x,t_0)} f_k(s)ds-\int_{0}^{u_k(x,t_0)} f(s)ds\right|\nonumber\\
		&\leq& \int_{-L}^{L} |f_k(s)-f(s)|ds =0, \hbox{ if }k\geq L.
		\end{eqnarray}

		Therefore, combining \eqref{victor3}, \eqref{victor4} and \eqref{victor5}, we obtain \eqref{conv a.e. F_k}. Thus, the above yields
		\begin{eqnarray}\label{main weak conv}
		F_k(u_k(\cdot,t_0))\rightharpoonup F(u(\cdot,t_0)) \hbox{ weakly in }L^{\frac{2^{\ast}}{p+1}}(\Omega ), \hbox{ as } k \rightarrow +\infty.
		\end{eqnarray}
		In addition, employing Strauss's Lemma we also deduce that
		\begin{eqnarray}\label{main strong conv}
		F_k(u_k(\cdot,t_0))\rightarrow F(u(\cdot,t_0)) \hbox{ strongly in }L^r(\Omega ), \hbox{ as } k \rightarrow +\infty
		\end{eqnarray}
		for all $1\leq r < \frac{2^{\ast}}{p+1}$.

		\medskip
		
		Now, we have the following result:
		\begin{Theorem}\label{existence}
			Assume that $a\in L^\infty(\Omega)$ is nonnegative and $f$ satisfies the conditions given in Assumption \ref{assmpflabel} with $n=3$ and $\mathcal{O}(g)=1$. Then, given $\{u_0,u_1\}\in H_0^1(\Omega) \times L^2(\Omega)$ problem (\ref{eq:*}) has a unique global solution
			{\small
				$$u\in C^0([0,T];H_0^1(\Omega)),~\partial_t u \in C^0([0,T]; L^2(\Omega)),~\partial_t^2 u \in L^2(0,T; H^{-1}(\Omega) ).$$}
			In addition, the energy identity below holds
			\begin{equation}\label{energy identity}
			E_u(t_1)+\int_{t_1}^{t_2}\int_{\Omega}a(x)g(\partial_t u(x,t)) \partial_t u(x,t) \, dx dt=E_u(t_2),
			\end{equation}
			where $E_u$ is as in \eqref{Ident energ orig prob}.
		\end{Theorem}

		\section{Decay rate estimates}\label{section5}
		\setcounter{equation}{0}
		Throughout this section we will assume that $\mathcal{O}(g)=1$, $1< p < \frac{n+2}{n-2}$ for $n>3$ and $p>1$ if $n=1,2$. However, for simplicity, we shall focus on the case of dimension $n =3$.
		In this section, we will first establish the observability inequality to the truncated auxiliary problem (\ref{eq:AP}).
		The energy associated to problem (\ref{eq:AP}) is given by
		\begin{eqnarray}\label{energy}\qquad
		E_{u_k}(t):= \frac12\int_\Omega |\partial_t u_k(x,t)|^2 + |\nabla u_k(x,t)|^2\,dx + \int_\Omega F_k(u_k(x,t))\,dx,
		\end{eqnarray}
		and the energy identity associated to problem \eqref{eq:AP} reads as follows:
		\begin{eqnarray}\label{ident energy mu}
		E_{u_k}(t_2) - E_{u_k}(t_1) = - \int_{t_1}^{t_2}\int_\Omega  a(x) g(\partial_t u_k(x,t)) \partial_t u_k(x,t)\,dxdt
		\end{eqnarray}
		for $0 \leq t_1 \leq t_2 <+\infty$.
		Let $T_0>0$ be associated with the geometric control condition, namely, every ray of the geometric optics enters in $\omega$ at a time $T<T_0$.
		
		\subsection{Observability for the truncated problem}
		
		{
		In this section, we abuse notation and set $$E(u_0,u_1)\equiv \frac{1}{2}\left(\|u_1\|_{L^2(\Omega)}^2+\|\nabla u_0\|_{L^2(\Omega)}^2\right)+\int_{\Omega}F(u_0)dx,$$ where $u_0$ and $u_1$ represent data for the original model and set
		$$E_k(u_{0,k},u_{1,k})\equiv \frac{1}{2}\left(\|u_{1,k}\|_{L^2(\Omega)}^2+\|\nabla u_{0,k}\|_{L^2(\Omega)}^2\right)+\int_{\Omega}F_k(u_{0,k})dx,$$ where $u_{0,k},u_{1,k}$ represent data for the truncated model.	
		
		Let $0<r<R$ and $u_0,u_1,u_{0,k},u_{1,k}$ be such that $r<E(u_0,u_1)<R$ and 
			\begin{eqnarray}\label{sequencebound}
			r<E_k(u_{0,k},u_{1,k}) <R, k\ge 1.
		\end{eqnarray}  
	
	Our first goal is to prove that the observability inequality for the truncated model holds true.  This is achieved in the first part of the lemma below in which the constant of the observability inequality depends on $k$ as well as the bounded set (annulus) from which data are taken. 

Our second goal is to show that if we choose a sequence of approximate data $\{u_{0,k},u_{1,k}\}$  which satisfies \eqref{conv init data} as well as \eqref{sequencebound}, then the constant of the observability inequality for approximate solutions can be made independent of $k$.
		 Note that a sequence which satisfies \eqref{conv init data} as well as \eqref{sequencebound} can easily be constructed by using density of smooth compactly supported functions in the class of functions from which $u_0,u_1$ are taken.
We claim the following lemma.

}
		
		\begin{Lemma}\label{ObsApprox} Let $r,R>0$, $T\ge T_0$, $k\ge 1$. Then, there is some $C_k=C_k(r,R,T)>0$ such that the corresponding solution $u_k$ of (\ref{eq:AP}) with $( u_{0,k}, u_{1,k}) \in \mathscr{D}(\mathbb{A})$ away from zero (i.e., verifying \eqref{sequencebound}) satisfies the observability inequality
			\begin{eqnarray}\label{obs ineq}
			E_k(u_{0,k}, u_{1,k})  \leq {C_k} \int_{0}^{T}\int_\Omega  a(x) \left(|\partial_t u_k|^2 + |g(\partial_t u_k)|^2\right) \,dxdt.
			\end{eqnarray}
{		Moreover, for a convergent sequence of data, i.e., if there is $(u_0,u_1)$ in $\mathcal{H}$ with $r<E(u_0,u_1)<R$ such that $( u_{0,k}, u_{1,k})\rightarrow (u_0,u_1)$ in $\mathcal{H}$ and \eqref{sequencebound} holds, then there is $T^*\ge T_0$ such that the constant of the inequality \eqref{obs ineq} can be chosen independent of $k$ provided $T\ge T^*$.}
		\end{Lemma}
		\begin{proof}
Our proof relies on a contradiction argument.  So, if Lemma is false, for every constant $C>0$ there exists initial data $( u^C_{0,k},u^C_{1,k}) \in \mathscr{D}(\mathbb{A})$ verifying (\ref{sequencebound}) for which corresponding solution $u^C_k$ violates (\ref{obs ineq}).
			
			In particular,  for each fixed $m\in\mathbb{N}$, we obtain the existence of initial data $\{u^m_{0,k},u^m_{1,k}\}$ verifying (\ref{sequencebound}) and for which corresponding solution $u^m_k$ satisfies the reverse inequality
			
			\begin{eqnarray}\label{false}
			E_{u^m_k}(0)> m \left(\int_{0}^{T}\int_\Omega  a(x) \left(|\partial_t u_k^m|^2 + |g(\partial_t u_k^m)|^2\right) \,dxdt\right).
			\end{eqnarray}
			Then, we obtain a sequence $\{u_k^m\}_{m\in \mathbb{N}}$ of solutions to problem (\ref{eq:AP}) such that
			\begin{eqnarray}\label{normal conv}
			\lim_{m \rightarrow +\infty}\frac{\int_{0}^{T}\int_\Omega a(x)\left(|\partial_t u_k^m|^2 + |g(\partial_t u_k^m)|^2\right) \,dxdt}{E_{u_k^m}(0)}=0.
			\end{eqnarray}
			
			Once $E_{u_k^m}(0)$ is uniformly bounded, from (\ref{normal conv}) we infer
			\begin{eqnarray}\label{conv damp}
			\lim_{m \rightarrow +\infty}\int_{0}^{T}\int_\Omega a(x)\left(|\partial_t u_k^m|^2 + |g(\partial_t u_k^m)|^2\right) \,dxdt=0.
			\end{eqnarray}
			
			Furthermore, we deduce there exists a subsequence of $\{u_k^m\}_{m\in \mathbb{N}}$, which from now on we will denote by the same notation, such that
			\begin{eqnarray}
			&&u_k^m \rightharpoonup u_k \hbox{ weakly-star in } L^{\infty}(0,T; H_0^1(\Omega)), \hbox{ as } m\rightarrow +\infty,\label{conv1'}\\
			&&\partial_t u_k^m \rightharpoonup \partial_t u_k \hbox{ weakly-star in } L^{\infty}(0,T; L^2(\Omega)), \hbox{ as }m\rightarrow +\infty,\label{conv2'}\\
			&& u_k^m \rightarrow u_k \hbox{ strongly in } L^{\infty} (0,T; L^q(\Omega)),  \hbox{ as }m\rightarrow +\infty,  \hbox{ for all } q\in \left[2, \frac{2n}{n-2}\right),\label{conv3'}
			\end{eqnarray}
			where the last convergence is due to Simon \cite{Simon}. At this point in the proof we will divide it into two cases: $u_k\ne 0$ and $u_k=0$.
			
			\medskip
			Case~(a):~$u_k\ne 0$. ~
			\medskip
			
			Passing to the limit in problem
			\begin{equation*}
			\begin{cases}
			\partial_t^2 u_{k}^{m} -\Delta u_{k}^{m} + f_k(u_{k}^{m}) + a(x)  g(\partial_t u_k^m) = 0 & \hbox{ in } \Omega \times (0, T),\\
			u_k^m=0&\hbox{ on } \partial \Omega \times (0, T)\\
			u_k^m(x,0)=u_{0,k}^m(x);\quad \partial_t u_k^m{(x,0)} =u_{1,k}^m(x), & \hbox{ in }\Omega,
			\end{cases}
			\end{equation*}
			and taking (\ref{conv damp}) into consideration, we deduce that $u_k$ solves
			\begin{equation}\label{limit1'}
			\begin{cases}
			\partial_t^2 u_k -\Delta u_k + f_k(u_k)= 0 & \hbox{ in } \Omega \times (0, T),\\
			u_k=0& \hbox{ on } \partial \Omega \times (0, T ),\\
			\partial_t u_k=0 & \hbox{ a.e. in }\omega,
			\end{cases}
			\end{equation}
			and for $y_k = \partial_t u_k$, in the distributional sense, one has
			\begin{equation*}
			\begin{cases}
			\partial_t^2 y_k -\Delta y_k + f_k'(u_k)y_k= 0 & \hbox{ in } \Omega \times (0, +\infty),\\
			y_k=0& \hbox{ on } \partial \Omega \times (0,+\infty ),\\
			y_k=0 & \hbox{ a.e. in }\omega.
			\end{cases}
			\end{equation*}
			
			Since $ f_k'(u_k)\in L^\infty (\Omega \times (0,T))$ because $f_k$ is globally Lipschitz, for each $k \in \mathbb{N}$, we deduce from  Assumption \ref{assumption1.3} that $y_k= \partial_t u_k\equiv 0$. Returning to (\ref{limit1'}), we deduce that $u_k\equiv 0$ as well, which is a contradiction.
			
			\medskip
			Case~(b):~$u_k= 0$. ~
			\medskip
			
			From (\ref{conv1'}), (\ref{conv2'}) and (\ref{conv3'}), we now have:
			\begin{eqnarray}
			&&u_k^m \rightharpoonup 0\hbox{ weakly-star in } L^{\infty}(0,T; H_0^1(\Omega)), \hbox{ as }m\rightarrow +\infty,\label{conv1''}\\
			&&\partial_t u_k^m \rightharpoonup  0 \hbox{ weakly-star in } L^{\infty}(0,T; L^2(\Omega)), \hbox{ as }m\rightarrow +\infty,\label{conv2''}\\
			&& u_k^m \rightarrow 0 \hbox{ strongly in } L^{\infty} (0,T; L^q(\Omega)),  \hbox{ as }m\rightarrow +\infty, \hbox{ for all } q\in \left[2, \frac{2n}{n-2}\right).\label{conv3''}
			\end{eqnarray}
			
			Setting
			\begin{eqnarray}\label{def v_k}
			\alpha_m := \sqrt{E_{u_k^m}(0)} ~ \hbox{ and }~v_k^m:= \frac{u_k^m}{\alpha_m},
			\end{eqnarray}
			(\ref{normal conv}) yields
			\begin{eqnarray}\label{damping conv}
			\lim_{m \rightarrow +\infty} \int_{0}^{T}\int_\Omega a(x) \left(|\partial_t v_k^m|^2 + \frac{1}{\alpha_m^2}|g(\partial_t u_k)|^2\right) \,dxdt=0.
			\end{eqnarray}
			
			Let $\{v_k^m\}_{m\in \mathbb{N}}$ be as in (\ref{def v_k}). Then we have the sequence of normalized problems
			\begin{equation}\label{eq:NP}
			\begin{cases}
			\partial_t^2 v_{k}^m -\Delta v_{k}^m + \frac{1}{\alpha_m} f_k(u_k^m)  + \frac{1}{\alpha_m} a(x)g(\partial_t u_k^m) = 0& \hbox{ in } \Omega \times (0, +\infty) \\
			v_k^m=0 & \hbox{ on } \partial \Omega \times (0,+\infty ) \\
			\displaystyle v_k^m(x,0)=\frac{u_{0,k}^m}{\alpha_m}; \partial_t v_k^m(x,0)=\frac{u_{1,k}^m}{\alpha_m} & \hbox{ in } \Omega.
			\end{cases}
			\end{equation}
			
			We observe that
			\begin{eqnarray*}
				\frac{1}{\alpha_m}\int_\Omega f_k(u_k^m) \partial_t v_k^m\,dx &=&\frac{1}{\alpha_m^2} \int_\Omega f_k(\alpha_m v_k^m) \partial_t(\alpha_m v_k^m)\,dx\\
				&=& \frac{1}{\alpha_m^2} \frac{d}{dt}\int_\Omega F_k(\alpha_m v_k^m)\,dx\\
				&=& \frac{1}{\alpha_m^2} \frac{d}{dt}\int_\Omega F_k(u_k^m)\,dx,
			\end{eqnarray*}
			so that
			\begin{eqnarray*}
				E_{v_k^m}(t) = \frac12 \int_\Omega \left(|\partial_t v_k^m|^2 + |\nabla v_k^m|^2 \right)\,dx + \frac{1}{\alpha_m^2} \int_\Omega F_k(u_k^m)\,dx.
			\end{eqnarray*}

			It is not difficult to check that $E_{v_k^m}(t)= \frac{1}{\alpha_m^2} E_{u_{k}^m}(t)$ for all $t\geq 0$. Then, in particular,
			\begin{eqnarray}\label{norm initial energy}
			E_{v_k^m}(0)= \frac{1}{\alpha_k^2} E_{u_{k}^m}(0)=1, \hbox{ for all } m\in \mathbb{N}.
			\end{eqnarray}
			
			In order to achieve a contradiction, our main goal is to prove that
			\begin{eqnarray}\label{main goal}
			\lim_{m\rightarrow +\infty} E_{v_k^m}(0)=0.
			\end{eqnarray}
			
			First we observe that from (\ref{norm initial energy}) we deduce there exists a subsequence of $\{v_k^m\}_{m\in \mathbb{N}}$, which from now on denote by the same notation, such that
			\begin{eqnarray}
			&&v_k^m \rightharpoonup v_k \hbox{ weakly-star in } L^{\infty}(0,T; H_0^1(\Omega)), \hbox{ as }m\rightarrow +\infty,\label{conv1}\\
			&&\partial_t v_k^m \rightharpoonup \partial_t v_k \hbox{ weakly-star in } L^{\infty}(0,T; L^2(\Omega)), \hbox{ as }m\rightarrow +\infty,\label{conv2}\\
			&& v_k^m \rightarrow v_k \hbox{ strongly in } L^{\infty} (0,T; L^q(\Omega)), \hbox{ as }m\rightarrow +\infty, \hbox{ for all } q\in [2, \frac{2n}{n-2}).\label{conv3}
			\end{eqnarray}
			
			For some eventual subsequence, still denoted with index $m$, we have $\alpha_m \rightarrow \alpha$ with $\alpha>0$. Case $\alpha=0$ is excluded, since \eqref{sequencebound} is assumed.
			
%
%
%
%
			Passing to the limit in (\ref{eq:NP}) as $m\rightarrow +\infty$ and taking (\ref{damping conv}), (\ref{conv1}), (\ref{conv2}) and (\ref{conv3}) into account, we deduce
			\begin{equation}\label{limit1}
			\begin{cases}
			\partial_t^2 v_{k} -\Delta v_{k} = 0 & \hbox{ in }\Omega \times (0, T), \\
			v_k=0 & \hbox{ on } \partial \Omega \times (0,T),\\
			\partial_t v_k=0 & \hbox{ a.e. in }\omega.
			\end{cases}
			\end{equation}
			$w_k= \partial_t v_{k}$ yields, in the distributional sense,
			\begin{equation}
			\begin{cases}
			\partial_t^2 w_{k} -\Delta w_{k} = 0 & \hbox{ in } \Omega \times (0, +\infty),\\
			w_k=0 & \hbox{ on } \partial \Omega \times (0,+\infty),\\
			w_k=0 & \hbox{ a.e. in }\omega.
			\end{cases}
			\end{equation}
			
			We deduce from the observability of the linear wave equation, see \cite{Bardos}, that $w_k= \partial_t v_k\equiv 0$. Therefore, returning to (\ref{limit1}), we deduce that $v_k\equiv 0$ as well.

			We observe that Lemma \ref{Lema2} yields
			\begin{eqnarray*}
			\frac{1}{\alpha_m^2} |f_k(u_k^m)|^2 \leq c_k \frac{1}{\alpha_m^2} |u_k^m|^2 = \frac{c_k}{\alpha_m^2}\alpha_m^2 |v_k^m|^2,
			\end{eqnarray*}
			from which we deduce that
			\begin{eqnarray}\label{crucial bound'}
			\frac{1}{\alpha_m^2} \int_0^T \int_\Omega |f_k(u_k^m)|^2 \,dxdt \leq  c_k \int_0^T\int_\Omega |v_k^m|^2\,dxdt
			\end{eqnarray}
			
			
			Coming back to (\ref{crucial bound'}), considering (\ref{conv3}) and $v_k=0$,  we deduce that
			\begin{eqnarray}\label{crucial bound}
			\frac{1}{\alpha_m^2} \int_0^T \int_\Omega |f_k(u_k^m)|^2 \,dxdt \rightarrow 0 \hbox{ in } L^2(0,T,L^2(\Omega)).
			\end{eqnarray}
			
			Let $\Box= \partial_{t}^{2} - \Delta$ the d'Alembert operator. We have from (\ref{eq:NP}) that
			\begin{eqnarray*}
				\Box v_k^m = -\frac{1}{\alpha_m} f_k(v_k^m)  - \frac{1}{\alpha_m}  a(x) g(\partial_t u_k^m),
			\end{eqnarray*}
			from which we deduce (see (\ref{damping conv}) and (\ref{crucial bound}) and having in mind that $v_k\equiv0$), that
			{\small
				\begin{eqnarray}\label{main identity}\quad
				\partial_t\Box v_k^m =\partial_t \left( -\frac{1}{\alpha_m} f_k(v_k^m)  - \frac{1}{\alpha_m}  a(x) g(\partial_t u_k^m)\right)\rightarrow 0 \hbox{ strongly in } H^{-1}_{loc}(\Omega \times (0,T)).
				\end{eqnarray}}

			Let $\mu$ be the microlocal defect measure associated with $\{\partial_t v_k^m\}_{m \in \mathbb{N}}$ in $L^2_{loc}(\Omega \times (0,T))$. In view of (\ref{main identity}), we deduce that:
			
			(i) The $\hbox{supp}(\mu)$ is contained in the characteristic set of the wave operator $\{\tau^2=\|\xi\|^2\}$ (see Theorem \ref{Theorem 4.55}).

			Our wish is to propagate the convergence of $\partial_t v_k^m$ from $L^2(\omega \times (0,T))$ to the whole $L^2(\Omega \times (0,T))$. Indeed, from (\ref{main identity}) we recall that:
			\begin{eqnarray}\label{Gerard2}
			\Box \partial_t v_k^m \rightarrow 0 \hbox{ in } H^{-1}_{loc}(\Omega \times (0,T)), \hbox{ as }m \rightarrow +\infty,
			\end{eqnarray}
			which also implies that
			
			(ii) $\mu$ propagates along the bicharacteristic  flow of this operator, which signifies, particularly, that if some point $\omega_0=(t_0,x_0,\tau_0,\xi_0)$ does not belong to the $\hbox{supp}(\mu),$ the whole bicharacteristic issued from $\omega_0$ is out of $\hbox{supp}(\mu)$.
			
			However, since $\hbox{supp}(\mu) \subset (\Omega \backslash \omega) \times (0,T)$ and we have a frictional damping acting in $\omega \times (0,T)$, we can propagate the kinetic energy from $\omega \times (0,T)$ towards $ (\Omega\backslash\omega) \times (0,T)$.
			
			Indeed, from Theorem \ref{Theorem 4.60} and Proposition \ref{Prop 4.63}, it follows that $\hbox{supp}(\mu)$ in $(0,T)\times \Omega\times S^n$ is a union of curves like
			\begin{eqnarray}\label{geodesics}
			t \in I\cap (0,\infty) \mapsto m_\pm(t)=\left(t, x(t), \frac{\pm1}{\sqrt{1+|\dot{x}|^2}},\frac{\mp  G(x)\dot{x}}{\sqrt{1+|\dot{x}|^2}} \right),
			\end{eqnarray}
			where $t\in I \mapsto x(t)\in \hbox{int}\,\Omega$ is a geodesic associated with the metric .
			
			From the convergence $\partial_t v_k^m \rightarrow 0$ strongly in $L^2(\omega\times (0,T))$ as $m \rightarrow +\infty,$ it follows that $\mu$ is supported in the set $(0,T) \times \Omega\backslash \omega \times S^n$. We affirm that $\operatorname{supp} \mu = \emptyset$. Indeed, let $(t_0,x_0,\tau_0,\xi_0) \in \operatorname{supp} \mu$ and $x$ be a geodesic of $G=I_d$ defined near $t_0$. Once the geodesics inside $\Omega\backslash \omega$ enters necessarily in the region $\omega$ we deduce from (\ref{geodesics}) that $m_{\mp}(t) \notin \operatorname{supp} \mu$ and, as consequence, $(t_0,x_0,\tau_0,\xi_0) \notin \operatorname{supp} \mu$. Therefore, $\hbox{supp}(\mu)$ is empty.
			
			From Remark \ref{Rem4.2}, it follows that
			\begin{equation}\label{convloc}
			\partial_t v_k^m \rightarrow 0 \hbox{ in } L_{loc}^{2}(\Omega \times (0,T)).
			\end{equation}
			Moreover, from the above convergences, we deduce that
			\begin{equation}\label{c 24}
			\partial_t v_{k}^{m} \rightarrow 0 \mbox{ in } L^{2}(\Omega \times (0,T)).
			\end{equation}
			
			Indeed,
			\begin{align}
			\begin{split}
			\int_{0}^{T} \int_{\Omega} |\partial_t v_{k}^{m}|^2 \, dx \, dt &= \int_{0}^{T} \int_{\omega} |\partial_t v_{k}^{m} |^{2}\, dx \, dt + \int_{0}^{T} \int_{\Omega\setminus\omega} |\partial_t v_{k}^{m}|^{2}\, dx \, dt \\
			&= L_1 + L_2,
			\end{split}
			\end{align}
			where
			$$L_1=\int_{0}^{T} \int_{\omega} |\partial_t v_{k}^{m} |^{2}\, dx \, dt \hbox{ and }  L_2=\int_{0}^{T} \int_{\Omega\setminus\omega} |\partial_t v_{k}^{m}|^{2}\, dx \, dt.$$
			From (\ref{normal conv}) we have that $L_1 \rightarrow 0$ where $m \rightarrow \infty.$ For $L_2,$ consider the following decomposition:
			\begin{align}
			\begin{split}
			L_2 &= \int_{0}^{\varepsilon} \int_{\Omega\setminus\omega} |\partial_t v_{k}^{m}|^{2} \, dx \, dt + \int_{\varepsilon}^{T - \varepsilon} \int_{\Omega\setminus\omega} |\partial_t v_{k}^{m}|^{2}\, dx \, dt + \int_{T - \varepsilon}^{T} \int_{\Omega \setminus \omega} |\partial_t v_{k}^{m}|^{2} \, dx \, dt\\
			&= J_1 + J_2 + J_3,
			\end{split}
			\end{align}
			where
			$$J_1=\int_{0}^{\varepsilon} \int_{\Omega\setminus\omega} |\partial_t v_{k}^{m}|^{2} \, dx \, dt, \ J_2=\int_{\varepsilon}^{T - \varepsilon} \int_{\Omega\setminus\omega} |\partial_t v_{k}^{m}|^{2}\, dx \, dt \hbox{ and } J_3=\int_{T - \varepsilon}^{T} \int_{\Omega \setminus \omega} |\partial_t v_{k}^{m}|^{2} \, dx \, dt.$$
			
			Note that
			\begin{align}
			\begin{split}
			J_1 = \int_{0}^{\varepsilon} \int_{\Omega\setminus \omega} |\partial_t v_{k}^{m}|^{2} \, dx \, dt \leq \int_{0}^{\varepsilon} 2 E_{v_{k}^{m}}(t) \, dx \, dt \leq 2 \varepsilon E_{v_{k}^{m}}(0) \leq 2 \varepsilon,
			\end{split}
			\end{align}
			since $ E_{v_{k}^{m}}(0) =1.$ Therefore, $\lim_{m \rightarrow \infty} J_1 \leq 2\varepsilon $ for all $T>\varepsilon>0$. Since $\varepsilon>0$ is arbitrary, it follows that $\lim_{m \rightarrow \infty} J_1=0$. Proceeding in the same way, we show that $J_3 \rightarrow 0$ as $m \rightarrow \infty$. Finally, from \eqref{convloc} we deduce that $J_2 \rightarrow 0$ as $m \rightarrow \infty$.
			
			Thus,
			\begin{eqnarray}\label{kinectic convergence}
			\int_0^T \int_\Omega |\partial_t  v_k^m(x,t)|^2\,dxdt \rightarrow 0, \hbox{ as } m \rightarrow + \infty.
			\end{eqnarray}
			
			Now, we are going to prove that $E_{v_k^m}(0)$ converges to zero. Indeed, let us consider the following cut-off function:
			\begin{align*}
			&\theta\in C^{\infty}(0,T),  \quad    0\leq \theta(t) \leq 1,  \quad   \theta(t)=1 \ \mbox{in} \ (\varepsilon,T-\varepsilon).
			\end{align*}
			
			Multiplying equation (\ref{eq:NP}) by $v_k^m \theta$ and integrating by parts, we infer
			\begin{eqnarray}\label{equipartition}
			&& -\int_0^T \theta(t)\int_\Omega |\partial_t v_k^m|^2\,dxdt - \int_0^T \theta'(t)\int_\Omega \partial_t v_k^m v_k^m\,dxdt\\
			&&+\int_0^T \theta(t) \int_\Omega  |\nabla v_k^m|^2 \,dxdt + \frac{1}{\alpha_m}\int_0^T \theta(t) \int_\Omega f_k(u_k^m) v_k^m \,dxdt\nonumber\\
			&&+ \frac{1}{\alpha_m}\int_0^T \theta(t) \int_\Omega a(x) g(\partial_t u_k^m)  v_k^m\,dxdt=0.\nonumber
			\end{eqnarray}
			
			Considering the convergences (\ref{damping conv}), (\ref{conv1}), (\ref{conv2}), (\ref{conv3}) and (\ref{kinectic convergence}) and having in mind that $v_k=0,$ from (\ref{equipartition}), we deduce that
			\begin{equation*}
			\lim_{m \rightarrow +\infty}\int_\varepsilon^{T-\varepsilon} \int_\Omega ( |\nabla v_k^m|^2 + \frac{1}{\alpha_m} f_k(u_k^m)v_k^m)\,dxdt = 0,
			\end{equation*}
			from which one also has
			\begin{eqnarray*}
				\lim_{m \rightarrow +\infty} \frac{1}{\alpha_m^2}\int_\varepsilon^{T-\varepsilon} \int_\Omega F_k(u_k^m) \,dxdt=0,
			\end{eqnarray*}
			which implies jointly with all convergences above that $\int_\varepsilon^{T-\varepsilon} E_{v_k^m}(t) \rightarrow 0$. Then, by the decrease of the energy, we obtain
			\begin{equation*}
			(T- 2\varepsilon) E_{v_k^m}(T-  \varepsilon) \rightarrow 0, \hbox{ as } m\rightarrow +\infty.
			\end{equation*}
			
			This implies, together with the energy identity
			\begin{equation*}
			E_{v_k^m}(T- \varepsilon) - E_{v_k^m}(\varepsilon) = - \frac{1}{\alpha_m}\int_{\varepsilon}^{T- \varepsilon} \int_{\Omega} a(x)\,g(\partial_t u_k^m) \partial_t v_k^m\,dxdt
			\end{equation*}
			and (\ref{damping conv}), that $E_{v_k^m}(0) \rightarrow 0$ as $m\rightarrow +\infty$ for $1 < p< \frac{n+2}{n-2}$, as we aimed to prove. \end{proof}

In order to prove the last statement of the lemma, we assume the contrary.  	
{
	
Our strategy has the following steps:
\begin{itemize}
	\item[(i)] We will first observe that if observability constants cannot be uniformly bounded, then for a subsequence of approximate solutions, the limit of $\partial_t u_k$ must vanish on $\omega$ as $k\rightarrow \infty$.
	\item[(ii)] The limit of approximate solutions found in the first step will imply that the energy of the original solution must be constant in time.
	\item[(iii)] We will use the fact that approximate solutions satisfy the observability inequality to deduce that their energy must decay in time in contrast with the original solution.  This will allow us to construct an increasing countable sequence of times at which the energy of approximate solutions are much smaller than the constant energy of the original solution.
	\item[(iv)] The strong convergence results obtained in the wellposedness section will allow us to extract a countable sequence of approximate solution-time pairs through Cantor's diagonalization so that the values of each element will be close to energy of original solution for large indices.
	\item[(v)] Finally, the last two steps will be combined to get a contradiction.
\end{itemize}

So, suppose $(u_0,u_1)$ are such that $r<E(u_0,u_1)<R$ and $(u_{0,k},u_{1,k})$ is a sequence of data strongly converging to $(u_0,u_1)$ verifying also \eqref{sequencebound}. Then, the observability inequality that was proved in the first part of the lemma applies for each $k$.  In order to prove the last part of the lemma, assume to the contrary that 
		\begin{equation}\label{supisinf}
			\sup_k  \frac{E_k(u_{0,k},u_{1,k})}{ \int_{0}^{T}\int_\Omega  a(x) \left(|\partial_t u_k|^2 + |g(\partial_t u_k)|^2\right) \,dxdt}=\infty.
			\end{equation} 
		{
			First, observe that in the contradiction argument we can assume \eqref{supisinf} for all $T\ge T_0$. Indeed, if \eqref{supisinf} were not true for some $T=T^*\ge T_0$, then for any $T'>T^*$,  \eqref{supisinf} would be false with $T=T'$ since the denominator gets larger as $T$ gets larger.    This would imply that the desired result (i.e., existence of a uniform bound) would readily hold true for all $T\ge T^*$, which would be sufficient for the purposes of this paper.
			
		Now, since the numerator in \eqref{supisinf}  is bounded from below and above, \eqref{supisinf} can only be possible if there is a subsequence still denoted with the index $k$ such that
		\begin{equation}\label{righthandside-k}
			 \lim_k\int_{0}^{T}\int_\Omega  a(x) \left(|\partial_t u_k|^2 + |g(\partial_t u_k)|^2\right) \,dxdt=0.
			\end{equation}
for all $T \geq T_0$. Now, let $T\ge T_0$ be a fixed time.  We know then there is a subsequence $\{u_{{k}_{1j}}\}$ which convergences to some $u^1$ in the sense of \eqref{conv1'}-\eqref{conv3'} on $[0,T]$. Since we are assuming \eqref{righthandside-k} for any $T\ge T_0$, it is in particular true if $T$ is replaced with $2T$. Since $\{u_{{k}_{1j}}\}$ is a subsequence of $\{u_{k}\}$, there exists also a subsequence $\{u_{{k}_{2j}}\}$ of $\{u_{{k}_{1j}}\}$ which convergences to some $u^2$ in the sense of \eqref{conv1'}-\eqref{conv3'} on $[0,2T]$ and moreover one has $u^2=u^1$ on $[0,T]$. Inductively, we obtain a subsequence $\{u_{{k}_{mj}}\}$ which converges to some $u^m$ on $[0,mT]$ in the sense of \eqref{conv1'}-\eqref{conv3'}, and $u^m=u^{m-1}$ on $[0,(m-1)T]$. 

 Considering the diagonal sequence $\{u_{{k}_{mm}}\}_{m \in \mathbb{N}}$ we see that this subsequence converges to $u$, solution of the original problem with data $(u_0,u_1)$ over compact subsets of $[0,\infty)$  satisfying $u(t)=u^{m}(t)$ on  $[0,mT]$. This allows us to conclude that
	 \begin{equation}\label{righthandside}
		\int_{0}^{T}\int_\Omega  a(x) \left(|\partial_t u|^2 + |g(\partial_t u)|^2\right) \,dxdt=0,
	\end{equation} 
for all $T \geq T_0$.
	Therefore, $u$ must solve the following problem:
	\begin{equation}\label{limit1'aaa}
		\begin{cases}
			\partial_t^2 u -\Delta u + f(u)= 0 & \hbox{ in } \Omega \times (0, \infty),\\
			u=0& \hbox{ on } \partial \Omega \times (0, \infty),\\
			u(x,0)=u_0(x), u_t(x,0)=u_1(x) & \hbox{ in } \Omega.
		\end{cases}
	\end{equation} 
}

We see from \eqref{limit1'aaa} that the energy of $u$ is constant and we have \begin{equation}E_u(t)=E(u_0,u_1)>0, \quad t\ge 0.\end{equation}  Now, define $\delta:=E_u(t)=E(u_0,u_1)$. 
Any approximate solution $u_k$ readily satisfies the same type of decay estimates (though with constants depending on $k$ of course) which are to be proven for the original solution in Section 4 below. This is because approximate solutions have observability property as proved in the first part of the lemma and they also satisfy Assumption \ref{as:regularity} since they are solutions of truncated problems with smooth data. Therefore, we have decay of truncated solutions.  Hence, for any fixed $k$, there corresponds sufficiently large time - say $t_k$ for which $E_{u_k}(t_k)$ is very small (say less than $\delta/4$). Moreover, we can choose $t_k$'s in such a way that $t_k \uparrow\infty$ (i.e., increasing unboundedly) as $k\rightarrow \infty$. For such $t_k$ we have
\begin{equation}
	E_{u_k}(t_k)\le \frac{\delta}{4}, k\ge 1.
\end{equation}

On the other hand, we know from the well-posedness analysis in the previous section that for fixed $t_0$, we have $E_{u_k}(t_0)\rightarrow E_u(t_0)=\delta$ as $k\rightarrow \infty$. In particular, for each fixed $t=t_m$, where $t_m$ is the sequence of times we just constructed above, we have
 $E_{u_k}(t_m)\rightarrow \delta$ as $k\rightarrow \infty$. We can consider $E_{u_k}(t_m)$ as a sequence of numbers indexed with $k$ for each fixed $m$. Note that each such sequence has the same limit $\delta$ as $k\rightarrow \infty$. Therefore, we can apply Cantor's diagonalization and deduce that the diaogonal subsequence of numbers denoted $E_{u_k}(t_k)$ also converges to $\delta$. Hence, we have both $E_{u_k}(t_k)\le \frac{\delta}{4}$, $\forall k\ge 1$ and 
 $E_{u_k}(t_k)\rightarrow \delta$ as $k\rightarrow \infty$. These two conditions yield a contradiction together.
}
Hence, one must have
$$	\sup_k  \frac{E_k(u_{0,k},u_{1,k})}{ \int_{0}^{T}\int_\Omega  a(x) \left(|\partial_t u_k|^2 + |g(\partial_t u_k)|^2\right) \,dxdt}=:C<\infty.$$
Therefore, we conclude that if $(u_0,u_1)$ in $\mathcal{H}$ with $r<E(u_0,u_1)<R$ such that $( u_{0,k}, u_{1,k})\in \mathscr{D}(\mathbb{A})$ and $( u_{0,k}, u_{1,k})\rightarrow (u_0,u_1)$ in $\mathcal{H}$ verifying also \eqref{sequencebound} , then the constant of inequality \eqref{obs ineq} can be chosen independent of $k$.

{
		
	\begin{Remark}
	In what follows, we will abuse notation and write $T_0$ in the sense of $T^*$ found in the second part of Lemma \ref{ObsApprox}.
	\end{Remark}
}

		\subsection{Observability for the original problem}
		
		The main result of this section reads as follows:
		
		\begin{Lemma}\label{obs}  Assume that $f$ satisfies the conditions given in Assumption \ref{assmpflabel} with $n=3$ and $\mathcal{O}(g)=1$. Then, for all $T\geq T_0$ and $r,R>0$, there exists a constant $C=C(r,R,T)>0$ such that for any $(u_0,u_1)\in \mathcal{H}$ with $r<E(u_0,u_1)<R$, the solution of \eqref{eq:*} satisfies
			\begin{eqnarray}\label{obs ineq1}
			E(u_0,u_1) \leq C \int_{0}^{T}\int_\Omega  a(x) \left(|\partial_t u|^2 + |g(\partial_t u)|^2\right) \,dxdt.
			\end{eqnarray}
		\end{Lemma}
		\begin{proof} We can take a sequence of initial data $( u_{0,k}, u_{1,k})\in \mathscr{D}(\mathbb{A})$ converging strongly to $(u_0,u_1)$ in $\mathcal{H}$.
			Then for some $k_0\ge 1$ and $k \geq k_0\in \mathbb{N}$, (\ref{obs ineq}), will hold and moreover due to Lemma \ref{ObsApprox}
			\begin{equation}
			\label{final_inequality}
			E(u_{0,k}, u_{1,k}) \leq C\,\int_0^{T}\int_{\Omega}a(x)\left(  |\partial_t u_k|^2 +  |g(\partial_t u_k)|^2 \right)\,dx\,dt,
			\end{equation}
			 where $C$ is a positive constant which does not depend on $k$.
			
			Taking (\ref{Cauchy conv}) and Lebesgue's Dominated Convergence Theorem into account, we deduce that
			\begin{equation}\label{convag2}
			\int_{0}^{T}\int_{\Omega} a(x)|g(\partial_t u_k(x,t))|^2 \, dx \,dt \rightarrow \int_{0}^{T}\int_{\Omega} a(x)|g(\partial_t u(x,t))|^2 \, dx \,dt \hbox{ as } k \rightarrow \infty.
			\end{equation}
			From (\ref{main strong conv}), \eqref{final_inequality} and \eqref{convag2}, we also obtain the observability inequality associated with the original problem (\ref{eq:*}), that is,
			\begin{equation}
			\label{obser ineq original prob}
			E(u_0,u_1) \leq C\,\int_0^{T}\int_{\Omega} a(x)\left( |\partial_t u|^2  + |g(\partial_t u)|^2 \right)\,dx\,dt, \hbox{ for all } T\geq T_0.
			\end{equation}
		\end{proof}

		\section{Combining estimates at the origin and infinity}\label{section6}

		{In order to calculate explicit decay rates for the solutions of the system \eqref{mainproblem}, and for reader's convenience, we enunciate and give the proof of some nice results already existing in the literature, which extend some results introduced by I. Lasiecka, D. Tataru and D. Toundykov in \cite{Lasiecka-Tataru} and \cite{las-tou:06} and have already been used in \cite{AMO}, \cite{TAMS} and \cite{Moez}, \cite{Toundykov2} and \cite{Toundykov}.} As is well known, sublinear and superlinear at infinity feedback maps require more regularity of solutions. Thus, the regularity below is only needed when $\mathcal{O}(g) \neq 1$.
		
		\begin{assum}[Regularity for sub and superlinear feedbacks at infinity {(See Assumption 5.1 in \cite{AMO})}] \label{as:regularity}
			This assumption  is imposed only when $g$ is not linearly bounded at infinity:
			\begin{itemize}
				\item    If $g$ is sublinear at infinity, i.e. $\mathcal{O}(g)<1$, then assume that there exists $p_0>2$ such that $$\|u_t\|_{L^\infty(\mathbb{R}_+;L^{p_0}(\Omega))}<D_0.$$
				
				\item    If $g$ is superlinear at infinity, i.e. $\mathcal{O}(g)>1$, then assume that there exists $p_0>2r$ such that $$\|u_t\|_{L^\infty(\mathbb{R}_+;L^{p_0}(\Omega))}<D_0$$
			\end{itemize}
			for some positive constant $D_0$.
		\end{assum}
		\begin{Remark}
			Note that since the system is monotone dissipative,  the regularity Assumption \ref{as:regularity} can be satisfied to a certain extent by starting with smooth initial data. Thus, if a solution is \textbf{strong}, then, for all $T>0$, $\|u_{t}(t)\|_{H_{0}^{1}(\Omega)} \leq C$ for all $t \in (0,T)$ and some constant $C>0$ that does not depend on $T>0$. Hence, $u_{t}\in L^{\infty}(\mathbb{R}_{+};L^{p_0}(\Omega))$ for any $p_0<\frac{2n}{n-2}$ if $n \geq 3$ and $u_{t}\in L^{\infty}(\mathbb{R}_{+};L^{p_0}(\Omega))$ for any $p_0<\infty$ if $n \leq 2$.
		\end{Remark}
		
	{Following \cite{dao}, let} $h_0$ be defined as in \eqref{ass on h} and set
		\begin{equation}\label{h}
		h=h_1+h_0 \circ \frac{1}{\operatorname{meas}(Q_T)},
		\end{equation}
		where $h_1$ is defined as follows:
		\begin{equation}\label{h1}
		\begin{aligned}
		&\bullet \hbox{ if } g \hbox{ is sublinear at infinity, i.e., } r=\mathcal{O}(g)<1, \hbox{ then let } h_1(s):=s^{\frac{p_0-2}{p_0-r-1}},\\
		&\bullet \hbox{ if } g \hbox{ is superlinear at infinity, i.e., } r=\mathcal{O}(g)>1, \hbox{ then let } h_1(s):=s^{\frac{p_0-2r}{p_0-r-1}},\\
		&\bullet \hbox{ if } g \hbox{ is linear at infinity, i.e., } r=\mathcal{O}(g)=1, \hbox{ then let } h_1(s):=s
		\end{aligned}
		\end{equation}
		{and let (as in equation (11) of \cite{las-tou:06})}
		\begin{equation}\label{q}
		q=I-(I+(I+h)^{-1}\circ (K^{-1} \cdot I) )^{-1}
		\end{equation}
		for some constant $K>0$ to be specified later. Since $h$ is a strictly increasing concave function with $h(0)=0$, $q$ is a monotone increasing function vanishing at zero.
		
{
	
		\begin{Theorem}\label{theo 4}
			Suppose Assumption \ref{as:regularity} holds, $\mathcal{O}(g)\neq 1$ and $1 \leq p <\frac{n}{n-2}$. Let $T\geq T_0$, $r,R>0$, then there is a constant $C=C(r,R,T)>0$ such that for any strong solution of \eqref{eq:*} (see item iii) of Theorem \ref{thm:well-posedness}) with initial data satisfying $r<E(u_0,u_1)<R$, the following inequality holds:
			\begin{equation}\label{i:mixed}
			E(u_0,u_1) \leq C \int_{0}^{T}\int_\Omega  a(x) \left(|\partial_t u|^2 + |g(\partial_t u)|^2\right) \,dxdt.
			\end{equation}	
		\end{Theorem}
}
		
		\begin{Remark}
		Note that in Theorem \ref{theo 4}, the truncation is not necessary, since the non-linearity grows according to Sobolev's immersion.
		\end{Remark}
		
		{As in equation (6.7) of \cite{AMO}}, henceforth we will also use the notation
		\begin{equation}\label{def:damping}
		{\bf D}_{a}^{b}\big(g(s); u_{t}\big) :=
		\int_{a}^{b}\int_{\Omega}a(x)g(u_{t}) u_{t} \, dQ_T.
		\end{equation}

		{The main result of the literature, which allows the explicit derivation of decay rates, is the following:}

		\begin{Corollary}[{Theorem 2.1  in \cite{dao}, Lemma 8.2 in \cite{AMO}, Corollary 3.1.1 in \cite{TAMS}}]\label{uniformdecay}
			Suppose Assumption \ref{as:regularity} holds. Let $T_0>0$ be given by \eqref{obser ineq original prob} or Theorem \ref{theo 4} . Let $h$ and $q$ be the functions defined in $\eqref{h}$ and $\eqref{q}$, respectively. Then,
			\begin{equation}\label{decay}
			E_u(t) \leq S\left(\frac{t}{T_0}-1\right) \hbox{ for all } t \geq T_0,
			\end{equation}
			with $\lim_{t \rightarrow \infty} S(t)=0$, where $S$ is the solution to the following nonlinear ODE:
			\begin{equation}\label{ODE}
			S_t+q(S)=0, \quad S(0)=E_u(0),
			\end{equation}
			where
			\begin{equation}\label{Krequal}
			K=\left[C(T,a,\operatorname{meas}(Q_T))\right]
			\end{equation}
			if $g$ is linearly bounded near infinity,
			\begin{equation}\label{Krbig}
			K=\left[C(T,a,\operatorname{meas}(Q_T)) \cdot D_{0}^{\frac{p_0(r-1)}{p_0-r-1}} \right]
			\end{equation}
			if $g$ is superlinear near infinity, and
			\begin{equation}\label{Krsmall}
			K=\left[C(T,a,\operatorname{meas}(Q_T)) \cdot D_{0}^{\frac{p_0(1-\theta)}{p_0-r-1}} \right]
			\end{equation}
			if $g$ is sublinear near infinity.
		\end{Corollary}
		\begin{proof}
			In what follows, we proceed as in Lasiecka and
			Tataru's  work \cite{Lasiecka-Tataru} (see Lemma 3.2 and Lemma 3.3
			of the referred paper) adapted to our context. Let
			\begin{eqnarray*}
				\Sigma _{\alpha } &=&\left\{ \left( x,t\right) \in Q_T:=\Omega \times (0,T) : %
				\left\vert \partial_t u(x,t)\right\vert >1\text{  a.e. }\right\} , \\
				\Sigma _{\beta } &=&Q_T \backslash \Sigma _{\alpha }.
			\end{eqnarray*}
			
			The proof will require the following inequalities:
			
			\begin{enumerate}
				\item [I)]  \textbf{Damping near the origin}. From (\ref{ass on h}) and the fact that $h$ is concave
				and increasing, $a(x)\leq \|a\|_{\infty}+1,$
				and $\frac{a(x)}{1+\|a\|_{\infty}} < a(x),$ we deduce that
				\begin{eqnarray} \label{5.4.2}\qquad
				\int_{\Sigma _{\beta }}a(x)\left( \left[ g\left(\partial_t u\right)
				\right] ^{2}+\left(\partial_t  u\right) ^{2}\right) d\Sigma _{\beta }
				&\leq& \int_{\Sigma _{\beta }}a(x)h_0\left( g\left(\partial_t u\right)
				\partial_t u\right) d\Sigma_{\beta }\\
				&=& \int_{\Sigma _{\beta
				}}(1+\|a\|_{\infty})\frac{a(x)}{1+\|a\|_{\infty}}h_0\left( g\left(
				\partial_t u\right) \partial_t u\right) d\Sigma_{\beta }\nonumber\\
				&\leq& \int_{\Sigma _{\beta
				}}(1+\|a\|_{\infty})h_0\left(\frac{a(x)}{1+\|a\|_{\infty}} g\left(
				\partial_t u\right) \partial_t u\right) d\Sigma_{\beta }\nonumber\\
				&\leq& \int_{\Sigma _{\beta }}(1+\|a\|_{\infty})h_0\left(a(x)
				g\left( \partial_t u\right) \partial_t u\right) d\Sigma_{\beta }.\nonumber
				\end{eqnarray}
				
				Then, by Jensen's inequality,
				\begin{eqnarray}
				&&(1+\|a\|_{\infty})\int_{\Sigma _{\beta }}h_0\left(a(x) g\left(
				\partial_t u\right) \partial_t u\right) d\Sigma _{\beta }\label{5.4.3}\\
				&&\leq (1+\|a\|_{\infty})\operatorname{meas}\left( Q_T\right) h_0\left(
				\frac{1}{\operatorname{meas}\left( Q_T\right) }\int_{Q_T}a(x)g\left(\partial_t  u\right)
				\partial_t u\,dQ
				\right) \smallskip  \notag \\
				&=&(1+\|a\|_{\infty})\operatorname{meas}\left( Q_T\right) r\left(
				\int_{Q_T}a(x)g\left( \partial_t u\right) \partial_t u\, dQ\right) \nonumber,
				\end{eqnarray}
				where $r\left( s\right) =h_0\left( \frac{s}{\operatorname{meas}\left(Q_T \right) }
				\right)$.
				\begin{equation}\label{i:origin}
				\int_{\Sigma _{\beta }}a(x)\left( \left[ g\left(\partial_t u\right)
				\right] ^{2}+\left(\partial_t  u\right) ^{2}\right) d\Sigma _{\beta } \leq 	
				(1+\|a\|_{\infty})\operatorname{meas}\left( Q_T\right) r\left(
				\int_{Q_T}a(x)g\left( \partial_t u\right) \partial_t u\, dQ\right).
				\end{equation}
				
				\item[II)] \textbf{Linearly-bounded damping at infinity}. If $\mathcal{O}(g) =1$ according to Definition \eqref{def:order}, then, from hypothesis (\ref{ass on g}), we obtain
				\begin{equation}
				\int_{\Sigma _{\alpha }}a(x)\left( |g\left(\partial_t  u\right)
				|^{2}+|\partial_t  u|^{2}\right) d\Sigma _{\alpha
				}\leq \left( \tilde{m}^{-1}+ \tilde{M}\right) \int_{\Sigma _{\alpha }}a(x)g\left(
				\partial_t u\right) \partial_t u \,d\Sigma _{\alpha }.  \label{i:lin}
				\end{equation}%
				
				\item[III)] \textbf{Superlinear damping at infinity}.
				Suppose  $\mathcal{O}(g) = r> 1$. From \eqref{ass on g} we obtain $|g(s)| > \tilde{m} |s|$ for $|s|>1$ with $\tilde{m}>0$ independent of $s$, and we trivially estimate
				\begin{equation}\label{i:superlin:1}
				\int_{\Sigma_{\alpha }} a(x)|u_t|^2 \, d\Sigma_{\alpha } \leq \frac{1}{\tilde{m}^2}\int_{\Sigma_{\alpha }} a(x)|g(u_t)|^2 \, d\Sigma_{\alpha }.
				\end{equation}
				Next, for $\lambda \in ]0,1[,$
				\begin{equation}\label{i:superlin:2}
				\int_{\Sigma_{\alpha }} a(x)|g(u_t)|^2 \, d\Sigma_{\alpha }  = \int_{\Sigma_{\alpha }} a(x)|g(u_t)|^{2 \lambda}|g(u_t)|^{2(1-\lambda)} \, d\Sigma_{\alpha }=: J_1.\\
				\end{equation}
				Choose any $p_0>2r$ and estimate $J_1$ using H\"older's inequality with conjugate exponents $\frac{p_0}{2\lambda r}$ and $\frac{p_0}{p_0-2\lambda r}$ (splitting $a(x)$ as $a(x)^{2\lambda r/{p_0}} \cdot a(x)^{(p_0-2\lambda r)/{p_0}}$):
				\begin{equation}\label{i:superlin:3}
				J_1 \leq \left(\int_{\Sigma_{\alpha }} a(x)|g(u_t)|^{p_0/r} \, d\Sigma_{\alpha } \right)^{2\lambda r /{p_0}}\left(\int_{\Sigma_{\alpha }}
				a(x)|g(u_t)|^{\frac{2 (1-\lambda)p_0}{p_0-2\lambda r}} \, d \Sigma_{\alpha }\right)^{\frac{p_0-2\lambda r}{p_0}}.
				\end{equation}
				Note that $\mathcal{O}(g)=r$ implies
				\begin{equation}\label{g-sim-superlin}
				g(s)s \sim s^{r+1} \sim g(s)^{(r+1)/r},\quad  |s|>1.
				\end{equation}
				Thus,  for  $|g(u_t)|^{\frac{2 (1-\lambda)p_0}{p_0-2\lambda r}}$ to be equivalent to the dissipation integrand $g(u_t)u_t,$ we solve
				\[
				\frac{2(1-\lambda)p_0}{p_0-2\lambda r } = \frac{1+ r}{r}    \implies  \lambda = \frac{p_0(r-1)}{2r(p_0-r-1)}.
				\]
				With this choice of $\lambda,$ we combine \eqref{i:superlin:1}, \eqref{i:superlin:2} and \eqref{i:superlin:3} to  conclude
				\begin{equation}\label{i:superlin:final}
				\int_{\Sigma_{\alpha }} a(x) (|u_t|^2+|g(u_t)|^2 ) \, d \Sigma_{\alpha } \lesssim C_a\|u_t\|_{L^\infty(\mathbb{R}_+;L^{p_0}(\Omega) )}^{\frac{p_0(r-1)}{p_0-r-1}}\left(\int_{Q_T} a(x)g(u_t)u_t \, dQ_T \right)^{\frac{p_0-2\mathcal{O}(g)}{p_0-1-\mathcal{O}(g)}}
				\end{equation}
				where $C_a=(\sup a(x))^{\frac{2 \lambda r}{p_0}}$.
				
				\item[IV)] \textbf{Sublinear damping at infinity}.
				Assume  $\mathcal{O}(g) = r < 1$ in view of Definition \ref{def:order}. From \eqref{ass on g} we obtain  $\tilde{M}|s| >  |g(s)|$ for $|s|>1$ with $\tilde{M} >0$ independent of $s$:
				\begin{equation}\label{i:sublin:1}
				\int_{\Sigma_{\alpha }} a(x)|g(u_t)|^2 \, d \Sigma_{\alpha } \leq \tilde{M} \int_{\Sigma_{\alpha }} a(x)|g(u_t)||u_t| \, d \Sigma_{\alpha } =\tilde{M} \int_{\Sigma_{\alpha }} a(x)g(u_t)u_t \, d \Sigma_{\alpha }.
				\end{equation}
				For $\lambda \in ]0,1[,$
				\begin{equation}\label{i:sublin:2}
				\int_{\Sigma_{\alpha }} a(x)|u_t| \, d \Sigma_{\alpha } = \int_{\Sigma_{\alpha }} a(x) |u_t|^{2\lambda} |u_t|^{2(1-\lambda)} \, d \Sigma_{\alpha }=:J_2.
				\end{equation}
				Let $p_0>2$ and estimate $J_2$ using H\"older's inequality with exponents   $\frac{p_0}{2\lambda}$ and $\frac{p_0}{p_0-2\lambda}$ (splitting $a(x)$ as $a(x)^{2\lambda/p_0} \cdot a(x)^{(p_0-2 \lambda)/p_0}$):
				\begin{equation}\label{i:sublin:3}
				J_2 \leq \left(\int_{\Sigma_{\alpha }} a(x)|u_t|^{p_0} \, d \Sigma_{\alpha } \right)^{2\lambda/{p_0}}\left(\int_{\Sigma_{\alpha }}a(x)
				|u_t|^{\frac{2 (1-\lambda)p_0}{p_0-2\lambda}} \, d \Sigma_{\alpha }\right)^{\frac{p_0-2\lambda}{p_0}}
				\end{equation}
				The value of $\lambda\in]0,1[$ is chosen so that
				\begin{equation}\label{g-sim-sublin}
				|u_t|^{\frac{2 (1-\lambda)p_0}{p_0-2\lambda}}  = u_t^{r+1} \sim g(u_t)u_t \quad\text{for}\quad |u_t|>1,
				\end{equation}
				namely,
				\[
				\frac{2(1-\lambda)p_0}{p_0-2\lambda} = 1+ r    \implies  \lambda = \frac{p_0(1-r)}{2(p_0-1-r)}.
				\]
				Combining \eqref{i:sublin:1}-\eqref{i:sublin:3}:
				\begin{equation}\label{i:sublin:final}
				\int_{\Sigma_{\alpha }} a(x)(|u_t|^2+|g(u_t)|^2) \, d \Sigma_{\alpha } \lesssim C_a \|u_t\|_{L^\infty(\mathbb{R}_+;L^{p_0}(\Omega))}^{\frac{p_0(1-r)}{p_0-r-1}}\left(\int_{Q_T} a(x)g(u_t)u_t \, dQ_T \right)^{\frac{p_0-2}{p_0-1-\mathcal{O}(g)}}
				\end{equation}
				where $C_a=(\sup a(x))^{\frac{2 \lambda}{p_0}}$.
			\end{enumerate}
			We combine the above estimates, observability inequality \eqref{obser ineq original prob}, and \eqref{i:mixed} from Theorem \ref{theo 4} with  inequality \eqref{i:origin} and with either \eqref{i:lin}, \eqref{i:superlin:final}, or \eqref{i:sublin:final}, depending on whether  $\mathcal{O}(g)=1$, $\mathcal{O}(g)>1$, or $\mathcal{O}(g)<1$, respectively. Using definition  \eqref{def:damping} and relabeling constants, we obtain
			\begin{equation*}
			E_u(0) \leq K \left(I({\bf D}_{0}^{T}\big(g(s); u_{t}\big))+h_0\left( \frac{1}{\operatorname{meas}(Q_T)} {\bf D}_{0}^{T}\big(g(s); u_{t}\big) \right)+h_1({\bf D}_{0}^{T}\big(g(s); u_{t}\big)) \right),
			\end{equation*}
			where $h_1$ is given by \eqref{h1} and $K=C(T,a,\operatorname{meas}(Q_T))$ if $g$ is linearly bounded near infinity, $K=C(T,a,\operatorname{meas}(Q_T)) \cdot D_{0}^{\frac{p_0(r-1)}{p_0-r-1}}$
			if $g$ is superlinear near infinity, and $K=C(T,a,\operatorname{meas}(Q_T)) \cdot D_{0}^{\frac{p_0(1-\theta)}{p_0-r-1}}$ if $g$ is sublinear near infinity.
			
			In all cases there exists a constant $K>0$ and a function $h$ such that
			\begin{equation*}
			E_u(0) \leq K (I+h)({\bf D}_{0}^{T}\big(g(s); u_{t}\big) ).
			\end{equation*}
			By using the energy identity, we obtain
			\begin{equation*}
			(I+h)^{-1}(K^{-1} E_u(T)) \leq E_u(0)-E_u(T).
			\end{equation*}

			To finish the proof of the decay, we invoke the
			following result due to I. Lasiecka and Tataru
			\cite{Lasiecka-Tataru}:
			
			\noindent \textbf{Lemma A}:\textit{\ Let }$p$\textit{\ be a
				positive, increasing function such that }$p(0)=0$\textit{. Since
			}$p$\textit{\ is
				increasing, we can define an increasing function }$q,$ $q(x)=x-(I+p)^{-1}%
			\left( x\right) .$\textit{\ Consider a sequence }$s_{m}$\textit{\
				of
				positive numbers that satisfies}%
			\begin{equation*}
			s_{m+1}+p(s_{m+1})\leq s_{m}.
			\end{equation*}%
			\textit{Then }$s_{m}\leq S(m)$\textit{, where }$S(t)$\textit{\ is
				a solution
				of the differential equation}%
			\begin{equation*}
			\frac{d}{dt}S(t)+q(S(t))=0,\text{ \ }S(0)=s_{0}.
			\end{equation*}%
			\textit{Moreover, if }$p(x)>0$\textit{\ for }$x>0$\textit{, then }$\underset{%
				t\rightarrow \infty }{\lim }$\textit{\ }$S(t)=0.$
			
			The semigroup property of $U(t)=(u(t),u_t(t))$ yields
			\begin{equation}\label{decayl}
			E_u((m+1)T)+(I+h)^{-1}(K^{-1} E_u((m+1)T)) \leq E_u(mT)
			\end{equation}
			for all $m \in \mathbb{N}$.
			
			Applying Lemma A with $s_{m}=E_{u}(mT)$ thus results in%
			\begin{equation}
			E_{u}(mT)\leq S(m),\text{ \ \ }m=0,1,....  \label{3.35'}
			\end{equation}
			
			Finally, using the dissipativity of $E_{u}(t)$ , we have for $t=mT+\tau ,$ $0\leq \tau \leq T,$%
			\begin{equation}
			E_{u}(t)\leq E_{u}(mT)\leq S(m)= S\left( \frac{t-\tau }{T}\right) \leq
			S\left( \frac{t}{T}-1\right) \text{ \ \ for \ }t>T_0\text{.}
			\label{3.36'}
			\end{equation}%
			Above, we use the fact that $S(\cdot)$ is dissipative. The
			proof of decay for $E_{u}$ is now complete.
		\end{proof}

		The algorithm for computations of decay rates given by Corollary \ref{uniformdecay} is very general and
		provides explicit decay rates without any restrictions on the growth of the dissipation $g$ at the
		origin. Indeed, as shown in \cite{Lasiecka-Tataru}, this algorithm gives exponential decay rates for the damping
		that is bounded from below by a linear function and algebraic decay rates for polynomially
		decaying dissipation at the origin. We shall illustrate below how other cases can be treated as
		well. By specializing a bit further the class of nonlinear dissipation, we will be able to obtain
		an explicit description of the decay rates. The obtained decay rates are optimal, since they are the
		same as these optimal rates derived in \cite{alabau:05} for the model that does not account for the sources. In
		addition, we will be able to obtain decay rates for nondifferentiable dissipation, such as fractional
		powers.
		
		{Next, we present the result that allows the derivation of the explicit decay rates. }
		
		\begin{Corollary}[{Lasiecka-Toundykov, Corollary 1 in \cite{las-tou:06}}]\label{uniformdecay2}
			Suppose that the assumptions of Corollary \ref{uniformdecay} are satisfied and that the function $h$ in \eqref{h} can be expressed as $h=h_b+h_s$ where $h_b$ and  $h_s$ are concave, monotone increasing, zero at the origin, and, in addition, $h_s=o(h_b).$ The latter means that $\lim_{x \rightarrow 0^+} h_s(x)/h_b(x)=0$ and $h_b$ has no upper linear bound on $[0,1)$. Then, given any positive $\gamma<1$, there exists $t_0=t_0(\gamma) \geq T$ such that the following energy estimate holds:\begin{equation}
			E_u(t) \leq \bar{S}\left(\frac{t}{T_0}-1\right), \hbox{ for all } t \geq 2t_0,
			\end{equation}
			where $\bar{S}$ is the solution of the following nonlinear ODE:
			\begin{equation}\label{approximate}
			\bar{S}_t+h_b^{-1}\left( \gamma K^{-1} \bar{S}\right) =0, \quad \bar{S}(0)=E_u(0),
			\end{equation}
			with $K$ and $T_0$ as in Corollary \ref{uniformdecay}.
		\end{Corollary}
		\begin{proof}
			Assumptions on $h$ imply that $q \geq h_b^{-1} \circ (\gamma K^{-1} I)$ near the origin. Indeed, writing $\mathcal{F}=K(I+h)$ we obtain
			\begin{equation}
			\begin{aligned}
			q= {} &I-(I+\mathcal{F}^{-1})^{-1} \\
			= {} &\left[(I+\mathcal{F}^{-1}) \circ (I+\mathcal{F}^{-1})^{-1} \right]-(I+\mathcal{F}^{-1})^{-1}\\
			= {} & \mathcal{F}^{-1} \circ (I+\mathcal{F}^{-1})^{-1}\\
			= {} & \mathcal{F}^{-1} \circ (\mathcal{F} \circ \mathcal{F}^{-1}+\mathcal{F}^{-1})^{-1}\\
			= {} & \mathcal{F}^{-1} \circ \left[(K(I+h)) \circ \mathcal{F}^{-1} +\mathcal{F}^{-1} \right]^{-1}\\
			= {} & \mathcal{F}^{-1} \circ [K\mathcal{F}^{-1}+K h\circ \mathcal{F}^{-1}+\mathcal{F}^{-1} ]^{-1}\\
			= {} & \mathcal{F}^{-1} \left[(Kh+(K+1)I) \circ \mathcal{F}^{-1} \right]^{-1}\\
			= {} & \mathcal{F}^{-1} \circ \mathcal{F} \circ (Kh+(K+1)I)^{-1}\\
			= {}& (Kh_b+Kh_s + (K+1)I)^{-1}.
			\end{aligned}
			\end{equation}
			Note that $Kh_s + (K+1)I=o(h_b)$ since $\lim_{x \rightarrow 0^{+}} \frac{Kh_s(x)+(K+1)x}{h_b(x)}=0$. Thus, $q=(Kh_b+o(h_b))^{-1}$. Moreover, if $w=o(h_b)$ then, for all $\varepsilon>0$ there exists $\delta>0$ such that $\frac{o(h_b)(x)}{K h_b(x)}=\frac{w(x)}{Kh_b(x)}<\varepsilon$ whenever $x \in (0,\delta)$. Therefore, $o(h_b)(x) < (1+\varepsilon) K h_b(x)-Kh_b(x)$ for all $x \in (0,\delta)$, which implies that $(1+\varepsilon)Kh_b \geq Kh_b+o(h_b)$ in $[0,\delta)$. Consequently, if we let $\gamma=(1+\varepsilon)^{-1}$, then $h_b^{-1} \circ \gamma K^{-1} I \leq (Kh_b+o(h_b))^{-1}=q$ in $[0,\delta)$.
			
			Let $\ell: \mathbb{R}_+ \rightarrow \mathbb{R}_{+}$ be given by $\ell=(I+h^{-1}) \circ (K^{-1} I )$. From Corollary \ref{uniformdecay} we know $\displaystyle \lim_{t \rightarrow \infty} S(t)=0$. Thus, there exists $t_0\geq T_0$ such that $S(t_0)<\delta$. Define $s_m=E_u(mt_0)$. Let us show that
			\begin{equation}\label{decayapproximate}
			E_u(t) \leq \tilde{S}\left(\frac{t}{T_0}-1\right), \hbox{ for all } t \geq 2t_0.
			\end{equation}
			where
			\begin{equation}\label{node}
			\tilde{S}_t+h_b^{-1}(\gamma K^{-1}(\tilde{S}(t)))=0,\quad \tilde{S}(0)=E_u(t_0).
			\end{equation}
			To this end, it is sufficient to prove that
			\begin{equation}\label{fundamentalestimate}
			s_{m+1} \leq \tilde{S}(m), \hbox{ for all } m=0,1,2,\dots.
			\end{equation}
			Indeed, since every $t \geq 2t_0$ can be written as $t=(m+1)t_0+\tau$ with $m \geq 1$, it follows that $E_u(t) \leq E_u((m+1)t_0)\leq E_u(mt_0)\leq \tilde{S}(m+1)=\tilde{S} \left(\frac{t-\tau}{t_0}\right) \leq \tilde{S} \left(\frac{t}{t_0}-1\right)$.
			
			For $m=0$, estimate \eqref{fundamentalestimate} is clear. Assume $s_{m+1} \leq \tilde{S}(m)$. From \eqref{q} and \eqref{decayl} we deduce
			\begin{equation}\label{desfun}
			s_{m+2}+\ell(s_{m+2}) \leq s_{m+1}, \hbox{ for all } m\geq 1.
			\end{equation}Inequality \eqref{desfun} is equivalent to
			\begin{equation*}
			(I+\ell)s_{m+2} \leq s_{m+1}
			\end{equation*}
			and since $(I+\ell)^{-1}$ is monotone increasing, $s_{m+2} \leq (I+\ell)^{-1}s_{m+1}=s_{m+1}-q(s_{m+1})$, hence
			\begin{equation}\label{desfun2}
			s_{m+2} \leq s_{m+1}-q(s_{m+1}).
			\end{equation}
			Integrating equation \eqref{node} from $m$ to $m+1$ yields
			\begin{equation*}
			\tilde{S}(m+1)-\tilde{S}(m)=-\int_{m}^{m+1}h_b^{-1}(\gamma K^{-1}(\tilde{S}(t)))\, dt.
			\end{equation*}
			Therefore,
			\begin{equation}
			\begin{aligned}
			\tilde{S}(m+1) {} &= \tilde{S}(m)-\int_{m}^{m+1}h_b^{-1}(\gamma K^{-1}(\tilde{S}(t)))\, dt\\
			{} &\geq \tilde{S}(m) -h_b^{-1}(\gamma K^{-1}(\tilde{S}(m)))\\
			{} &\geq \tilde{S}(m) -q(\tilde{S}(m))\\
			{} &= (I+\ell)^{-1}(\tilde{S}(m))\\
			{} &\geq (I+\ell)^{-1}s_{m+1}\\
			{} &= s_{m+1}-q(s_{m+1})\\
			{} &\geq s_{m+2}.
			\end{aligned}
			\end{equation}
			If we consider the equation
			\begin{equation}\label{node1}
			\bar{S}_t+h_b^{-1}(\gamma K^{-1}(\bar{S}(t)))=0, \bar{S}(0)=E_u(0),
			\end{equation}
			then $\tilde{S}\left(\frac{t}{t_0}-1\right) \leq \bar{S}\left(\frac{t}{t_0}-1\right)$ for all $t \geq 2t_0$.
			
			Indeed, note that $\tilde{S}(t)=G^{-1}(-t)$ where $G(S)=\int_{E_u(t_0)}^{S} \frac{1}{h_b^{-1}(\gamma K^{-1} \tau)} \, d \tau$ and $\bar{S}(t)=\bar{G}^{-1}(-t)$ where $\bar{G}(S)=\int_{E_u(0)}^{S} \frac{1}{h_b^{-1}(\gamma K^{-1} \tau)} \, d \tau$. Thus,
			\begin{equation}\label{decres1}
			\begin{aligned}
			\int_{E_u(t_0)}^{\tilde{S}\left(\frac{t}{t_0}-1\right)} \frac{1}{h_b^{-1}(\gamma K^{-1} \tau)} \, d \tau {} = & G\left(\tilde{S}\left(\frac{t}{t_0}-1\right)\right)=-\left(\frac{t}{t_0}-1\right).
			\end{aligned}
			\end{equation}
			Since $\frac{t}{t_0} -1\geq 0$, $\tilde{S}\left(\frac{t}{t_0}-1\right) \leq \tilde{S}(0)=E_u(t_0) \leq E_u(0)$. From \eqref{decres1} we obtain
			\begin{equation}\label{decres2}
			\frac{t}{t_0}-1=\int_{\tilde{S}\left(\frac{t}{t_0}-1\right)}^{E_u(t_0)}\frac{1}{h_b^{-1}(\gamma K^{-1} \tau)} \, d \tau \leq \int_{\tilde{S}\left(\frac{t}{t_0}-1\right)}^{E_u(0)} \frac{1}{h_b^{-1}(\gamma K^{-1} \tau)} \, d \tau.
			\end{equation}
			Inequality \eqref{decres2} yields
			\begin{equation}\label{decres3}
			\frac{t}{t_0}-1 \leq -\int_{E_u(0)}^{\tilde{S}\left(\frac{t}{t_0}-1\right)} \frac{1}{h_b^{-1}(\gamma K^{-1} \tau)}\, d \tau =-\bar{G}\left(\tilde{S}\left(\frac{t}{t_0}-1\right) \right).
			\end{equation}
			Therefore, \eqref{decres3} implies
			\begin{equation}\label{decres4}
			\tilde{S}\left(\frac{t}{t_0}-1\right) \leq \bar{G}^{-1}\left( - \left(\frac{t}{t_0}-1\right)\right)=\bar{S}\left(\frac{t}{t_0}-1\right), \hbox{ for all } t \geq 2t_0.
			\end{equation}
			Combining \eqref{decayapproximate} and \eqref{decres4}, the result follows.
		\end{proof}
		
		{
		\begin{Remark}
			We emphasize that the Corollaries \ref{uniformdecay} and \ref{uniformdecay2} hold for solutions of system \eqref{mainproblem} with subcritical source terms and if $\mathcal{O}(g)=1$.
		\end{Remark}}
		
		\begin{Corollary}[{Corollary 2 in \cite{las-tou:06}}]\label{uniformdecay3}
			Under hypothesis of Corollary \ref{uniformdecay}, suppose that damping is linearly bounded at infinity.
			\begin{itemize}
				\item [a)] If $g$ is sublinear at the origin, i.e., $m_0x^{\theta+1} \leq g(x)x \leq M_0 x^{\theta+1}$ with $\theta \in (0,1)$ for $|x|<1$, the auxiliary function $h_0$ may be defined as  $h_0(x)=cx^{\frac{2\theta}{1+\theta}}$, where $c=\frac{2(M_0^2+1)}{m_{0}^{\frac{2\theta}{1+\theta}}}$. Then ODE \eqref{ODE} can be replaced with
				\begin{equation}\label{ODE1}
				S_t+\frac{1}{c^{\frac{1+\theta}{2\theta}}}(\gamma K^{-1}S(t))^{\frac{1+\theta}{2\theta}}=0, \quad S(0)=E_u(0),
				\end{equation}
				and energy decay \eqref{decay} holds for all $t \geq 2t_0$ with sufficiently large threshold value $t_0>0$.
				
				\item[b)] If $g$ is superlinear at the origin, i.e., $m_0x^{r+1} \leq g(x)x \leq M_0 x^{r+1}$ with $r>1$ for $|x|<1$ and the function 	$x \mapsto \sqrt{x}g(\sqrt{x})$ is convex and superlinear (no lower linear bound with positive slope) on interval $[0,\delta)$, for some $\delta>0$, then ODE \eqref{ODE} can be replaced with
				\begin{equation}\label{ODE2}
				S_t+\sqrt{\gamma K^{-1}} S^{\frac{1}{2}}g[(\gamma K^{-1} S)^{\frac{1}{2}}]=0, \quad S(0)=E_u(0),
				\end{equation}
				and energy estimate \eqref{decay} holds whenever $t \geq 2t_0$ for a sufficiently large $t_0>0$.
			\end{itemize}
		\end{Corollary}
		\begin{proof}{Replace the sentence ``Corollary 1'' by ``Corollary \ref{uniformdecay2}'' in the proof of Corollary 2 in \cite{las-tou:06}.}
		\end{proof}
		
		\begin{Remark}
			By integrating the differential equation \eqref{ODE2} we obtain $S(t)=G^{-1}\left(-\frac{\tilde{C}}{2} t,E_u(0)\right)$ where $$G(y,E_u(0))=\int_{\sqrt{\tilde{C}E_u(0)}}^{\sqrt{\tilde{C}y}} \frac{1}{g(u)}du$$ with $\tilde{C}=\gamma K^{-1}$.
		\end{Remark}
		
		\section{Computing the decay rates}\label{section7}

		Examples in this section illustrate the above results.  We shall account for various configurations that include sub and superlinear damping, both at infinity and the origin. {We emphasize that the examples found here are adapted from (\cite{Cavalcanti0}, section 8) and Examples 3.2 and 3.3 in \cite{TAMS}.}
		
		\subsection{Linearly bounded case}\label{6.1}
		
		Throughout this section assume that $mx^2 \leq g(x)x \leq M x^2$ for all $|x|\geq 1$. In the examples below, the constant $T_0$ comes from Lemma \ref{obs ineq1} and the constant $K$ comes from  \eqref{Krequal}.
		
		\begin{Example}[Linearly bounded damping at the origin]\label{Example0}
			In this case one obtains exponential decay rates. Suppose $m_0x^2 \leq g(x)x \leq M_0 x^2$ for all $|x|< 1$. Then $h_0(x)=(M_0+m_{0}^{-1})x$, function $q$ in \eqref{q} can be computed explicitly, and ODE \eqref{ODE} is
			\begin{equation*}
			S_t+\tilde{C}S=0, \quad S(0)=E_u(0),
			\end{equation*}
			where $\tilde{C}=\frac{K^{-1}\operatorname{meas}(Q_T) }{\operatorname{meas}(Q_T)(2+K^{-1})+M_0+m_{0}^{-1}}$. This gives us $S(t)=E_u(0)e^{-\tilde{C}t}$ and $E_u(t) \leq E_u(0)e^{ \left(-\tilde{C} \left(\frac{t}{T_0}-1\right) \right)}$ for all $t \geq T_0$.
		\end{Example}
		
		\begin{Example}[Superlinear polynomial damping near the origin]\label{Example1}
			Suppose $m_0x^{p+1} \leq g(x)x \leq M_0 x^{p+1}$ for all $|x|<1$ and some $p>1$. Since the function $s \mapsto \sqrt{s}g(\sqrt{s})= s^{(p+1)/2}$ is convex for $p \geq 1$, By Corollary \ref{uniformdecay3} item b), the energy decay is determined by the solution of
			\begin{equation*}
			S_t+\tilde{C}S^{\frac{p+1}{2}}=0, \quad S(0)=E_u(0),
			\end{equation*}
			where $\tilde{C}=(\gamma K^{-1})^{\frac{p+1}{2}}$. This equation can be integrated directly, of course. However, for sake of illustrating the general formula, we find
			\begin{equation*}
			G(y,E_u(0))=\int_{\sqrt{E_u(0)}}^{\sqrt{y}}u^{-p} du=\frac{1}{1-p}[y^{\frac{-p+1}{2}}-E_u(0)^{\frac{-p+1}{2}}].
			\end{equation*}
			From here, $G^{-1}(t)=[E_u(0)^{\frac{-p+1}{2}}+t(1-p)]^{\frac{2}{-p+1}}$. Thus,
			\begin{equation*}
			E_u(t) \leq S\left(\frac{t}{T_0}-1\right)=G^{-1}\left(-\frac{\tilde{C}}{2}\left(\frac{t}{T_0}-1\right)\right)= \left[E_u(0)^{\frac{-p+1}{2}}+\frac{(p-1)\tilde{C}}{2} \left(\frac{t}{T_0}-1\right)\right]^{-\frac{2}{p-1}}
			\end{equation*}
			for all  $t \geq T_0$.
		\end{Example}
		
		\begin{Example}[Sublinear damping at the origin]\label{Example2}
			We consider $g(s)=s^\theta$ for all $|s|<1$ and some $\theta \in (0,1)$. By Corollary \ref{uniformdecay3} item a), the energy decay is determined by the solution of 	\begin{equation*}
			S_t+\tilde{C}S^{\frac{\theta+1}{2\theta}}=0, \quad S(0)=E_u(0),
			\end{equation*}
			where $\tilde{C}=\left(\frac{\gamma K^{-1}}{c} \right)^{\frac{1+\theta}{2\theta}}$ with $c=\frac{2(M_0^2+1)}{m_{0}^{\frac{2\theta}{1+\theta}}}$.
			In this case,
			\begin{equation*}
			G(s,E_u(0))=\frac{\sqrt{S}^{-\frac{1}{\theta}+1}}{-\frac{1}{\theta}+1}-\frac{\sqrt{E_u(0)}^{-\frac{1}{\theta}+1}}{-\frac{1}{\theta}+1}.
			\end{equation*}
			Therefore, $G^{-1}(t)=\left[ \left( -\frac{1}{\theta}+1 \right)t+ \sqrt{E_u(0)}^{-\frac{1}{\theta}+1}\right]^{\frac{2\theta}{\theta-1}} $. Thus,
			\begin{equation*}
			E_u(t) \leq \left[ \left( \frac{1-\theta}{2\theta} \right)\frac{\tilde{C}}{2}\left(\frac{t}{T_0}-1\right)+ E_u(0)^{\frac{-1+\theta}{2 \theta}} \right]^{-\frac{2\theta}{1-\theta}} \hbox{ for all } t \geq T_0.
			\end{equation*}
		\end{Example}
		
		\begin{Example}[Exponential damping at the origin]\label{Example3}
			We take $g(s)=s^3e^{-1/s^2}$ for $s$ at the origin. Since the function $h^{-1}(s)=g(\sqrt{s})\sqrt{s}=s^2e^{-1/s}$is convex on $[0,\delta)$ for some small $\delta>0$, we solve
			\begin{equation*}
			S_t+\tilde{C}^{2}S^2e^{-(\tilde{C}S)^{-1}}=S_t+\sqrt{\tilde{C} S}g[(\tilde{C} S)^{\frac{1}{2}}]=0,
			\end{equation*}
			where $\tilde{C}=\gamma K^{-1}$. In this case, $G(y,E_u(0))=-\frac{1}{2}[e^{1/(\tilde{C} y)}-e^{1/(\tilde{C} E_u(0))}]$ and $G^{-1}(t,E_u(0))=\tilde{C}^{-1}[\ln(e^{1/(\tilde{C}E_u(0))}-2t) ]^{-1}$. Hence,
			\begin{equation*}
			E_u(t) \leq \tilde{C}^{-1}\left[\ln\left(\tilde{C} \left(\frac{t}{T_0}-1 \right)+e^{1/(\tilde{C}E_u(0))} \right)\right]^{-1} \hbox{ for all } t \geq T_0.
			\end{equation*}
		\end{Example}
		
		\begin{Example}\label{Example4}
			We consider $g(s)=s|s|e^{-1/|s|}$ for $s$ near zero. Since the function $g(\sqrt{s})\sqrt{s}=s^{3/2}e^{-\frac{1}{\sqrt{s}}}$ is convex on $[0,\delta)$ for some small $\delta>0$ we are led to differential equation
			\begin{equation*}
			S_t+(\tilde{C}S)^{\frac{3}{2}}e^{-\frac{1}{\sqrt{\tilde{C}S}}}=0, \quad S(0)=E_u(0),
			\end{equation*}
			where $\tilde{C}=\gamma K^{-1}$. In this case $S(t)=\frac{1}{\tilde{C}} \ln^{-2}\left(\frac{1}{2} (\tilde{C}t+2e^{\frac{1}{\sqrt{\tilde{C}E_u(0)}}} ) \right)$. Therefore,
			\begin{equation*}
			E_u(t) \leq \frac{1}{\tilde{C}}  \frac{1}{\ln^2\left( e^{\frac{1}{\sqrt{\tilde{C}E_u(0)}}}+\frac{\tilde{C}}{2}\left(\frac{t}{T_0}-1 \right) \right)} \hbox{ for all } t \geq T_0.
			\end{equation*}
		\end{Example}
		
		\subsection{Sublinear and superlinear cases}
		
		\begin{Example}[Sublinear damping at infinity] Suppose $g$ is linearly bounded near $0$, namely that it satisfies $m_0s^2 \leq g(s)s \leq M_0 s^2$ for all $|s|< 1$, whereas at infinity $\mathcal{O}(g)<1$, which means that for a given positive $\theta<1$,
			\begin{equation*}
			ms^{\theta+1} \leq g(s)s \leq Ms^{\theta+1}, \hbox{ for all } |s|>1.
			\end{equation*}
			Then, $h_0$, as in Example \ref{Example0} is linear: $h_0(s)=(M_0+m_{0}^{-1})s$, whereas $h_1(s)=s^{\frac{p_0-2}{p_0-\theta-1}}$ as in \eqref{h1}. Assuming there is $p_0>2$ so that $u_t \in L^\infty(\mathbb{R}_+;L^{p_0}(\Omega))$ we obtain $\frac{h_0(x)}{h_1(x)} \rightarrow 0$ as $x \rightarrow 0^+$, and from Corollary \ref{uniformdecay2} the equation \eqref{approximate} reads as
			\begin{equation*}
			S_t+(\alpha S)^{\rho}=0, \quad S(0)=E_u(0),
			\end{equation*}
			where $\rho=\frac{p_0-\theta-1}{p_0-2}$ and $\alpha=\gamma K^{-1}$ for any positive $\gamma<1$.
			
			As stated in Corollary \ref{decay}, the constant $K>0$ is given by $$K=\left[C(T,a,\operatorname{meas}(Q_T)) \cdot D_{0}^{\frac{p_0(1-\theta)}{p_0-r-1}} \right].$$ The solution to this ODE is
			\begin{equation*}
			S(t)=(\alpha^\rho(\rho-1)(t+c_0))^{-\frac{1}{\rho-1}}, \quad c_0=\frac{E_u(0)^{1-\rho}}{\alpha^\rho(\rho-1) }.
			\end{equation*}
			Thus, for $t$ sufficiently large ($t\geq 2t_0$), where $t_0$ is given in Corollary \ref{uniformdecay2}, we have
			\begin{equation*}
			E_u(t) \leq \left[ E_u(0)^{1-\rho} +\alpha^\rho(\rho-1)\left(\frac{t}{t_0}-1 \right)\right]^{-\frac{p_0-2}{1-\theta}}.
			\end{equation*}
		\end{Example}
		
		\begin{Example}[Superlinear damping at infinity] Suppose again that $g$ is linearly bounded near $0$, and at infinity $\mathcal{O}(g)>1$, i.e., for a given $r>1$,
			\begin{equation*}
			ms^{r+1} \leq g(s)s \leq Ms^{r+1}, \hbox{ for all } |s|>1.
			\end{equation*}
			As in the previous example $h_0(x)=(M_0+m_{0}^{-1})s$, whereas $h_1(s)=s^{\frac{p_0-2r}{p_0-r-1}}$ as defined in \eqref{h1}. Assuming there is $p_0>2r$ so that $u_t \in L^\infty(\mathbb{R}_+;L^{p_0}(\Omega))$ we obtain $\frac{h_0(x)}{h_1(x)} \rightarrow 0$ as $x \rightarrow 0^+$, and from Corollary \ref{uniformdecay2} the equation \eqref{approximate} reads as
			\begin{equation*}
			S_t+(\alpha S)^{\rho}=0, \quad S(0)=E_u(0),
			\end{equation*}
			where $\rho=\frac{p_0-r-1}{p_0-2r}$ and $\alpha=\gamma K^{-1}$ for any positive $\gamma<1$.
			
			As stated in Corollary \ref{decay}, the constant $K>0$ is given by $$K=\left[C(T,a,\operatorname{meas}(Q_T)) \cdot D_{0}^{\frac{p_0(r-1)}{p_0-r-1}} \right].$$ The solution to this ODE is
			\begin{equation*}
			S(t)=(\alpha^\rho(\rho-1)(t+c_0))^{-\frac{1}{\rho-1}}, \quad c_0=\frac{E_u(0)^{1-\rho}}{\alpha^\rho(\rho-1) }.
			\end{equation*}
			Thus, for $t$ sufficiently large ($t\geq 2t_0$), where $t_0$ is given in Corollary \ref{uniformdecay2}, we have
			\begin{equation*}
			E_u(t) \leq \left[ E_u(0)^{1-\rho} +\alpha^\rho\left(\frac{r-1}{p_0-2r}\right)\left(\frac{t}{t_0}-1 \right)\right]^{-\frac{p_0-2r}{r-1}}.
			\end{equation*}
		\end{Example}
		
		Below, we summarize the results obtained in Section \ref{6.1}. Decay rates computed in Table $1$ assume that the feedback map is linearly bounded at infinity.
		
		\begin{table}[H]
			\begin{center}
				{\def\arraystretch{4}\tabcolsep=5pt \resizebox{0.99999\textwidth}{!}{	\begin{tabular}{|c|c|c|c|c|c|}
							\hline
							& \multicolumn{5}{c|}{\fonte Behavior of the feedback near the origin ($|s|<1$)}                                                                                                                                                                                                                                                                                                                                                                                                \\ \hline
							& \fonte{$s$}                     & \fonte{$s^{p>1}$                                                                        } & \fonte{$s^{\theta<1}$                                                                                                                             } & \fonte{$s^{3}e^{-1/s^2}$                                                                   } & \fonte{$s|s|e^{-1/|s|}$                                                                                         } \\ \hline
							\fonte{Regularity} & \multicolumn{5}{c|}{\fonte{Finite energy}    }                                                                                                                                                                                                                                                                                                                                                                                                                       \\ \hline
							\fonte{ $S_t+y=0$}        & \fonte{$y=\tilde{C}S$}      & \fonte{$y=\tilde{C}S^{\frac{p+1}{2}}$              }                                  & \fonte{$y=\tilde{C}S^{\frac{\theta+1}{2\theta}}$   }                                                                                            & \fonte{$y=\tilde{C}^{2}S^2e^{-(\tilde{C}S)^{-1}}$                                      } & \fonte{$y=(\tilde{C}S)^{\frac{3}{2}}e^{-\frac{1}{\sqrt{\tilde{C}S}}}$                                      }  \\ \hline
							\fonte{$S(t)$}     & \fonte{$c_0e^{-\tilde{C}t}$} & \fonte{$\left[c_0+c_1(p-1) t \right]^{-\frac{2}{p-1}}$} & \fonte{$\left[ \left( \frac{1-\theta}{2\theta} \right)c_1t+ c_0 \right]^{-\frac{2\theta}{1-\theta}}$ }& \fonte{${c_1}^{-1}\ln^{-1}\left(c_1 t+c_0 \right)$} & \fonte{$c_1^{-1} \ln^{-2}\left(\frac{1}{2} (c_1t+c_0 ) \right)$} \\ \hline
				\end{tabular}}}
				\caption{\small Asymptotic energy decay rates for linearly bounded feedback $g(s)$.}
			\end{center}
		\end{table}

		The asymptotic decay rates computed in Table $2$ assume that the feedback map is linearly bounded at the origin and has order not equal to $1$ at infinity, in view of Definition \ref{def:order}. In this case decay in finite energy space requires regularity of solutions in stronger topology.
		
		\begin{table}[H]
			\begin{center}
				{\def\arraystretch{2.2}\tabcolsep=15pt	\begin{tabular}{|c|c|c|}
						\hline
						& \multicolumn{2}{c|}{Behavior of the feedback for $|s|>1$}                                                                                                                             \\ \hline
						& $g(s)=s^{\theta<1}$                                                                       & $g(s)=s^{r>1}$                                                                            \\ \hline
						Regularity & $u_t \in L^\infty(\mathbb{R}_{+};L^{p_0}(\Omega))$, $p_0>2$                               & $u_t \in L^\infty(\mathbb{R}_{+};L^{p_0}(\Omega))$, $p_0>2r$                              \\ \hline
						ODE        & $S_t+(\alpha S)^{\frac{p_0-\theta-1}{p_0-2}}=0$                                               & $S_t+(\alpha S)^{\frac{p_0-r-1}{p_0-2r}}=0$                                               \\ \hline
						$S(t)$     & $(c_1\frac{(1-\theta)}{p_0-2}(t+c_0))^{-\frac{p-2}{1-\theta}}$ & $(c_1\frac{(r-1)}{p_0-2r}(t+c_0))^{-\frac{p_0-2r}{r-1}}$ \\ \hline
				\end{tabular}}
				\caption{\small Asymptotic energy decay rates for feedbacks $g(s)$ linearly bounded at the origin (for $|s|<1$) and not linearly bounded at infinity (for $|s|\geq 1$).}
			\end{center}
		\end{table}
		
		\appendix
		\section{Preliminaries on microlocal analysis}\label{section2}
		In this section, we supplement details to proofs of some of the theorems on pseudo-differential operators and microlocal defect measures, whose original versions in French  can be found in the {elegant} lecture notes of  Burq and G\'erard \cite{Burq-Gerard}.
		
		\subsection{Pseudo-differential operators}
		Let $\Omega$ be an open and nonempty subset of $\mathbb{R}^d$, $d\ge 1$. A differential operator on $\Omega$ is a linear map $P:\mathcal{D}'(\Omega) \rightarrow \mathcal{D}'(\Omega)$ of the form
		\begin{eqnarray}\label{1.1}
		P u(x) := \sum_{|\alpha|\leq m} a_\alpha(x) \partial^\alpha_x u(x),~(\partial^\alpha:=\partial_{x_1}^{\alpha_1} \cdots \partial_{x_d}^{\alpha_d})
		\end{eqnarray}
		where $a_\alpha\in C^\infty(\Omega)$  are complex valued functions. The greatest integer $m$ such that the functions $a_\alpha$, ~$|\alpha|=m$ are not all zero is called the order of $P$.
		The map $p:\Omega \times \mathbb{R}^d \rightarrow \mathbb{C}$, defined by
		\begin{eqnarray*}
			p(x,\xi):= \sum_{|\alpha|\leq m}a_\alpha(x)\, (i\xi)^{\alpha},
		\end{eqnarray*}
		is called the symbol of $P$.
		
		We observe that $P$ is characterized by the identity
		\begin{eqnarray}\label{1.2}
		P(e_\xi)(x)=p(x,\xi)\,e_\xi(x),~~\hbox{ where }~e_\xi(\cdot) =e^{ i\,\left<(\cdot),\xi\right>}= e^{  i\,(\cdot)\cdot\xi}.
		\end{eqnarray}
		
		Adopting to the notation
		\begin{eqnarray*}
			D=\frac{1}{i} \partial, \ D_j=\frac{1}{ i}\partial_j \hbox{ and }D^\alpha = \frac{1}{i^{|\alpha|}}\partial^\alpha,
		\end{eqnarray*}
		the operator $P$ can be rewritten as
		\begin{eqnarray}\label{1.3}
		P= \sum_{|\alpha|\leq m} a_\alpha(x) i^{|\alpha|}D^\alpha = p(x,D).
		\end{eqnarray}
		
		The formula (\ref{1.2}) can be generalized as follows: for all $(x,\xi)\in \Omega \times \mathbb{R}^d$ and for all $u\in \mathcal{D}(\Omega)$,
		\begin{eqnarray}\label{1.4}
		P(u e_\xi) = p(x,D)(u e_\xi)= e_\xi\,p(x,\xi+D)(u)=e_\xi\, \sum_{| \alpha | \leq m} \frac{\partial_\xi^\alpha p(x,\xi)}{\alpha!} D^\alpha u,
		\end{eqnarray}
		where the previous sum is finite once $p$ is a polynomial in the variable $\xi$.
		
		If $P$ is a differential operator of order $m$
		and symbol $p$, then the principal symbol of order $m$, denoted by $\sigma_m(P)$, is the homogeneous part of degree $m$ in $\xi$ of the polynomial function $p(x,\xi)$, namely
		\begin{eqnarray}\label{1.5}
		\sigma_m(P)(x, \xi) = \sum_{|\alpha|=m}a_\alpha(x) (i\xi)^\alpha.
		\end{eqnarray}
		
		\begin{Definition}
			Let $m\in \mathbb{R}$. Then a symbol of order at most $m$ in $\Omega$ is said to be a function $a:\Omega \times \mathbb{R}^d \rightarrow \mathbb{C}$ of class $C^\infty$, with support in $K \times \mathbb{R}^d$, where $K$ is a compact subset of $\Omega$, such that for all $\alpha \in \mathbb{N}^d$, $\beta \in \mathbb{N}^d$, there exists a constant $C_{\alpha,\beta}$ with
			\begin{eqnarray}\label{1.6}
			\left|\partial_x^\alpha \partial_\xi^\beta a(x,\xi) \right| \leq C_{\alpha,\beta,K} \left(1+ |\xi|\right)^{m-|\beta|} .
			\end{eqnarray}
			We shall denote the vectorial space of all symbols of order at most $m$ in $\Omega$ by $\mathcal{S}_c^m(\Omega)$ .
		\end{Definition}
		
		\begin{Proposition}
			If $a\in \mathcal{S}_c^m(\Omega )$, the formula
			\begin{eqnarray}\label{1.7}
			Au(x) = \frac{1}{(2 \pi)^d}\int_{\mathbb{R}^d} e^{i x\cdot \xi} a(x,\xi)\hat u(\xi)\,d\xi
			\end{eqnarray}
			defines, for all $u\in C_0^\infty(\Omega)$, an element $Au$ of $C_0^\infty(\Omega)$.
		\end{Proposition}
		
		The formula \eqref{1.7} defines a linear map $A:C_0^\infty(\Omega) \longrightarrow C_0^\infty(\Omega)$, which is called the pseudo-differential operator of order $m$ and symbol $a$. We will often denote the map $A$ by $a(x,D)$.
		
		The set of all pseudo-differential operators of order $m$ on $\Omega$ will be denoted by $\Psi^m_c(\Omega)$.
		
		\begin{Definition}\label{essehomo}
			An operator $A\in \Psi^m_c(\Omega)$ is essentially homogeneous if there exists a function $a_m=a_m(x,\xi)$ with $\supp a_m \subset K \times (\mathbb{R}^{d}\setminus\{0\})$, homogeneous of order $m$ in $\xi$ and smooth except at $\xi=0$ and a function $\chi\in C^\infty(\mathbb{R}^d)$ being zero near $0$ and $1$ in the infinity such that
			\begin{eqnarray}\label{1.8}
			a(x,\xi)=a_m(x,\xi)\chi(\xi)+ r(x,\xi),
			\end{eqnarray}
			for some $r\in \mathcal{S}^{m-1}_{c}(\Omega)$.
		\end{Definition}
		
		\begin{Proposition}\label{Prop1.1}
			Let $A\in \Psi^m_c(\Omega)$ be essentially homogeneous. Then, for all $u\in C_0^\infty(\Omega)$, $\xi \in \mathbb{R}^d\backslash \{0\}$, and $x\in \Omega$,
			\begin{eqnarray}\label{1.9}
			t^{-m} e^{- i (t x)\cdot\xi} A(u e_{t\xi})(x) \rightarrow a_m(x,\xi)u(x),~\hbox{ as }t\rightarrow +\infty.
			\end{eqnarray}
		\end{Proposition}

		\begin{Definition}\label{definition 3.50}
			Under the conditions of Proposition \ref{Prop1.1}, we say that $A$ admits a principal symbol of order $m$. The function $a_m$ characterized by (\ref{1.9}) is called  the principal symbol of order $m$ of $A$ and is denoted by $\sigma_m(A)$.
		\end{Definition}

		\begin{Theorem}\label{Theorem 3.10'}
			Let $A$ be a pseudo-differential operator of symbol $a\in \mathcal{S}_{c}^{m}(\Omega )$ with $\supp a \subset K \times \mathbb{R}^{d}$, and let $\chi \in C_0^\infty(\Omega)$ satisfy $\chi=1$ in a neighborhood of $K$. Then,  there exists a pseudo-differential operator $A_{\chi}^\ast$ on $\Omega$ such that, for all $u,v \in C_0^\infty(\Omega)$,
			\begin{eqnarray*}
				\left(A(\chi u),v\right)_{L^2(\Omega)} = \left(u, A_{\chi}^\ast v \right)_{L^2(\Omega)}.
			\end{eqnarray*}
			In addition, $A_{\chi}^\ast$ admits a symbol $a_\chi^\ast\in \mathcal{S}_{c}^{m}(\Omega )$ verifying, for all $N\in \mathbb{N}$,
			\begin{eqnarray}\label{1.10}
			a_\chi^\ast - \sum_{|\alpha|\leq N} \frac{1}{\alpha!} D_x^\alpha \partial_\xi^\alpha \overline{a}\in \mathcal{S}_{c}^{m-N-1}(\Omega ).
			\end{eqnarray}
			In particular, if $A$ admits a principal symbol of order $m$, it is the same of $A^\ast$, and
			\begin{eqnarray}\label{1.11}
			\sigma_m(A_\chi^\ast)= \overline{\sigma_m(A)}.
			\end{eqnarray}
		\end{Theorem}
		
		\begin{Theorem}\label{Theorem 3.13}
			Let $A$ and $B$ be pseudo-differential operators of symbols $a\in \mathcal{S}^{m}_{c}(\Omega )$, $b\in \mathcal{S}_{c}^{n}(\Omega )$, respectively. Then, the composition $AB$ is a pseudo-differential operator admitting a symbol $a\#b \in S_c^{m+n}(\Omega)$ which satisfies
			\begin{eqnarray}\label{1.12}
			a\# b -\sum_{|\alpha|\leq N} \frac{1}{\alpha!} \partial_\xi^\alpha a D_x^\alpha b \in \mathcal{S}_{c}^{m+n-N-1}(\Omega ),
			\end{eqnarray}
			for all $N \in \mathbb{N}$.
			
			In addition, if $A$ admits a principal symbol of order $m$ and $B$ admits a principal symbol of order $n$, then $AB$ admits a principal symbol of order $m+n$ and $[A,B]$ admits a principal symbol of order $m+n-1$,  given by
			\begin{eqnarray}
			&& \sigma_{m+n}(AB) =\sigma_m(A) \sigma_n(B),\label{1.13}\\
			&& \sigma_{m+n-1}([A,B])=\frac{1}{i}\left\{\sigma_m (A), \sigma_n(B)\right\}\label{1.14}.
			\end{eqnarray}
		\end{Theorem}
		
		\begin{Definition}
			For any compact set $K$ contained in $\Omega$ and for all $s\in \mathbb{R}$, $H_K^s(\Omega)$ shall denote the space of distributions with compact support in $K$ whose extensions by zero belong to $H^s(\mathbb{R}^d)$, i.e.,
			\begin{eqnarray*}
				H_K^s(\Omega) & \equiv  & \{u \in \mathcal{E}'(\Omega): \tilde{u} \in H^s(\mathbb{R}^d) \hbox{ and } supp\, u \subset K  \}
			\end{eqnarray*}
			We set  $H_{comp}^s(\Omega)=\underset{K}\bigcup H_K^s(\Omega)$ where $K$ ranges over all compact subsets of $\Omega$.
			We equip $H_{comp}^s(\Omega)$ with the finest locally convex topology such that all the inclusion maps
			\begin{eqnarray*}
				H_K^s(\Omega) \hookrightarrow H_{comp}^s(\Omega)
			\end{eqnarray*}
			are continuous.
		\end{Definition}

		\begin{Theorem}\label{Theorem 3.14}
			Let $a\in \mathcal{S}_{c}^{m} (\Omega )$ and let $K$ be the projection on $\Omega$ of $\supp(a)$. Thus, for all $s\in \mathbb{R}$, the operator defined in (\ref{1.7}) admits a unique extension to a linear and continuous map from $H_{comp}^s(\Omega)$ in $H_K^{s-m}(\Omega)$.
		\end{Theorem}

		\begin{Remark}\label{compact}
			If $A \in \Psi_{c}^{m}(\Omega)$ is a pseudo-differential operator with
			$m < 0$ then $A: L_{comp}^{2}(\Omega) \rightarrow L^{2}(\Omega)$ is a compact operator. Indeed, from Rellich's Theorem the inclusion $H_{comp}^{-m} \hookrightarrow L^2(\Omega)$ is compact.
		\end{Remark}
		
		\begin{Remark}\label{extension}
			Let $P \in \Psi^m(\Omega)$ be a compactly supported operator, then $P$ extends continuously from $H_{loc}^{s}(\Omega)$ to $H^{s-m}_{comp}(\Omega).$
		\end{Remark}
		
		\begin{Theorem}[{\bf A G$\rm\overset{\circ}a$rding type inequality}]\label{Theorem 3.59}~Let $A$ be a pseudo-differential operator of order $0$ on $\Omega$ whose principal symbol $\sigma_0(A)$ exists and is a positive function in $\mathcal{S}_{c}^{0}(\Omega)$. Then, for all $\delta>0$ there exists $C_\delta$ such that
			\begin{eqnarray}\label{1.15}
			Re\, (Av,v)_{L^2}\geq -\delta \|v\|_{L^2}^2 - C_\delta \|v\|_{H^{-1/2}}^2,~\hbox{ for all }v\in L_{comp}^2.
			\end{eqnarray}
			
			Furthermore, there exists $C>0$ such that
			\begin{eqnarray}\label{1.16}
			|Im (Av,v)_{L^2}| \leq C \|v\|_{H^{-1/2}}^2, ~\hbox{ for all }v\in L_{comp}^2.
			\end{eqnarray}
		\end{Theorem}
		
		\subsection{Microlocal defect measures}
		Let $\{u_n\}_{n\in \mathbb{N}}$ be a bounded sequence in $L^2_{loc}(\Omega)$, i.e.,
		\begin{eqnarray*}
			\sup_{n\in \mathbb{N}} \int_K |u_n(x)|^2\,dx<+\infty,
		\end{eqnarray*}
		for any compact set $K$ contained in $\Omega$.
		
		We shall say that $u_n$ converges weakly to $u\in L^2_{loc}(\Omega)$ if one has
		\begin{eqnarray*}
			\int_\Omega u_n(x)f(x)\,dx \underset{n\rightarrow \infty}\longrightarrow \int_\Omega u(x)f(x)\,dx
		\end{eqnarray*} for each $f\in L_{comp}^2(\Omega)$.
		
		We describe the loss of strong convergence of the sequence $u_n$ to $0$ in $L^2_{loc}(\Omega)$, by means of a positive Radon measure on $\Omega \times S^{d-1}$, that is, $u_n \rightarrow 0$ in $L^2_{loc}(\Omega)$ if, and only if, $\mu=0$.
		
		\begin{Lemma}\label{Lemma 4.48}
			Let $\{u_n\}_{n\in \mathbb{N}}$ be a bounded sequence in $L_{loc}^2(\Omega)$  weakly converging to $0$. Let $A$ be a pseudo-differential operator of order $0$ on $\Omega$ that admits a principal symbol $\sigma_0(A)$ of order $0$. If $\sigma_0(A) \geq 0$, then one has
			\begin{eqnarray*}
				Im (A(\chi u_n), \chi u_n)_{L^2}\underset{n\rightarrow +\infty}\longrightarrow 0~\hbox{ and }~\liminf_{n\rightarrow +\infty} Re(A(\chi u_n),\chi u_n)_{L^2}\geq0.
			\end{eqnarray*}
		\end{Lemma}
		\begin{proof}
			First, note that since $u_n$ is bounded in $L_{loc}^2(\Omega)$ and converges weakly to $0$, we have
			\begin{eqnarray}\label{1.17}
			\|\chi u_n\|_{H^{-1/2}}^2 \underset{n\rightarrow +\infty}\longrightarrow 0.
			\end{eqnarray}
			
			Indeed, recall that
			\begin{eqnarray*}
				\|\chi u_{n}\|_{H^{-1/2}}^2 = (2\pi)^{-d}\int_{\mathbb{R}^d}(1+|\xi|^2)^{-1/2}|\widehat{\chi u_n}(\xi)|^2\,d\xi,
			\end{eqnarray*}
			where
			\begin{eqnarray*}
				\widehat{\chi u_n}(\xi) =\int_{\mathbb{R}^d} \chi(x) u_n(x) e^{- ix\cdot \xi}\,dx
			\end{eqnarray*}
			tends to $0$ for all $\xi \in \mathbb{R}^d$ and remains uniformly bounded in $\xi$ and $n$. Indeed, from the Cauchy-Schwarz inequality we have
			\begin{eqnarray*}
				\left| \int_{\mathbb{R}^{d}}{\chi(x)u_{n}(x)e^{- i x \xi}}dx \right| & = &
				\left| \int_{\supp( \chi )}{\chi(x)u_{n}(x)e^{- i x \xi}}dx \right| \\
				& \leq & \limsup_{n \to \infty} \| \chi u_{n} \|_{L^{2}}\int_{\supp (\chi)} 1 dx  < \infty  .
			\end{eqnarray*}
			Applying the dominated convergence theorem, we have, for all $R>0$,
			\begin{eqnarray*}
				(2\pi)^{-d}\int_{|\xi|\leq R} (1+|\xi|^2)^{-1/2}|\widehat{\chi u_n}(\xi)|^2\,d\xi\leq	(2\pi)^{-d}\int_{|\xi|\leq R} |\widehat{\chi u_n}(\xi)|^2\,d\xi \underset{n\rightarrow +\infty}\longrightarrow 0.
			\end{eqnarray*}
			
			On the other hand, the Plancherel theorem yields
			\begin{eqnarray*}
				(2\pi)^{-d}\int_{|\xi|>R} |\widehat{\chi u_n}(\xi)|^2(1+|\xi|^2)^{-1/2}\,d\xi\leq \frac{1}{R}\|\chi u_n\|_{L^2}^2,
			\end{eqnarray*}
			so,
			\begin{eqnarray*}
				\limsup_{n\rightarrow +\infty}\|\chi u_n\|_{H^{-1/2}}^2 \leq \frac{1}{R}\limsup_{n\rightarrow +\infty}\|\chi u_n\|_{L^2}^2,
			\end{eqnarray*}
			which proves (\ref{1.17}), since $R$ is arbitrary. Applying (\ref{1.15}) and (\ref{1.16}) to $v=\chi u_n$, letting $n\rightarrow +\infty$, and using (\ref{1.17}), we obtain
			\begin{eqnarray*}
				\liminf_{n\rightarrow +\infty}Re (A(\chi u_n),\chi u_n)_{L^2}
				&\geq&\liminf_{n\rightarrow +\infty} [-\delta \|\chi u_n\|_{L^2}^2] + \liminf_{n\rightarrow 0}[-C_\delta \|\chi u_n \|_{H^{-1/2}}] \\
				&\geq& -\delta \limsup_{n\rightarrow +\infty}\|\chi u_n\|_{L^2}^2.
			\end{eqnarray*}
			In addition,
			\begin{equation*}
			Im (A(\chi u_n), \chi u_n)_{L^2} \underset{n\rightarrow +\infty}\longrightarrow 0,
			\end{equation*}
			which proves the lemma since $\delta>0$ is arbitrary.
		\end{proof}
		
		\begin{Theorem}\label{Theorem 4.49}
			Let $(u_n)_{n\in \mathbb{N}}$ be a bounded sequence in $L_{loc}^2(\Omega)$  weakly converging to $0$. Then, there exists a subsequence $(u_n)_{n \in \mathbb{N}'}$ and a positive Radon measure $\mu$ on $\Omega \times S^{d-1}$ such that for any essentially homogeneous pseudo-differential operators $A \in \Psi^{0}_{c}(\Omega)$ with principal symbol $\sigma_{0}(A)$, one has
			\begin{eqnarray}\label{1.18}
			\left(A(\chi u_n), \chi u_n\right)_{n \in \mathbb{N}'} \underset{n\rightarrow +\infty}\longrightarrow \int_{\Omega \times S^{n-1}}\sigma_0(A)(x,\xi)\,d\mu
			\end{eqnarray}
			for all $\chi \in C_{0}^{\infty}(\Omega)$ such that $\chi=1$ in $\pi_x(\supp (\sigma_0(A)))$.
		\end{Theorem}
		\begin{proof}
			From the assumption of the theorem, one has
			\begin{eqnarray}\label{1.19}
			\limsup_{n \rightarrow +\infty}\|A(\chi u_{n})\|_{L^{2}(\Omega)}\leq C(K,\chi)\,\underset{(x,\xi)\in \Omega \times S^{d-1}}{\max}|\sigma_{0}(A)|.
			\end{eqnarray}
			
			Indeed, first observe that since $A\in \Psi_c^0(\Omega)$, it maps $H_{comp}^s(\Omega)$ into $H_K^s(\Omega)$ for all $s\in \mathbb{R}$ and some compact set $K\subset \Omega$. In particular, $A$ maps $L_{comp}^2(\Omega)$ continuously into $L_K^2(\Omega)$. Since $(u_n)_{n\in \mathbb{N}}$ is bounded in $L^2_{loc}(\Omega)$, it follows that $(A (\chi u_n))$ is bounded in $L_{K}^2(\Omega)$.
			
			Now, from the Cauchy-Schwarz inequality,
			\begin{eqnarray}\label{1.20}
			|(A(\chi u_n), \chi u_n)_{L^2(\Omega)}|&\leq& \|A(\chi u_n)\|_{L^2(\Omega)}\|\chi u_n\|_{L^2(\Omega)} \nonumber \\
			& = & (A(\chi u_{n}), A(\chi u_{n}))_{L^{2}(\Omega)}\|\chi u_n\|_{L^2(\Omega)}  \nonumber \\
			& = &   (A_{\chi}^{\ast} A (\chi u_n), u_n)\|\chi u_n\|_{L^2(\Omega)}.
			\end{eqnarray}
			The principal symbol of the operator $A_{\chi}^{\ast}A$ is $|\sigma_0(A)|^2$. Hence, $A_{\chi}^{\ast}A= b(x,D) +r(x,D)$, where $r(x,D)\in \Psi_{c}^{-1}(\Omega)$ and $b(x,D)$ is an operator with symbol $\gamma(\xi)|\sigma_{0}(A)|^{2}$, in which $\gamma$ is $0$ near the origin and $1$ at infinity.
			
			Then,
			$$\gamma(\xi) \sigma_{0}(A_{\chi}^{\ast}A) =   \gamma(\xi) \chi |\sigma_{0}(A)|^{2} \leq  \chi \left(\displaystyle\max_{(x,\xi) \in \Omega \times S^{d-1}}|\sigma_{0}(A)| \right)^{2} = \chi M^{2}.$$
			
			Observe that $\chi u_{n} \rightharpoonup 0$ weakly in $L_{comp}^{2}(\Omega)$. Indeed, let $g \in L_{loc}^{2}(\Omega)$. Then, $\chi g \in L_{comp}^{2}(\Omega)$. Thus
			\begin{eqnarray*}
				\int_{\Omega}(\chi u_n)gdx &=& \int_{\Omega} u_n(\chi g)dx \rightarrow 0,
			\end{eqnarray*}
			since $u_n \rightharpoonup 0$ in $L_{loc}^{2}(\Omega)$.
			
			From Remark \ref{compact}, $r(x, D)(\chi u_{n}) \rightarrow 0$ strongly in $L^{2}(\Omega)$, and hence
			\begin{eqnarray*}
				(r(x, D)(\chi u_{n}), u_{n})_{L^2(\Omega)} \rightarrow 0 \hbox{ as }  n \to \infty.
			\end{eqnarray*}
			On the other hand, since $\chi=1$ in $K = \pi_{x}(\supp (\sigma_{0}(A)))$,
			\begin{eqnarray*}
				(b(x,D)(\chi u_{n}), u_{n}) & = & \int_{K} b(x,D)(\chi u_{n})(x) u_{n}(x)dx \\
				& = & (b(x,D)(\chi u_{n}), \chi u_{n}) \\
				& = & - ((M^{2}\chi I - b(x,D))(\chi u_{n}), \chi u_{n}) +(( M^{2}\chi I)\chi u_{n}, \chi u_{n}).
			\end{eqnarray*}
			We estimate the term $\displaystyle\limsup_{n \to \infty} \| A(\chi u_{n}) \|_{L^{2}(\Omega)}^{2}$ as follows:
			\begin{eqnarray*}
				\limsup_{n \to \infty} \| A(\chi u_{n}) \|_{L^{2}(\Omega)}^{2} & = & \limsup_{n \to \infty}  ( A_{\chi}^{\ast}A(\chi u_{n}), u_{n}) \\
				& = & \limsup_{n \to \infty}  Re ( A_{\chi}^{\ast}A(\chi u_{n}), u_{n})\\
				& \leq & \limsup_{n \to \infty} Re ( b(x,D)(\chi u_{n}), u_{n})+ \limsup_{n \to \infty}Re(r(x,D)(\chi u_n),u_n)\\
				& \leq & \limsup_{n \to \infty} Re -(M^2 \chi I - b(x,D)(\chi u_n),\chi u_n)\\&& + \limsup_{n \to \infty}(( M^{2}\chi I)\chi u_{n}, \chi u_{n}) \\
				& \leq & - \liminf_{n \to \infty}Re((M^{2} \chi I - b(x,D))(\chi u_{n}), \chi u_{n})\\&& + M^{2} \max_{x \in \Omega}\chi(x) \limsup_{n \to \infty}\|\chi u_n\|^{2}\\
				& \leq &  M^{2} \max_{x \in \Omega}\chi(x) \limsup_{n \to \infty}\|\chi u_{n}\|^{2},
			\end{eqnarray*}
			where the last inequality follows from Lemma $\ref{Lemma 4.48}$.
			
			Therefore, from the properties of $\limsup$, it follows that
			\begin{eqnarray*}
				\limsup_{n \to \infty} \|A(\chi u_{n})\| \leq C(\chi, K)\underset{(x,\xi)\in \Omega \times S^{d-1}}{\max}|\sigma_{0}(A)|.
			\end{eqnarray*}
			Thus, from \eqref{1.20} we obtain
			\begin{eqnarray}\label{1.21}
			\limsup_{n \to \infty}|(A(\chi u_{n}), \chi u_{n})| \leq C \underset{(x,\xi)\in \Omega \times S^{d-1}}{\max}|\sigma_{0}(A)|.
			\end{eqnarray}
			
			For any compact set $K$ of $\Omega$, $C_K^\infty(\Omega \times S^{d-1})$ will denote the vectorial space of functions $C^\infty$ on $\Omega \times S^{d-1}$, with compact support in $K\times S^{d-1}$, endowed with the $L^\infty$ norm. The space $C_K^\infty(\Omega \times S^{d-1})$ is separable, since it is isometric to a subspace of the separable space $C(K \times S^{d - 1})$.
			
			So, let $D := \operatorname{span}\{ a_i : i\in \mathbb{N}\}$ be a countable and dense subset of $C_K^\infty(\Omega \times S^{d-1})$. For all $i \in \mathbb{N}$, let $A_i$ be a pseudo-differential operator such that $\sigma_0(A_i)=a_i$.
			
			Taking \eqref{1.21} into account, for each  fixed $i$ there exists a constant $$C_i:=C(\chi,K)\underset{(x,\xi)\in \Omega \times S^{d-1}}{\max}|\sigma_{0}(A_i)|$$ such that
			$$\limsup_{n\rightarrow +\infty}|(A_i(\chi u_n),\chi u_n)|\leq C_i\hbox{ for all }n\in \mathbb{N}.$$
			
			By virtue of  Cantor's diagonal argument there exists a infinite subset $\mathbb{N}'$ of $\mathbb{N}$ such that the quantity $(A_i(\chi u_{n}),\chi u_n)_{n \in \mathbb{N}'}$ has a limit for all $i$.
			
			In fact, the sequence $\{(A_1(\chi u_n),\chi u_n)\}_{n \in \mathbb{N}}$, being bounded, has a convergent subsequence. Thus there exist an infinite subset $\mathbb{N}_{1} \subset \mathbb{N}$ and a number $\alpha_1$ for which $\underset{n \in \mathbb{N}_1}{\lim} (A_1(\chi u_n),\chi u_n)=\alpha_1$. The sequence $(A_2(\chi u_n),\chi u_n)_{n \in \mathbb{N}_1}$ is also bounded. So there exist an infinite subset $\mathbb{N}_2 \subset \mathbb{N}_1$ and a number $\alpha_2$ such that $\underset{n \in \mathbb{N}_2}{\lim} (A_1(\chi u_n),\chi u_n)=\alpha_2$. Proceeding in the same fashion, we obtain, for each $i \in \mathbb{N}$, an infinite subset $\mathbb{N}_i \subset \mathbb{N}$, such that $\mathbb{N}_1 \supset \mathbb{N}_2 \supset \cdots \supset \mathbb{N}_i \supset \cdots$ and a number $\alpha_i$ such that $\underset{n \in \mathbb{N}_i}{\lim} (A_i(\chi u_n),\chi u_n)=\alpha_i$. Let us then define an infinite subset $\mathbb{N}' \subset \mathbb{N}$, taking the $i$-th element of $\mathbb{N}'$ as the $i$-th element of $\mathbb{N}_i$. For each $i \in \mathbb{N}$, the sequence $(A_i(\chi u_n),\chi u_n)_{n \in \mathbb{N}'}$ is, from its $i$-th element, a subsequence of $(A_i(\chi u_n),\chi u_n)_{n \in \mathbb{N}_i}$ and therefore converges. Thus, $(A_i(\chi u_n),\chi u_n)_{n \in \mathbb{N}'} \rightarrow \alpha_i$ for all $i \in \mathbb{N}$.
			
			Define $L: D \subset C_{K}^{\infty}(\Omega \times S^{d-1}) \longrightarrow \mathbb{C}$ by setting $L(a_k)=\alpha_k$. From \eqref{1.21}, it follows that
			\begin{eqnarray}\label{1.22}
			|L(a_k)|&=&|\alpha_ k| \nonumber\\
			&=& \lim_{n \in \mathbb{N}'} |(A_k(\chi u_n),\chi u_n)|\nonumber \\
			&\leq & C  \underset{(x,\xi)\in \Omega \times S^{d-1}}{\max}|\sigma_{0}(A_k)|\nonumber \\
			&=& C \underset{(x,\xi)\in \Omega \times S^{d-1}}{\max}|a_k|.
			\end{eqnarray}
			
			Note that $\alpha_k$ does not depend on the choice of the operator $A_k$ satisfying $\sigma_0(A_k)=a_k$. In fact, let $B_k$ be another operator with $\sigma_0(B_k)=a_k$, then from \eqref{1.21}, it follows that
			$$\limsup_{n \in \mathbb{N}'} |((B_k-A_k)(\chi u_n),\chi u_n)|\leq C  \underset{(x,\xi)\in \Omega \times S^{d-1}}{\max}|\sigma_{0}(B_k-A_k)|=0.$$
			Therefore,
			\begin{eqnarray*}
				\lim_{n \in \mathbb{N}'}(B_k(\chi u_n),\chi u_n)&=& \lim_{n \in \mathbb{N}'}((B_k-A_k)(\chi u_n),\chi u_n)+\lim_{n \in \mathbb{N}'}(A_k(\chi u_n),\chi u_n)\\&=& \alpha_k.
			\end{eqnarray*}
			Thus, the mapping $L$ is well defined and densely defined. Hence, it uniquely extends to a bounded linear functional $\tilde{L}: C_{K}^{\infty}(\Omega\times S^{d-1}) \rightarrow \mathbb{C}$ satisfying the estimate
			$$|\tilde{L}(a)|\leq C \underset{(x,\xi)\in \Omega \times S^{d-1}}{\max}|a|.$$

			By construction, given $\sigma_0(A) \in C_{K}^{\infty}(\Omega \times S^{d-1})$, we have
			\begin{equation*}
			(A(\chi u_n),\chi u_n)_{n \in \mathbb{N}'} \underset{n \rightarrow \infty}{\longrightarrow} \tilde{L}(\sigma_0(A)).
			\end{equation*}
			
			In fact, let $(a_i) \subset D$ with $a_i \rightarrow \sigma_0(A)$ as $i\rightarrow +\infty$; that is, given $\varepsilon > 0,$ there exists $i_{0} \in \mathbb{N}$ such that
			\begin{equation}\label{1.23}
			\|\sigma_{0}(A) - a_{i}\|_{\infty} < \frac{\varepsilon}{3C}, \hbox{ for all } i \geq i_{0}.
			\end{equation}
			
			Moreover, there is $n_{0} \in \mathbb{N},$ such that for all $n \geq n_{0}$,
			\begin{eqnarray}\label{1.24}
			| (A(\chi u_{n}), \chi u_{n}) - (A_{i_{0}}(\chi u_{n}), \chi u_{n})| &\leq& C \|\sigma_{0}(A) - a_{i_{0}}\|_{\infty} \nonumber\\
			&\leq& C \frac{\varepsilon}{3 C} =\frac{\varepsilon}{3}.
			\end{eqnarray}
			
			On the other hand, there exists $n_{1} \in \mathbb{N}$ such that for all $n \geq n_{1}$,
			\begin{eqnarray}\label{1.25}
			| (A_{i_{0}}(\chi u_{n}), \chi u_{n}) - \tilde{L}(a_{i_{0}})| < \frac{\varepsilon}{3}.
			\end{eqnarray}
			
			So, taking $n_{2} = \max\{n_{0}, n_{1}\}$, from (\ref{1.23}), (\ref{1.24}) and (\ref{1.25}), it follows that for all $n \geq n_{2},$
			\begin{eqnarray*}
				| (A(\chi u_{n}), \chi u_{n}) - \tilde{L}(\sigma_{0}(A))| &\leq& |(A(\chi u_{n}), \chi u_{n}) - (A_{i_{0}}(\chi u_{n}), \chi u_{n})| \\
				&&+ |(A_{i_{0}}(\chi u_{n}), \chi u_{n}) - \tilde{L}(a_{i_{0}})|\\
				&&+|\tilde{L}(a_{i_{0}}) - \tilde{L}(\sigma_{0}(A))|\\
				&\leq& \frac{\varepsilon}{3} + \frac{\varepsilon}{3} + C \|\sigma_{0}(A) - a_{i}\|_{\infty}\\
				&\leq& \frac{\varepsilon}{3} + \frac{\varepsilon}{3} + \frac{\varepsilon}{3} = \varepsilon.
			\end{eqnarray*}
			
			Then,
			\begin{equation*}
			(A(\chi u_n),\chi u_n)_{n \in \mathbb{N}'} \underset{n \rightarrow \infty}{\longrightarrow} \tilde{L}(\sigma_0(A)).
			\end{equation*}
			
			Note that $L$ does not depend on the choice of $\chi$. Indeed, let $\chi, \tilde{\chi} \in C_{0}^{\infty}(\Omega)$ with
			$\chi \sigma_{0}(A) = \sigma_{0}(A) = \tilde{ \chi}\sigma_{0}(A)$. For $\eta \in C_{0}^{\infty}(\Omega)$ such that $\eta=1$ in a neighborhood of $\supp \chi \cup \supp \tilde{\chi}$, we have
			\begin{small}
				\begin{eqnarray*}
					\left(A(\tilde{\chi}u_{n}), \tilde{\chi} u_{n} \right) - \left(A(\chi u_{n}), \chi u_{n} \right) & = & \left(b(x,D)(\tilde{\chi}u_{n}), \tilde{\chi} u_{n} \right) - \left(b(x,D)(\chi u_{n}), \chi u_{n} \right) \\
					&&+ \left(r(x, D)(\tilde{\chi}u_{n}), \tilde{\chi} u_{n} \right) - \left(r(x, D)(\chi u_{n}), \chi u_{n} \right) \\
					&=& \left(b(x,D)((\tilde{\chi}-\chi)u_{n}), \tilde{\chi} u_{n} \right) + \left(b(x,D)(\chi u_{n}), (\tilde{\chi}-\chi) u_{n} \right) \\
					&&+ \left(r(x, D)(\tilde{\chi}u_{n}), \tilde{\chi} u_{n} \right) - \left(r(x, D)(\chi u_{n}), \chi u_{n} \right) \\
					&=& \left(b(x,D)(\eta(\tilde{\chi}-\chi)u_{n}), \tilde{\chi} u_{n} \right) + \left(b(x,D)(\chi u_{n}), (\tilde{\chi}-\chi) u_{n} \right) \\
					&&+ \left(r(x, D)(\tilde{\chi}u_{n}), \tilde{\chi} u_{n} \right) - \left(r(x, D)(\chi u_{n}), \chi u_{n} \right) \\
					&=& \left((\tilde{\chi}-\chi)u_{n}, b(x,D)_{\eta}^{\ast}(\tilde{\chi} u_{n}) \right) + \left(b(x,D)(\chi u_{n}), (\tilde{\chi}-\chi) u_{n} \right) \\
					&&+ \left(r(x, D)(\tilde{\chi}u_{n}), \tilde{\chi} u_{n} \right) - \left(r(x, D)(\chi u_{n}), \chi u_{n} \right) \\
					&=& \left((\tilde{\chi}-\chi)u_{n}, \overline{b(x,D)}(\tilde{\chi} u_{n}) \right) + \left(b(x,D)(\chi u_{n}), (\tilde{\chi}-\chi) u_{n} \right) \\
					&&+ \left(r(x, D)(\tilde{\chi}u_{n}), \tilde{\chi} u_{n} \right) - \left(r(x, D)(\chi u_{n}), \chi u_{n} \right) \\
					&&+ \left((\tilde{\chi}-\chi)u_{n}, r_1(\tilde{\chi} u_{n}) \right),
				\end{eqnarray*}
			\end{small}
			since $A=b(x,D) + r(x,D)$, where $b_{\eta}^{\ast}(x,\xi)=\overline{b(x,\xi)}+r_1$ with $r_1 \in \mathcal{S}_{c}^{-1}(\Omega)$. Since $\chi - \tilde{\chi} =0$ in $\pi_{x}(\supp \sigma_{0}(A))$, we have
			\begin{eqnarray*}
				\left(A(\tilde{\chi}u_{n}), \tilde{\chi} u_{n} \right)_{n \in \mathbb{N}'} - \left(A(\chi u_{n}), \chi u_{n} \right)_{n \in \mathbb{N}'} \rightarrow 0.
			\end{eqnarray*}
			
			Therefore,
			\begin{eqnarray*}
				\left(A(\tilde{\chi}u_{n}), \tilde{\chi} u_{n} \right)_{n \in \mathbb{N}'} &=&  	\left(A(\tilde{\chi}u_{n}), \tilde{\chi} u_{n} \right)_{n \in \mathbb{N}'}- \left(A(\chi u_{n}), \chi u_{n} \right)_{n \in \mathbb{N}'} +\left(A(\chi u_{n}), \chi u_{n} \right)_{n \in \mathbb{N}'}\\
				&\rightarrow& L(\sigma_0(A)).
			\end{eqnarray*}
			
			So far we have extracted a subsequence of $(u_n)$ such that $$(A(\chi u_n),\chi u_n) \underset{n \rightarrow \infty}{\longrightarrow} \tilde{L}(\sigma_0(A))$$ for all $A \in \Psi_{c}^{0}(\Omega)$ with $\supp \sigma_0(A) \subset K \times \mathbb{R}^{d}$ in some fixed compact $K$. We need to remove the dependence on $K$.  To this end, let $K_i \subset \overset{\circ}{K}_{i+1}$ be a monotone sequence of compact subsets of $\Omega$ such that $\Omega = \overset{\infty}{\underset{n=1}\bigcup}K_n$. From the construction, there exists an infinite subset $\mathbb{N}_1 \subset \mathbb{N}$ and a continuous linear form $\tilde{L}_1$ in $C_{K_1}^{\infty}(\Omega \times S^{d-1})$ such that
			$$ (A_1(\chi u_n),\chi u_n)_{\mathbb{N}_{1}} \underset{n \rightarrow \infty}{\longrightarrow} \tilde{L}_1(\sigma_0(A_1))$$ for all $A_1 \in \Psi_{c}^{0}(\Omega)$ with $\supp \sigma_0(A_1)\subset K_1 \times \Omega$. The sequence $(u_n)_{n \in \mathbb{N}_1}$ is still a bounded sequence in $L_{loc}^2(\Omega)$  weakly converging to $0$. Thus we can obtain an infinite subset $\mathbb{N}_{2} \subset \mathbb{N}_1$ and a continuous linear form $\tilde{L}_2$ in $C_{K_2}^{\infty}(\Omega \times S^{d-1})$ such that
			$$ (A_2(\chi u_n),\chi u_n)_{\mathbb{N}_{2}} \underset{n \rightarrow \infty}{\longrightarrow} \tilde{L}_2(\sigma_0(A_2)),$$ for all $A_2 \in \Psi_{c}^{0}(\Omega)$ with $\supp \sigma_0(A_2)\subset K_2 \times \Omega$.
			
			Proceeding in the same fashion, we obtain, for each $i \in \mathbb{N}$, an infinite subset $\mathbb{N}_i \subset \mathbb{N}$  and a continuous linear form $\tilde{L}_i$ in $C_{K_i}^{\infty}(\Omega \times S^{d-1})$ such that $\mathbb{N}_1 \supset \mathbb{N}_2 \supset \cdots \supset \mathbb{N}_i \supset \cdots$ and, for all $i\in \mathbb{N}$,
			$$ (A_i(\chi u_n),\chi u_n)_{\mathbb{N}_{i}} \underset{n \rightarrow \infty}{\longrightarrow} \tilde{L}_i(\sigma_0(A_i))$$ for all $A_i \in \Psi_{c}^{0}(\Omega)$ with $\supp \sigma_0(A_i)\subset K_i \times \Omega$.
			Let us then define an infinite subset $\mathbb{N}' \subset \mathbb{N}$, whose $i$-th element is the $i$-th element of $\mathbb{N}_i$. In this way, for each $i \in \mathbb{N}$, the sequence $(u_n)_{n \in \mathbb{N}'}$ is, starting from its $i$-th element, a subsequence of $(u_n)_{n \in \mathbb{N}_i}$.
			
			We want to define a linear functional $L:C_{0}^{\infty}(\Omega \times S^{d-1}) \longrightarrow \mathbb{C}$ satisfying the limit
			$$ (A(\chi u_n),\chi u_n)_{\mathbb{N}'} \underset{n \rightarrow \infty}{\longrightarrow} L(\sigma_0(A))$$ for all $A \in \Psi_{c}^{0}(\Omega)$ with $\sigma_0(A)$ having compact support in the variable $x$.
			
			Now, if $a \in C_{0}^{\infty}(\Omega \times S^{d-1})$, $\supp a \subset K \times S^{d-1}$, and $a=\sigma_0(A)$, with $K \subset K_{i_0}$, the same argument used to show the independence of the cut-off functions implies
			\begin{equation*}
			\lim_{n \in \mathbb{N}'}(A(\chi u_n),\chi u_n)=	\lim_{n \in \mathbb{N}'}(A(\chi_{i_0} u_n),\chi_{i_0} u_n)=\tilde{L}_{i_0}(\sigma_0(A)),
			\end{equation*}
			where $\chi=1$ in a neighborhood of $K$, $\chi_{i_0}=1$ in a neighborhood of $K_{i_0}$ and that the following inequality holds:
			\begin{equation*}
			|\tilde{L}_{i_0}(\sigma_0(A))|\leq C(\chi_{i_0})\|\sigma_0(A)\|.
			\end{equation*}
			Defining $L(\sigma_0(A))=\tilde{L}_{i_0}(\sigma_0(A))$, we obtain
			\begin{equation*}
			|L(\sigma_0(A))|\leq C(\chi_{i_0})\|\sigma_0(A)\|, \hbox{ if } \supp \sigma_0(A) \subset K \times S^{d-1}.
			\end{equation*}
			
			From Lemma \ref{Lemma 4.48} it follows that, $L$ extends to a Radon measure $\mu\geq 0$ on $\Omega \times S^{d-1}$, which finishes the proof.
		\end{proof}
		
		\begin{Definition}
			$\mu$ is said to be the microlocal defect measure of the sequence $\{u_n\}_{n\in \mathbb{N}'}$ in Theorem \ref{Theorem 4.49}.
		\end{Definition}
		
		\begin{Remark}\label{Rem4.2}
			Theorem \ref{Theorem 4.49} assures that any bounded sequence $(u_n)_{n\in \mathbb{N}}$ in $L_{loc}^2(\mathcal{O})$ that weakly converges to zero has a subsequence associated with a microlocal defect measure (in short, m.d.m.). We observe from (\ref{1.18}) that if $A=f\in C_0^\infty(\mathcal{O})$, then in particular
			\begin{eqnarray}\label{1.26}
			\int_\Omega f(x) |u_{\varphi(n)}(x)|^2\,dx \rightarrow \int_{\mathcal{O} \times S^{d-1}}f(x)\,d\mu(x,\xi),
			\end{eqnarray}
			so that $(u_n)_{\mathbb{N}'}$ strongly converges to $0$ if and only if $\mu=0$.
		\end{Remark}
		
		\begin{Theorem}\label{Theorem 4.56}
			Let $P$ be a differential operator of order $m$ on $\Omega$ with $P^\ast=P$, and let $(u_n)_{n\in \mathbb{N}}$ be a bounded sequence in $L_{loc}^2(\Omega)$ weakly converging to $0$ associated with a microlocal defect measure $\mu$. Let us assume that $P u_n \underset{n\rightarrow +\infty}\longrightarrow 0$ strongly in $H_{loc}^{1-m}$. Then, for any $a \in C^\infty(\Omega \times (\mathbb{R}^d)\backslash\{0\})$ homogeneous of degree $1-m$ in the second variable and with compact support in the first variable,
			\begin{eqnarray}\label{1.27}
			\int_{\Omega \times S^{d-1}}\{a,p\}(x,\xi)\,d\mu(x,\xi)=0.
			\end{eqnarray}
		\end{Theorem}
		\begin{proof}
			Let $\chi, \chi_1 \in C_{0}^{\infty}(\Omega)$ such that $\chi=1$ in a neighborhood of $\pi_1(\supp (a))$ and $\chi_1=1$ in a neighborhood of the support of the function $\chi$. Consider $\gamma \in C^\infty(\Omega)$ with $\gamma=0$ in a neighborhood of $0$ and $\gamma=1$ at infinity. Let $A$ be the operator defined by the symbol $b(x,\xi)=\gamma(\xi)a(x,\xi)$. The symbol of the adjoint operator $A_{\chi_1}^{\ast}$ is given by $b^*(x,\xi)=\overline{b(x,\xi)}+r(x,\xi)$, where $r \in S_{c}^{-m}(\Omega)$.  For a sufficiently regular function $u$, we have
			\begin{eqnarray}\label{1.28}
			\left([A,P](\chi u),\chi u\right)_{L^2(\Omega)}&=&\left( (AP-PA)(\chi u), \chi u\right)_{L^2(\Omega)}\nonumber\\
			&=&\left(AP(\chi u), \chi u\right)_{L^2(\Omega)}-\left( PA(\chi u),\chi u\right)_{L^2(\Omega)}\nonumber\\
			&=&\left(A(P(\chi u)), \chi u \right)_{L^2(\Omega)}-\left(A(\chi u), P^{\ast}(\chi u) \right)_{L^2(\Omega)}\nonumber\\
			&=&\left(A(\chi_1P(\chi u)), \chi u \right)_{L^2(\Omega)}-\left(A(\chi u), P(\chi u) \right)_{L^2(\Omega)}\nonumber\\
			&=&\left(P(\chi u)),A_{\chi_1}^{\ast} (\chi u) \right)_{L^2(\Omega)}-\left(A(\chi u), P(\chi u) \right)_{L^2(\Omega)}\nonumber\\
			&=&\left(P(\chi u)),A_{\chi_1}^{\ast} (\chi u) \right)_{L^2(\Omega)}-\left(A(\chi u), \chi Pu \right)_{L^2(\Omega)}\nonumber\\
			&&-(A(\chi u), [P,\chi]u)_{L^2(\Omega)}.
			\end{eqnarray}
			Note that
			\begin{eqnarray*}
				P(\chi u)&=&\sum_{|\alpha|\leq m}a_\alpha \partial_{x}^{\alpha}(\chi u)\\
				&=& \sum_{|\alpha|\leq m}a_\alpha\sum_{\beta \leq \alpha}\binom{\alpha}{\beta}\partial_{x}^{\beta}\chi \partial_{x}^{\alpha-\beta}u\\
				&=& \chi Pu+ \sum_{|\alpha|\leq m} a_\alpha \sum_{0<\beta\leq \alpha}\binom{\alpha}{\beta} \partial_{x}^{\beta} \chi \partial_{x}^{\alpha-\beta}u\\
				&=&\chi Pu+ [P,\chi]u.
			\end{eqnarray*}
			Therefore, $[P,\chi]$ is a differential operator of order $m-1$ involving the partial derivatives of $\chi$. We note that $A(\chi u)(x)=0$ for $x \not\in \pi_1(\supp (a))$. Thus, we can write
			\begin{eqnarray}\label{1.29}
			(A(\chi u),[P,\chi]u)_{L^2(\Omega)}&=& (A(\chi u),[P,\chi]u)_{L^2(\pi_1(\supp (a)))} \nonumber\\
			&=& 0
			\end{eqnarray}
			since $\chi=1$ in a neighborhood of this set. Therefore, $\partial_{x}^{\beta} \chi=0$ for $\beta>0$.
			
			Regarding the term $(P(\chi u),A_{\chi_1}^{\ast}(\chi u)_{L^2(\Omega)}$, we have
			\begin{equation}\label{1.30}
			(P(\chi u),A_{\chi_1}^{\ast}(\chi u))_{L^2(\Omega)}=(P(\chi u),\overline{b}(x,D)(\chi u))_{L^2(\Omega)}+(P(\chi u),r(x,D)(\chi u))_{L^2(\Omega)}.
			\end{equation}
			By the same argument above,
			\begin{eqnarray}\label{1.31}
			(P(\chi u),\overline{b}(x,D)(\chi u))_{L^2(\Omega)}&=& (\chi Pu+[P,\chi]u,\overline{b}(x,D)(\chi u))_{L^2(\Omega)}\nonumber\\
			&=& (\chi Pu,\overline{b}(x,D)(\chi u))_{L^2(\Omega)}\nonumber\\
			&&+([P,\chi]u,\overline{b}(x,D)(\chi u))_{L^2(\pi_1(\supp(a)))}\nonumber\\
			&=&(\chi Pu,\overline{b}(x,D)(\chi u))_{L^2(\Omega)}.
			\end{eqnarray}
			With respect to the term $(P(\chi u),r(x,D)(\chi u))_{L^2(\Omega)}$, we have,
			\begin{eqnarray*}
				(P(\chi u),r(x,D)(\chi u))_{L^2(\Omega)} &=& (P(\chi u),r(x,D)(\chi_1 \chi u))_{L^2(\Omega)}\\
				&=&(r(x,D)_{\chi_1}^{\ast}P(\chi u),\chi u)_{L^2(\Omega)}\\
				&=&(r(x,D)_{\chi_1}^{\ast}(\chi Pu),\chi u)_{L^2(\Omega)}+(r(x,D)_{\chi_1}^{\ast}([P,\chi]u),\chi u)_{L^2(\Omega)}.
			\end{eqnarray*}
			From the hypothesis, $\chi P u_n \rightarrow 0$ in $H_{comp}^{1-m}(\Omega)$. Then, considering $r(x,D)_{\chi_1}^{\ast}:H_{comp}^{1-m}(\Omega)\longrightarrow H_{K_1}^{1}$, $K_1=\supp \chi_1$, we have $r(x,D)_{\chi_1}^{\ast}(\chi P u_n) \rightarrow 0$ in $L^2(\Omega)$. Thus,
			\begin{equation*}
			|(r(x,D)_{\chi_1}^{\ast}(\chi Pu_n),\chi u_n)_{L^2(\Omega)}|\leq \|r(x,D)_{\chi_1}^{\ast}(\chi Pu_n)\|_{L^2(\Omega)}\|\chi u_n\|_{L^2(\Omega)} \rightarrow 0, \hbox{ as } n \rightarrow \infty.
			\end{equation*}
			Note that $[P,\chi]$ has order $m-1$ and is compactly supported, that is, has compact kernel. Thus, $[P,\chi]:L_{loc}^{2}(\Omega)\longrightarrow H_{comp}^{1-m}$. By Rellich's theorem,
			\begin{equation*}
			r(x,D)_{\chi_1}^{\ast} \circ [P,\chi]: L_{loc}^{2}(\Omega) \longrightarrow L^2(\Omega)
			\end{equation*}
			is a compact operator. Therefore, $r(x,D)_{\chi_1}^{\ast} \circ [P,\chi]$ maps weakly convergent sequences to strongly convergent sequences. Hence,
			\begin{eqnarray*}
				|(r(x,D)_{\chi_1}^{\ast}  [P,\chi](u_n),\chi u_n)_{L^2(\Omega)}|&\leq& \|(r(x,D)_{\chi_1}^{\ast}  [P,\chi](u_n)\|_{L^2(\Omega)}\|\chi u_n\|_{L^2(\Omega)}\\
				&\rightarrow& 0, \hbox{ as } n \rightarrow \infty.
			\end{eqnarray*}
			Hence,
			\begin{equation}\label{pchiun}
			|(P(\chi u_n),r(x,D)(\chi u_n))_{L^2(\Omega)}| \rightarrow 0, \hbox{ as } n \rightarrow \infty.
			\end{equation}
			From \eqref{1.28}, \eqref{1.29} and \eqref{1.30}, we obtain
			\begin{eqnarray*}
				([A,P](\chi u_n),\chi u_n)_{L^2(\Omega)}&=& (\chi Pu_n,\overline{b}(x,D)(\chi u_n))_{L^2(\Omega)}+(P(\chi u_n),r(x,D)(\chi u_n))_{L^2(\Omega)}\\
				&&-(A(\chi u_n),\chi Pu_n)_{L^2(\Omega)}.
			\end{eqnarray*}
			
			Observing that
			\begin{eqnarray*}
				\sigma_0([A,P])=\frac{1}{
					i}\{a,p\},~\hbox{ where }a=\sigma_{1-m}(A)\hbox{ and }p=\sigma_m(P),
			\end{eqnarray*}
			the second term of (\ref{1.27}) is
			\begin{eqnarray*}
				J:=\lim_{n\rightarrow +\infty} i\left([A,P](\chi u_n), \chi u_n \right)_{L^2(\Omega)}&=& i\lim_{n\rightarrow +\infty}(\chi Pu_n,\overline{b}(x,D)(\chi u_n))_{L^2(\Omega)}\\
				&&+i\lim_{n\rightarrow +\infty}(P(\chi u_n),r(x,D)(\chi u_n))_{L^2(\Omega)}\\
				&&i i\lim_{n\rightarrow +\infty}(A(\chi u_n),\chi Pu_n)_{L^2(\Omega)}.
			\end{eqnarray*}
			
			Combining the fact that $\chi P u_n \rightarrow 0$ in $H^{1-m}_{comp}(\Omega)$ with $\overline{b}(x,D)_{\chi_1}^{\ast}:H_{comp}^{1-m}(\Omega) \longrightarrow L^2(\Omega)$, we obtain
			\begin{eqnarray*}
				|(\chi Pu_n,\overline{b}(x,D)(\chi u_n))_{L^2(\Omega)}|&= &|(\overline{b}(x,D)_{\chi_1}^{\ast}(\chi Pu_n),\chi u_n)_{L^2(\Omega)}|\\
				&\leq& \|\overline{b}(x,D)_{\chi_1}^{\ast}(\chi Pu_n)\|_{L^2(\Omega)}\|\chi u_n\|_{L^2(\Omega)}\\
				&\rightarrow& 0, \hbox{ as } n \rightarrow \infty.
			\end{eqnarray*}
			Due to \eqref{pchiun}, it is necessary to prove that the term $|(A(\chi u_n),\chi Pu_n)_{L^2(\Omega)}|$ converges to $0$. In fact, $A:L_{comp}^{2}(\Omega) \longrightarrow H_{K}^{m-1}(\Omega)$ is a linear continuous operator. Therefore, $A$ maps bounded sets to bounded sets. In particular, $\{A(\chi u_n)\}$ is a bounded set of $H_{K}^{m-1}(\Omega)$. Therefore,
			\begin{eqnarray*}
				|(A(\chi u_n),\chi Pu_n)_{L^2(\Omega)}|&\leq & \|A(\chi u_n)\|_{H^{m-1}_{K}(\Omega)}\|\chi Pu_n\|_{H^{1-m}_{K}(\Omega)}\\
				&\rightarrow& 0, \hbox{ as } n \rightarrow \infty.
			\end{eqnarray*}
			Hence,
			\begin{eqnarray*}
				\int_{\Omega \times S^{d-1}}\{a,p\}(x,\xi)\,d\mu(x,\xi)=\lim_{n\rightarrow \infty}([A,P](\chi u_n),\chi u_n)_{L^2(\Omega)}=0,
			\end{eqnarray*}
			which concludes the proof.
		\end{proof}

		\begin{Theorem}\label{invariante0}
			Let $X$ be a locally compact Hausdorff space and $\mu$ be a positive Radon measure on $X$.
			\begin{itemize}
				\item [a)] Let $N$ be the union of all open $U \subset X$ such that $\mu(U)=0$. Then $N$ is open and $\mu(N)=0$. The complement of $N$ is called the support of $\mu$.
				\item [b)] $x \in \supp \mu$ if and only if $\int_{X} f d\mu>0$ for every $f \in C_{0}(X)$ with $f(x)>0$.
			\end{itemize}
		\end{Theorem}
		\begin{proof}
			a) clearly follows from $\mu(N)=\sup\{\mu(K):K \subset N, K \hbox{ compact}\}.$
			
			To prove b), suppose $x \in \supp \mu=N^c$. Let $f \in C_c(X,[0,1])$ such that $f(x)>0$. By continuity, there is an open neighborhood $U$ of $x$ such that $f(y)>\frac{1}{2}f(x)$ for all $y \in U$. If $\mu(U)=0$, then since $U$ is open, $U \subset N$, which contradicts $x \in N^c$. Thus, $\mu(U)>0$. Hence, $\int_{X}fd\,\mu\geq \int_U f\, d\mu\geq \frac{1}{2}f(x)\mu(U)>0$.
			
			Now, suppose $x \notin \supp \mu$ and let $F=\supp \mu$. Then, there exists an open neighbourhood $U$ of $x$ such that $\mu(U)=0$ and $U \cap F=\emptyset$. Let $K$ be a compact set so that $x \in K \subset U$. From Urysohn's theorem, there exists a function $f \in C_0(X,[0,1])$ such that $f=1$ in $K$ and $f=0$ outside of a compact subset of $U$. In particular, $f(x)=1>0$ and
			\begin{eqnarray*}
				\int_{X} f d\mu &=& \int_{U} f d \,\mu +\int_{X \setminus U} f \, d\mu\\
				&=& \int_{U}f d \, \mu+\int_{X \setminus U} 0 \, d\mu\\
				&=&0.
			\end{eqnarray*}
		\end{proof}
		\begin{Theorem}\label{suporte}
			Let $\{u_n\}_n$ be a bounded sequence in $L_{loc}^2(\Omega)$ that weakly converges to zero and admits a microlocal defect measurement $\mu$. Then, $(x_0,\xi_0) \notin \supp \mu$ if and only if there exists $A \in \Psi_{c}^{0}(\Omega)$ essentially homogeneous such that $\sigma_0(A)(x_0,\xi_0)\neq 0$ and $A(\chi u_n) \rightarrow 0$ in $L^2(\Omega)$ for all $\chi \in C_{0}^{\infty}(\Omega)$.
		\end{Theorem}
		\begin{proof}
			If $(x_0,\xi_0) \notin \supp \mu$, then there exist open subsets $U \subset \Omega$, $V \subset S^{d-1}$  such that $(x_0,\xi_0) \in U \times V$ and $U \times V \cap \supp \mu = \emptyset$. Let $K$ be a compact neighborhood of $\xi_0$, such that $K \subset V$. Let $\psi \in C^\infty(S^{d-1})$ such that $\supp \psi \subset K$ and $\psi=1$ in a neighborhood of $\xi_0$ contained in $K$. Consider $\phi \in C_{0}^{\infty}(\Omega)$ such that $\phi =1$ in a neighborhood of $x_0$ and $\supp \phi \subset U$. Pick $\eta \in C^\infty(\Omega)$ such that $\eta=0$ near $0$ and $\eta=1$ at infinity. Define $a(x,\xi)= \eta(\xi)\phi(x)\psi\left(\frac{\xi}{|\xi|}\right)$. Let $A$ be the operator defined by $a(\cdot,\cdot)$. Note that $\sigma_0(A)(x_0,\xi_0)=1$ and for all $\chi \in C_{0}^{\infty}(\Omega)$
			\begin{eqnarray}\label{1.32}
			\|A(\chi u_n)\|_{L^{2}(\Omega)}^{2} &=& (A(\chi u_n), A(\chi u_n ))\nonumber \\
			&=&(A(\chi u_n), A(\chi_1 \chi u_n ))\nonumber \\
			&=&(A_{\chi_1}^{*}\circ A (\chi u_n),\chi u_n )\nonumber \\
			&\rightarrow& \int_{\Omega\times S^{d-1} }|\sigma_0(A)(x,\xi)|^2 d \mu,
			\end{eqnarray}
			where $\chi_1 =1$ in a neighborhood of $\supp \chi$. Since $\supp \sigma_0(A) \subset \supp \phi \times \supp \psi \subset \Omega \times S^{d-1} \setminus \supp \mu$, we have
			\begin{equation*}
			\int_{\Omega \times S^{d-1}}|\sigma_0(A)|^2 d \mu=0.
			\end{equation*}
			Therefore, $\|A(\chi u_n)\|_{L^{2}(\Omega)}\rightarrow 0$ in $L^2(\Omega)$.
			
			Conversely, if there exists $A \in \Psi_{c}^{0}(\Omega)$ essentially homogeneous such that $\sigma_0(A)(x_0,\xi_0)\neq 0$ and $A(\chi u_n) \rightarrow 0$ in $L^2(\Omega)$, from \eqref{1.32} and by the uniqueness of the limit, we obtain
			\begin{equation*}
			\int_{\Omega \times S^{d-1}}|\sigma_0(A)|^2 d \mu=0.
			\end{equation*}
			By using Theorem \ref{invariante0}, we deduce that $(x_0,\xi_0) \notin \supp \mu$.
		\end{proof}
		
		We will call the vector field defined in $\Omega \times {\mathbb{R}}^d\backslash\{0\}$ and given by
		$$H_p(x,\xi)= \left(\frac {\partial p} { \partial \xi_1} (x,\xi), \cdots, \frac {\partial p} {\partial \xi_n}  (x,\xi); -\frac {\partial p} {\partial x_1} (x,\xi), \cdots, -\frac {\partial p} {\partial x_n} (x,\xi)\right)$$
		a Hamiltonian field of $p$, denoted by $H_p$.
		
		The Lie derivative of a function $f$ with respect to the Hamiltonian field $H_p$ is given by $H_p (f)=\{p,f\}$, where
		\begin{eqnarray*}
			\{p,f\}(x,\xi)=\sum_{j=1}^d \left(\frac{\partial p}{\partial \xi_j}\frac{\partial f}{\partial x_j}-\frac{\partial p}{\partial x_j}\frac{\partial f}{\partial \xi_j}\right).
		\end{eqnarray*}
		
		A Hamiltonian curve of $p$ is an integrable curve of the vector field $H_p$, which is a maximal solution $s\in I \mapsto (x(s),~\xi(s))$ to the Hamilton-Jacobi equations
		\begin{equation}\label{1.33}
		\left\{
		\begin{aligned}
		& \dot{x}=p_\xi(x,\xi)=\frac{\partial p}{\partial \xi}(x,\xi),\quad \dot{\xi}=-p_x(x,\xi)=- \frac{\partial p}{\partial x},\\
		\end{aligned}
		\right.
		\end{equation}
		where $I$ is an open interval of $\mathbb{R}$. By null bicharacteristics we mean those integral curves of $H_p$ along which $p=0$.

		\begin{Remark}\label{v11}
			Since $H_p p=0$, $p$ is constant along the integral curves of $H_p$. Indeed, let $\alpha(t)=(H_p)_t(x_0,\xi_0)$ be an integral curve of $H_p$ starting at $(x_0,\xi_0)$. Then, \begin{eqnarray*}
				0&=&H_p p (\alpha(t))\\
				&=& d(p)_{\alpha(t)}H_p(\alpha(t))\\
				&=&d(p)_{\alpha(t)}\alpha'(t)\\
				&=&(p \circ \alpha)'(t).
			\end{eqnarray*}
			Now, from the connectedness of the domain of $\alpha$, the remark follows.
		\end{Remark}
		
		\begin{Lemma}\label{v12}
			Assume $p(x_0,\xi_0)=0$ and $\lambda$ is a $C^\infty$ function on $T^0\Omega$ with real values that never vanishes. Then, for all $t \in (T_0,T_1)$ there exists $s_t \in \mathbb{R}$ such that $(H_p)_t(x_0,\xi_0)=(H_{\lambda p})_{s_t}(x_0,\xi_0)$. Moreover, there exists $S_0<0<S_1$ for which the mapping $t \mapsto s_t$ is a $C^\infty$ diffeomorphism from $(T_0,T_1)$ onto $(S_0,S_1)$.
		\end{Lemma}
		\begin{proof}
			Let us write $b=\frac{1}{\lambda}$ and $a=\lambda p$. Since $p=ab$, using Remark \ref{v11} we have \linebreak $(ab)((H_p)_t(x_0,\xi_0)=0$. Since $b$ never vanishes,  we obtain $a((H_p)_t(x_0,\xi_0))=0$ for all $t \in (T_0,T_1)$. Thus,
			\begin{equation*}
			\left\{
			\begin{aligned}
			& \dot{x}(t)=\frac{\partial p}{\partial \xi}=\frac{\partial a}{\partial \xi}b+\frac{\partial b}{\partial \xi}a=b\frac{\partial a}{\partial \xi}(x(t),\xi(t)),\\
			&\dot{\xi}(t)=-\frac{\partial p}{\partial x}=-\frac{\partial a}{\partial x}b-\frac{\partial b}{\partial x}a=-b\frac{\partial a}{\partial x}(x(t),\xi(t)).
			\end{aligned}
			\right.
			\end{equation*}
			Writing $f(t)=b(x(t),\xi(t))$, we have
			\begin{equation*}
			\left\{
			\begin{aligned}
			& \dot{x}(t)=f(t)\frac{\partial a}{\partial \xi}(x(t),\xi(t)),\\
			&\dot{\xi}(t)=-f(t)\frac{\partial a}{\partial x}(x(t),\xi(t)),
			\end{aligned}
			\right.
			\end{equation*}
			with $x(0)=x_0$ and $\xi(0)=\xi_0$. Now, let $(y(t),\eta(t))=(H_a)_t(x_0,\xi_0)=(H_{\lambda p})_t(x_0,\xi_0)$ and define
			\begin{equation*}
			s(t)=\int_{0}^{t}f(s)ds.
			\end{equation*}
			By virtue of the mean value theorem and the fact that $f$ never vanishes, we have either $f>0$ or $f<0$. Let's see that $s$ is injective. Indeed, let $t_1,t_2 \in (T_0,T_1)$ such that $t_1 \neq t_2 $ and suppose by contradiction that $s(t_1)=s(t_2)$. Without loss of generality, we can suppose $t_1<t_2$. Then
			\begin{equation*}
			\int_{0}^{t_1}f(s)ds=\int_{0}^{t_2}f(s)ds=\int_{0}^{t_1}f(s)ds+\int_{t_1}^{t_2}f(s)ds.
			\end{equation*}
			Therefore,
			\begin{equation*}
			\int_{t_1}^{t_2}f(s)ds=0.
			\end{equation*}
			In the same fashion, if $f>0$, then $\int_{t_1}^{t_2}f(s)ds>0$ and if $f<0$, then $\int_{t_1}^{t_2}f(s)ds<0$, which yields a contradiction in both cases. Therefore, $s(t_1)=s(t_2)$. Moreover, we have $s'(t)=f(t)\neq0$ for all $t\in (T_0,T_1)$. Hence, $s$ is a $C^\infty$ injective local diffeomorphism. Since $s(0)=0$, there exists an open neighbourhood of $0$, $(S_0,S_1)$, such that $s$ is a diffeomorphism from $(T_0,T_1)$ onto $(S_0,S_1)$. Let $(z(t),\gamma(t))=(y(s(t)),\eta(s(t)))$ and note that
			\begin{equation*}
			\left\{
			\begin{aligned}
			&\dot{z}(t)=\dot{y}(s(t))s'(t)=f(t)\frac{\partial a}{\partial \xi}(y(s(t)),\eta(s(t)))=f(t)\frac{\partial a}{\partial \xi}(z(t),\gamma(t))\\
			&\dot{\gamma}(t)=\dot{\eta}(s(t))s'(t)=-f(t)\frac{\partial a}{\partial x}(y(s(t)),\eta(s(t)))=-f(t)\frac{\partial a}{\partial x}(z(t),\gamma(t)),
			\end{aligned}
			\right.
			\end{equation*}
			with $z(0)=y(s(0))=y(0)=x_0$ and $\gamma(0)=\eta(s(0))=\eta(0)=\xi_0$.
			By virtue of uniqueness of solutions, we obtain $(x(t),\xi(t))=(y(s(t)),\eta(s(t))$. Therefore, the bicharacteristics of $\lambda p$  and $p$ coincide modulo a reparametrization.
		\end{proof}
		
		\begin{Theorem}\label{Theorem 4.55}
			Let $P$ be a differential operator of order $m$ on $\Omega$ and let $(u_n)_{n\in \mathbb{N}}$ be a bounded sequence in $L_{loc}^2(\Omega)$ weakly converging to $0$, associated with a microlocal defect measure $\mu$. Then, the following statements are equivalent:
			\begin{itemize}
				\item [(i)] $Pu_n \underset{n \rightarrow +\infty}\longrightarrow 0 \hbox{ strongly in }H_{loc}^{-m}(\Omega)~(m>0)$,
				\item [(ii)]	$\supp (\mu) \subset \{(x,\xi)\in \Omega \times S^{d-1}: \sigma_m(P)(x,\xi)=0\}.$
			\end{itemize}
		\end{Theorem}
		\begin{proof}
			Observe that $(i)$ is equivalent to $\chi Pu_{n} \rightarrow 0$ in $H^{-m}(\Omega)$ for every $\chi \in C_{0}^{\infty}(\Omega)$. However, $P$ is a
			properly supported operator, since it is a differential operator, so that there exists $\psi \in C_{0}^{\infty}(\Omega)$ such that $\chi Pu_{n} = \chi P(\psi u_{n})$ for every $n \in \mathbb{N}$. Therefore, $(i)$ is also equivalent to
			\begin{eqnarray*}
				(\Lambda^{-m}(\chi P(\psi u_{n})), \Lambda^{-m}(\chi P(\psi u_{n}) ))_{L^{2}} & = & (\Lambda^{-2m}\chi P(\psi u_{n}), \chi P(\psi u_{n}))_{L^{2}} \\
				& = & (P^{\ast} \overline{\chi}\Lambda^{-2m}\chi P( \psi u_{n}), \psi u_{n})_{L^{2}} \to 0.																						
			\end{eqnarray*}
			If we set $B = P^{\ast}\overline{\chi}\Lambda^{-m}\chi P \in \Psi_{c}^{0}(\Omega)$, then the principal symbol of $B$ is $$\sigma_{0}(B) = |\chi(x)|^{2}|\sigma_{m}(P)(x, \xi)|^{2}|\xi|^{-2m}$$  and
			\begin{eqnarray*}
				(\Lambda^{-m}(\chi P(\psi u_{n})), \Lambda^{-m}(\chi P(\psi u_{n}) ))_{L^{2}} \rightarrow \int_{\Omega \times S^{d-1}}{|\chi(x)|^{2}|\sigma_{m}(P)(x, \xi)|^{2}}d\mu(x, \xi).
			\end{eqnarray*}
			Employing Theorem \ref{invariante0} we conclude that the condition
			\begin{eqnarray*}
				\int_{\Omega \times S^{d-1}}{|\chi(x)|^{2}|\sigma_{m}(P)(x, \xi)|^{2}}d\mu = 0
			\end{eqnarray*}
			for every $\chi \in C_{0}^{\infty}(\Omega)$ is equivalent to $(ii)$, which completes the proof.
		\end{proof}

		\begin{Theorem}\label{Theorem 4.60} Let $P$ be a self-adjoint differential operator of order $m$ on $\Omega$ and let $(u_n)_{n\in \mathbb{N}}$ be a bounded sequence in $L_{loc}^2(\Omega)$ that weakly converges to $0$, with a microlocal defect measure $\mu$. Suppose that $P u_n$ converges to $0$ in $H_{loc}^{-(m-1)}$. Then the support of $\mu$, $\supp(\mu)$, is a union of curves like $s\in I \mapsto \left(x(s), \frac{\xi(s)}{|\xi(s)|}\right)$, where $s\in I \mapsto (x(s),\xi(s))$ is a null-bicharacteristic of $p,$ where $p$ is the principal symbol of $P.$
		\end{Theorem}
		\begin{proof}We first consider the function $q(x,\xi)=|\xi|^{1-m}p(x,\xi),$ which is smooth on $\Omega\times(\mathbb{R}^d\setminus\{0\})$ and homogeneous of degree $1$ in the variable $\xi.$
			
			We have already noticed that the null-bicharacteristics of $q$ are reparametrizations of the null-bicharacteristics of $p.$ Hence, it is enough to prove that $\supp(\mu)$ is a union of curves like $s\in I \mapsto \left(x(s), \frac{\xi(s)}{|\xi(s)|}\right)$, where $s\in I \mapsto (x(s),\xi(s))$ is a null-bicharacteristic of $q.$
			
			Notice that $\Omega\times S^{d-1}$ is covered by the bicharacteristics of $q$ (that is, the integral curves of the Hamiltonian $H_q$). Since $H^{1-m}_{loc}(\Omega)$ is continuously included in $H^{-m}_{loc}(\Omega),$ a previous result implies that $\supp(\mu)$ is a subset of $\{(x,\xi)\in\Omega\times S^{d-1}:p(x,\xi)=0\}=\{(x,\xi)\in\Omega\times S^{d-1}:q(x,\xi)=0\}.$
			
			If $s\in I \mapsto (x(s),\xi(s))$ is a bicharacteristic in $\Omega\times(\mathbb{R}^d\setminus\{0\}),$ in which $q$ never vanishes, then the homogeneity of $q$ in the second variable implies that $q(x(s),\xi(s)/|\xi(s)|)\neq0$ for all $s.$ Hence, the curve $s\in I \mapsto (x(s),\xi(s)/|\xi(s)|)$ never touches $\supp(\mu).$ It follows that $\supp(\mu)$ is a subset of the union of curves like $s\in I \mapsto \left(x(s), \frac{\xi(s)}{|\xi(s)|}\right)$, where $s\in I \mapsto (x(s),\xi(s))$ is a null-bicharacteristic of $q.$
			
			To complete the proof, we must show that, for a null-bicharacteristic of $q,$ $(x(s),\xi(s)),$ defined on an interval $I,$ such that $\left(x(s), \frac{\xi(s)}{|\xi(s)|}\right)\in\supp(\mu)$ for some $s_0\in I,$ we have $\left(x(s), \frac{\xi(s)}{|\xi(s)|}\right)\in\supp(\mu)$ for all $s\in I.$
			
			We first notice that it is enough to consider a local version of the above assertion. Indeed, the set $A=\left\{t\in I: \left(x(s), \frac{\xi(s)}{|\xi(s)|}\right)\in\supp(\mu)\right\}$ is closed in $I.$ Moreover, $A$ is open in $I$ if for each $s_0\in A,$ there exists $\varepsilon>0$ such that $\left(x(s), \frac{\xi(s)}{|\xi(s)|}\right)\in\supp(\mu),$ for all $s\in(s_0-\varepsilon,s_0+\varepsilon)\cap I.$ In this case, we have $A=I$ whenever $A\neq\emptyset.$
			
			By the remarks in the above paragraphs, the proof reduces to show that for each $s_0\in I$ such that $\left(x(s_0), \frac{\xi(s_0)}{|\xi(s_0)|}\right)\in\supp(\mu),$ there exists $\varepsilon>0$ such that $\left(x(s), \frac{\xi(s)}{|\xi(s)|}\right)\in\supp(\mu)$ for all $s\in(s_0-\varepsilon,s_0+\varepsilon)\cap I.$
			
			We will prove this by contradiction.
			
			Assume that $s_0\in I$ is such that $\left(x(s_0), \frac{\xi(s_0)}{|\xi(s_0)|}\right)\in\supp(\mu)$ and for all $\varepsilon>0$ we find  $s^{\ast}\in(s_0-\varepsilon,s_0+\varepsilon)\cap I$ such that
			$\left(x(s^{\ast}), \frac{\xi(s^{\ast})}{|\xi(s^{\ast})|}\right)\not\in\supp(\mu).$ Without loss of generality, we may assume $s^\ast\in(s_0,s_0+\varepsilon).$
			
			Let $(x_0,\xi_0)=(x(s_0),\xi(s_0))$ and $(x_0^{\ast},\xi_0^\ast)=(x(s^\ast),\xi(s^\ast)).$ By semigroup properties, we have $(x_0^{\ast},\xi_0^\ast)=(x(s^\ast-s_0,x_0,\xi_0),\xi(s^\ast-s_0,x_0,\xi_0)),$ where $(x(s,x,\xi),\xi(s,x,\xi))$ denote the bicharacteristics that satisfy $(x(0,x,\xi),\xi(0,x,\xi))=(x,\xi).$ Notice that we also have $(x_0,\xi_0)=(x(s_0-s^\ast,x_0^\ast,\xi_0^\ast),\xi(s_0-s^\ast,x_0^\ast,\xi_0^\ast)).$
			
			Since we are working only locally, we can assume that we are working on a connected component of $\Omega.$ Moreover, by choosing $\varepsilon>0$ sufficiently small, we may assume that $x_0^\ast\in B(x_0,r)\subset B[x_0,4r]\subset\Omega$ and that the bicharacteristics through the points on a compact set $B[x_0,4r]\times S^{d-1}\subset \Omega\times (\mathbb{R}^d\setminus\{0\})$ are all defined on a same interval $(-\varepsilon,\varepsilon).$ Such interval contains $s^\ast-s_0.$

			We claim that, by the continuous dependence on initial data, we may choose $\varepsilon>0$ so that, for all $s\in(-\varepsilon,\varepsilon)$ with $r\leq1/6$ sufficiently small and for all $(x,\xi)\in [B[x_0,4r]\setminus B(x_0,3r)]\times S^{d-1}$ we have $(x(s,x,\xi),\xi(s,x,\xi))\in (\Omega\setminus B[x_0,r])\times (\mathbb{R}^d\setminus\{0\}).$	
			
			Indeed, for each $(\tilde{x},\tilde{\xi})\in [B[x_0,4r]\setminus B(x_0,3r)]\times S^{d-1},$ the solution $(x(s,\tilde{x},\tilde{\xi}),\xi(s,\tilde{x},\tilde{\xi}))$ is contained in $(\Omega\setminus B(x_0,11r/4))\times (\mathbb{R}^d\setminus B(0,3r))$ for $s$ belonging to a sufficiently small interval $I_{\tilde{x},\tilde{\xi}}=[-\delta_{\tilde{x},\tilde{\xi}},\delta_{\tilde{x},\tilde{\xi}}].$ Recall that the continuous dependence on initial data in the system $\dot{x}=\partial q/\partial\xi$ and $\dot{\xi}=-\partial q/\partial x$ implies that for each $\varepsilon'>0,$ there exists $\delta=\delta(\delta_{\tilde{x},\tilde{\xi}},\varepsilon')>0$ such that for all $(x,\xi)$ satisfying $\|(x,\xi)-(\tilde{x},\tilde{\xi})\|_{\max}<\delta,$ we have $(x(s,x,\xi),\xi(s,x,\xi))$ defined on $I_{\tilde{x},\tilde{\xi}}$ and \[\|\left(x(s,x,\xi),\xi(s,x,\xi)\right)-(x(s,\tilde{x},\tilde{\xi}),\xi(s,\tilde{x},\tilde{\xi}))\|_{\max}<\varepsilon',\] for all $s\in I_{\tilde{x},\tilde{\xi}}.$ Hereafter, we assume that $\varepsilon_1<r/2$.	For each $s \in I_{\tilde{x},\tilde{\xi}}$ we define
			$$A_{\tilde{x},\tilde{\xi}}^{s}=\{(x,\xi) \in \Omega \times \mathbb{R}^{d}\setminus\{0\}: \|(x,\xi)-(x(s,\tilde{x},\tilde{\xi}),\xi(s,\tilde{x},\tilde{\xi}))\|<\eta \},$$
			where $\eta<\min\{\delta,r/4\}$.
			We also set
			$$A_{\tilde{x},\tilde{\xi}}=\bigcup_{s \in I_{\tilde{x},\tilde{\xi}}} A_{\tilde{x},\tilde{\xi}}^{s}$$
			and
			$$A=\bigcup_{(\tilde{x},\tilde{\xi}) \in B[x_0,4r]\setminus B(x_0,3r) \times S^{d-1}}A_{\tilde{x},\tilde{\xi}}.$$
			
			Note that the union of the sets $A_{\tilde{x},\tilde{\xi}}$ with $(\tilde{x},\tilde{\xi}) \in B[x_0,4r]\setminus B(x_0,3r) \times S^{d-1}$ does not intersect $B(x_0,5r/2)\times B(0,11r/4)$. Hence by virtue of the compactness of  $B[x_0,4r]\setminus B(x_0,3r) \times S^{d-1}$ there exist $(\tilde{x}_1,\tilde{\xi}_1),\dots, (\tilde{x}_N,\tilde{\xi}_N) \in B[x_0,4r]\setminus B(x_0,3r) \times S^{d-1}$ such that
			$$B[x_0,3r]\setminus B(x_0,2r) \times S^{d-1} \subset \bigcup_{i=1}^{N} A_{\tilde{x}_i,\tilde{\xi}_i}.$$
			Therefore, given $(x_1,\xi_1) \in B[x_0,4r]\setminus B(x_0,3r) \times S^{d-1}$, there exists $i_0 \in \{1,\dots, N\}$ such that
			$$(x_1,\xi_1) \in A_{\tilde{x}_{i_0},\tilde{\xi}_{i_0}}=\bigcup_{s \in I_{\tilde{x}_{i_0},\tilde{\xi}_{i_0}  }}  A_{\tilde{x}_{i_0},\tilde{\xi}_{i_0}}^{s}. $$
			Thus, there exists $s_0 \in  I_{\tilde{x}_{i_0},\tilde{\xi}_{i_0}}$ such that
			$$\|(x_1,\xi_1) -(x(s_0,\tilde{x}_{i_0},\tilde{\xi}_{i_0} ),\xi(s_0,\tilde{x}_{i_0},\tilde{\xi}_{i_0}))\|<\eta \leq \delta.$$
			Using the continuous dependence on data we deduce that
			$$\|(x(s,x_1,\xi_1),\xi(s,x_1,\xi_1))-(x(s+s_0,\tilde{x}_{i_0},\tilde{\xi}_{i_0}),\xi(s+s_0,\tilde{x}_{i_0},\tilde{\xi}_{i_0}))\|<\varepsilon_1<\frac{r}{2}$$
			for all $s \in  I_{\tilde{x}_{i_0},\tilde{\xi}_{i_0}}$ such that $s+s_0 \in I_{\tilde{x}_{i_0},\tilde{\xi}_{i_0}}$.
			Thus, for all $(a,b)\in B[x_0,r] \times B[0,r]$, we obtain
			\begin{equation*}
			\begin{gathered}
			\|(a,b) - (x(s,x_1,\xi_1),\xi(s,x_1,\xi_1))\| = 	\|(a,b) - (x(s+s_0,\tilde{x}_{i_0},\tilde{\xi}_{i_0}),\xi(s+s_0,\tilde{x}_{i_0},\tilde{\xi}_{i_0}))\|\\
			- \|(x(s+s_0,\tilde{x}_{i_0},\tilde{\xi}_{i_0}),\xi(s+s_0,\tilde{x}_{i_0},\tilde{\xi}_{i_0})) -(x(s,x_1,\xi_1),\xi(s,x_1,\xi_1))\| \\
			\geq \frac{3r}{2}-\frac{r}{2}=r \hbox{ for all }	s \in  I  \hbox{ such that } s+s_0 \in I,
			\end{gathered}
			\end{equation*}
			where $I$ denotes the smallest of the intervals $I_{\tilde{x}_i, \tilde{\xi}_i}$, for $i \in \{1,\dots, N\}$.
			The situation described above is illustrated in Figure \ref{Figure}.
			\begin{figure}[H]
				\centering
				\includegraphics[scale=0.6]{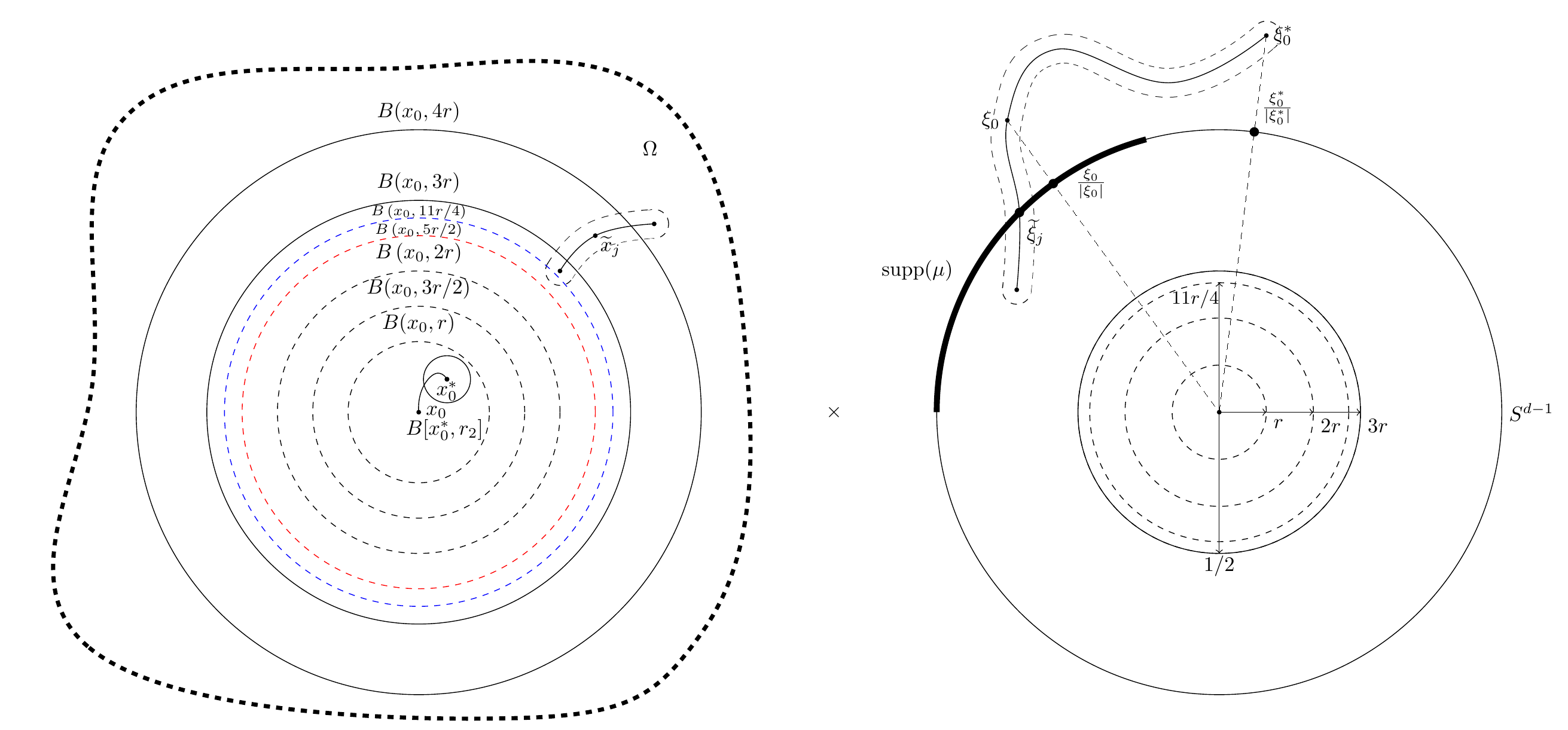}			
				\caption{\small Proof diagram}\label{Bicharacteristics}
				\label{Figure}
			\end{figure}

			Since $\left(x_{0}^{\ast}, \frac{\xi_{0}^{\ast}}{|\xi_0^{\ast}|}\right)\not\in\supp(\mu),$ it follows from Theorem \ref{invariante0} that there exists $g\in C^{\infty}_{0}(\Omega\times S^{d-1},[0,1])$ such that $g\left(x_0^{\ast}, \frac{\xi_0^{\ast}}{|\xi_0^{\ast}|}\right)>0$ and $\int_{\Omega\times S^{d-1}}gd\mu=0.$
			
			Let $\phi\in C^{\infty}_0(\Omega)$ be such that $\supp(\phi)\subset B[x_0^\ast,r_2]\subset B(x_0,r),$ $\phi(x_0^\ast)=1$ and $0\leq\phi\leq1.$
			
			We then define $\tilde{g}(x,\xi)=|\xi|^{1-m}\phi(x)g\left(x,\xi/|\xi|\right),$ $(x,\xi)\in\Omega\times (\mathbb{R}^{d}\setminus\{0\}).$ It follows that $\tilde{g}$ is smooth on $\Omega\times (\mathbb{R}^{d}\setminus\{0\}),$ homogeneous of degree $1-m$ in the variable $\xi,$ and $\tilde{g}$ has compact support in the variable $x.$ Notice that \[0\leq\int_{\Omega\times S^{d-1}}\tilde{g}d\mu=
			\int_{\Omega\times S^{d-1}}\phi(x)g(x,\xi)d\mu\leq
			\int_{\Omega\times S^{d-1}}gd\mu=0.\] Hence, $\int_{\Omega\times S^{d-1}}\tilde{g}d\mu=0.$
			
			Next, we use the bicharacteristics to bring the information given by $g$ in a neighborhood of $\left(x_{0}^{\ast}, \frac{\xi_{0}^{\ast}}{|\xi_0^{\ast}|}\right)$ to a neighborhood of $\left(x_0, \frac{\xi_0}{|\xi_0|}\right)$.
			
			
			
			

			

			For $(x,\xi)\in B(x_0,4r)\times S^{d-1},$ we define $f(x,\xi)=\tilde{g}(x(s^\ast-s_0,x,\xi),\xi(s^\ast-s_0,x,\xi)).$ It follows that $f$ is well-defined since we previously established that the bicharacteristics \linebreak$(x(s,x,\xi),\xi(s,x,\xi))$ are defined on a same interval which contains $s^\ast-s_0,$ for all $(x,\xi)\in B(x_0,4r)\times S^{d-1}.$ Notice that $f$ extends to a continuous function on $\Omega\times S^{d-1}$ since $f(x,\xi)=0$ for all $(x,\xi)\in(\Omega\setminus B[x_0,3r])\times S^{d-1}$ (recall that $(x(s^\ast-s_0,x,\xi),\xi(s^\ast-s_0,x,\xi))$ belongs to $(\Omega\setminus B[x_0,r])\times (\mathbb{R}^d\setminus\{0\}),$  for all $(x,\xi)\in(\Omega\setminus B[x_0,3r])\times S^{d-1}$). It follows that $f\in C_0(\Omega\times S^{d-1},[0,1]).$ In order to complete the proof, we will show that $f(x_0,\xi_0/|\xi_0|)>0$ and $\int_{\Omega\times S^{d-1}}fd\mu=0.$ Note that, by Theorem \ref{invariante0}, this implies a contradiction, since $(x_0,\xi_0/|\xi_0|)\in\supp(\mu).$
			
			Before we proceed, we note that for $\lambda>0$ we have \begin{equation}\label{1.34}(x(s,x,\lambda\xi),\xi(s,x,\lambda\xi))=(x(s,x,\xi),\lambda\xi(s,x,\xi))\end{equation} for all $(x,\xi)\in\Omega\times(\mathbb{R}^{d}\setminus\{0\}).$ Indeed, since $q$ is homogeneous of degree one in the variable $\xi,$ it follows that $\partial q/\partial\xi$ is homogeneous of degree zero and $\partial q/\partial x$ is homogeneous of degree one. Hence, \[\frac{d}{ds}x(s,x,\xi)=\frac{\partial q}{\partial\xi}(x(s,x,\xi),\xi(s,x,\xi))=\frac{\partial q}{\partial\xi}(x(s,x,\xi),\lambda\xi(s,x,\xi))\] and \[\lambda\frac{d}{ds}\xi(s,x,\xi)=-\lambda\frac{\partial q}{\partial x}(x(s,x,\xi),\xi(s,x,\xi))=-\frac{\partial q}{\partial x}(x(s,x,\xi),\lambda\xi(s,x,\xi)).\] It follows that both $(x(s,x,\xi),\lambda\xi(s,x,\xi))$ and $(x(s,x,\lambda\xi),\xi(s,x,\lambda\xi))$ are solutions through the point $(x,\lambda\xi).$ By uniqueness we conclude that \eqref{1.34} holds.
			
			It follows that \begin{align*}f(x_0,\xi_0/|\xi_0|)=&\tilde{g}(x(s^{\ast}-s_0,x_0,\xi_0/|\xi_0|),\xi(s^{\ast}-s_0,x_0,\xi_0/|\xi_0|))\\
			=&|\xi_0|^{m-1}\tilde{g}(x(s^{\ast}-s_0,x_0,\xi_0),\xi(s^{\ast}-s_0,x_0,\xi_0))\\
			=&|\xi_0|^{m-1}|\xi_0^{\ast}|^{1-m}\phi(x_0^\ast)g(x_0^{\ast},\xi_0^{\ast}/|\xi_0^\ast|)>0.\end{align*}
			
			In order to compute $\int_{\Omega\times S^{d-1}}fd\mu,$ we use the notation \[\phi_{s}(x,\xi)=(x(s,x,\xi),\xi(s,x,\xi))\] for $s\in(-\varepsilon,\varepsilon)$ and $(x,\xi)\in B(x_0,4r)\times S^{d-1}$. In particular, \[\phi_{s^\ast-s_0}(x,\xi)=(x(s^\ast-s_0,x,\xi),\xi(s^\ast-s_0,x,\xi))\] and \[0\leq \int_{\Omega\times S^{d-1}}fd\mu= \int_{\Omega\times S^{d-1}}\chi(\tilde{g}\circ\phi_{s^\ast-s_0})d\mu,\] where $\chi\in C^\infty_0(B(x_0,4r))$ is such that $0\leq\chi\leq1$ and $\chi\equiv1$ on a neighborhood of $B[x_0,3r].$
			
			We complete the proof by showing that \[\int_{\Omega\times S^{d-1}}\chi(\tilde{g}\circ\phi_{s^\ast-s_0})d\mu=\int_{\Omega\times S^{d-1}}\chi\tilde{g}d\mu,\] since \[\int_{\Omega\times S^{d-1}}\chi\tilde{g}d\mu\leq\int_{\Omega\times S^{d-1}}\tilde{g}d\mu=0.\]
			
			For $s\in(-\varepsilon,\varepsilon),$ we have
			
			\begin{eqnarray*}
				\frac{d}{ds}(\chi(\tilde{g}\circ\phi_{s}))(x,\xi)&=& \chi(x)\frac{d}{d\tau} \left(\tilde{g}\circ \phi_{s+\tau}\right)(x,\xi)|_{\tau=0}\\
				&=&   \chi(x)\frac{d}{d\tau}\left(\tilde{g}\circ \phi_s \circ\phi_\tau\right)(x,\xi)|_{\tau=0}\\
				&=& \chi(x)\{\tilde{g}\circ \phi_s,q\}(x,\xi).
			\end{eqnarray*}
			
			Observing that $\supp(H_q(\chi)(x,\xi))\subset B(x_0,4r)\setminus B[x_0,3r]$, it follows that the function $(\tilde{g}\circ \phi_s)(x,\xi)H_q(\chi)(x,\xi)$ vanishes identically on $\Omega\times S^{d-1}$. Thus, we also have \begin{align*}\{\chi(\tilde{g}\circ \phi_s),q\}(x,\xi)=&-H_q(\chi(\tilde{g}\circ \phi_s))(x,\xi)\\
			=&-\chi(x)H_q(\tilde{g}\circ \phi_s)-(\tilde{g}\circ \phi_s)(x,\xi)H_q(\chi)(x,\xi)\\
			=&-\chi(x)H_q(\tilde{g}\circ \phi_s)\\
			=&\chi(x)\{\tilde{g}\circ \phi_s,q\}(x,\xi).\end{align*}

			Hence,
			\begin{eqnarray}\label{1.35}
			\frac{d}{ds}\int_{\Omega \times S^{d-1}}\chi(\tilde{g}\circ \phi_s)\,d\mu&=& \int_{\Omega\times S^{d-1}}\{\chi(\tilde{g}\circ \phi_s),q\}\,d\mu\nonumber\\
			&=& \int_{\Omega \times S^{d-1}}\{\chi(\tilde{g}\circ \phi_s),p\}\,d\mu,
			\end{eqnarray} since $H_q=|\xi|^{1-m}H_p$ on the support of $\mu.$
			
			Applying Theorem \ref{Theorem 4.56}, it follows that \[\int_{\Omega \times S^{d-1}}\{\chi(\tilde{g}\circ \phi_s),p\}\,d\mu=0.\]
			
			Therefore, \[\int_{\Omega \times S^{d-1}}\chi(\tilde{g}\circ \phi_s)\,d\mu\] is constant on $(-\varepsilon,\varepsilon).$ Using $s=0$ and $s=s^{\ast}-s_0,$ we obtain
			\begin{equation*}\label{invariante}
			\int_{\Omega\times S^{d-1}} \chi(\tilde{g} \circ \phi_{s^\ast-s_0}) d \mu= \int_{\Omega \times S^{d-1}} \chi\tilde{g} d\mu,
			\end{equation*}
			which completes the proof.
		\end{proof}

		We finish this section by examining the case of the wave equation in an inhomogeneous medium:
		\begin{eqnarray*}
			P(x,D)u=-\rho(x) \partial_t^2u + \sum_{i,j=1}^d \partial_{x_i}( K(x) \partial_{x_j}u).
		\end{eqnarray*}
		whose principal symbol is given by
		\begin{eqnarray}\label{1.36}
		p(t,x,\tau,\xi)=-\rho(x)\tau^2 + \xi^{\top} \cdot K(x) \cdot \xi, \hbox{ where } \xi=(\xi_1,\cdots,\xi_d),
		\end{eqnarray}
		$t \in \mathbb{R}$, $x \in \Omega \subset \mathbb{R}^{d}$, $(\tau,\xi) \in \mathbb{R}^{d+1}$, $\rho \in C^\infty(\Omega)$, $0<\alpha\leq \rho(x)\leq \beta<\infty$, and $K(x)=(k_{ij}(x))_{1 \leq i,j \leq d}$ is a positive-definite matrix satisfying
		\begin{equation*}
		a|\xi|^2 \leq \xi^{\top} \cdot K(x) \xi \leq b |\xi|^2,
		\end{equation*}
		for $0<a<b<\infty$.
		
		Let's describe the bicharacteristics of $p$. From Lemma \ref{v12} the bicharacteristics do not change if we multiply $p$ by a non-zero function. So we can study the Hamiltonian curves of
		\begin{eqnarray}\label{1.37}
		\tilde{p}(t,x,\tau,\xi)=\frac{1}{2}\left(-\tau^2 + \xi^{\top} \cdot \frac{K(x)}{\rho(x)} \cdot \xi\right).
		\end{eqnarray}
		
		\begin{Proposition}\label{Prop 4.63}
			Up to a change of variables, the bicharacteristics of (\ref{1.37}) are curves of the form
			\begin{eqnarray*}
				s \mapsto \left(s, y(s), \tau, -\tau\left(\frac{K(y(s))}{\rho(y(s))}\right)^{-1}\dot{y}(s) \right),
			\end{eqnarray*}
			where $s \mapsto y(s)$ is a geodesic of the metric $G=\left( \frac{K}{\rho}\right)^{-1}$ on $\Omega$, parameterized by the arc length.
		\end{Proposition}
		\begin{proof}
			Let's define a curve
			\begin{equation*}
			\begin{array}{cccl}
			\alpha:& I& \longrightarrow& \Omega \times \mathbb{R}^{d}\setminus\{0\}\\
			{} & s & \mapsto & (t(s), x(s), \tau(s), \xi(s))
			\end{array}
			\end{equation*}
			such that
			\begin{equation}\label{1.38}
			\left\{
			\begin{aligned}
			\dot{t}(s) &= \frac{\partial \tilde p}{\partial \tau}(\alpha(s)),\\
			\dot{x}_j(s) &= \frac{\partial \tilde p}{\partial \xi_j}(\alpha(s)), j = 1, \cdots,d,\\
			\dot{\tau}(s) &=-\frac{\partial \tilde p}{\partial t}(\alpha(s)),\\
			\dot{\xi}_j(s)&=-\frac{\partial \tilde p}{\partial x_j}(\alpha(s)),~j= 1, \cdots,d.
			\end{aligned}
			\right.
			\end{equation}
			
			From \eqref{1.37} and $\eqref{1.38}_1$, we have
			\begin{equation*}
			\dot{t}(s) = - \tau(s).
			\end{equation*}
			
			Also, from \eqref{1.37} and $\eqref{1.38}_2$ it follows that
			\begin{equation}\label{1.39}
			\dot{x}(s) = \frac{ K(x(s))}{\rho(x(s))} \cdot \xi(s).
			\end{equation}

			Equation \eqref{1.37} and $\eqref{1.38}_3$ ensure that
			\begin{equation*}
			\dot{\tau}(s) = - \frac{\partial \tilde{p}}{\partial t} (\alpha(s)) = 0.
			\end{equation*}
			
			Finally, from $\eqref{1.38}_3$ we obtain
			\begin{equation*}
			\dot{\xi}_j(s) =-\frac{\partial \tilde{p}}{\partial x_j}(\alpha(s)) =- \frac{1}{2} \sum_{ 1 \leq i, j \leq d} \frac{\partial}{\partial x_j} \left( \frac{k_{i j}(x(s))}{\rho(x(s))} \right) \xi_i(s) \cdot \xi_j(s).
			\end{equation*}
			
			Thus, the sought after curve must satisfy the equations
			\begin{equation}\label{1.40}
			\left\{
			\begin{aligned}
			&\dot{t}(s) = -\tau(s),\\
			&\dot{x}(s) = \frac{K(x(s))}{\rho(x(s))} \cdot \xi(s),\\
			&\dot{\tau}(s) = 0,\\
			& \dot{\xi}(s) = - \frac{1}{2} {(\xi(s))}^{\top} \cdot \nabla \left(\frac{K(x(s))}{\rho(x(s))}\right) \cdot \xi(s)
			\end{aligned}
			\right.
			\end{equation}
			
			Introducing the matrix $G(x(s)) = \left( \frac{K(x(s))}{\rho(x(s))}\right)^{-1},$ from the second equation of (\ref{1.40}) we obtain
			\begin{equation}\label{1.41}
			\xi(s) = G(x(s)) \dot{x}(s).
			\end{equation}
			
			Let $f:\operatorname{GL}(n,\mathbb{R}) \longrightarrow \operatorname{GL}(n,\mathbb{R})$ be defined by $f(X) = X^{-1}$ so that its derivative is given by $d (f)_{A} (V) = - A^{-1} \cdot V \cdot A^{-1}.$ So,
			\begin{equation*}
			\begin{aligned}
			\frac{\partial G}{\partial x_k} (x) &= d \left(f \circ \frac{K}{\rho}\right)_{x}(e_k)\\
			&= d(f)_{\left(\frac{K}{\rho}\right)(x)} \circ d \left( \frac{K}{\rho}\right)_{x} (e_k)\\
			&= d(f)_{\left(\frac{K}{\rho}\right)(x)} \left(\frac{\partial}{\partial x_k} \left( \frac{K(x)}{\rho(x)}\right) \right)\\
			&= - \left(\frac{K(x)}{\rho(x)}\right)^{-1} \cdot \frac{\partial}{\partial x_k} \left( \frac{K(x)}{\rho(x)}\right) \cdot \left(\frac{K(x)}{\rho(x)}\right)^{-1}.
			\end{aligned}
			\end{equation*}
			
			That is,
			\begin{equation}\label{1.42}
			\frac{\partial}{\partial x_k} \left(\frac{K(x)}{\rho(x)}\right) = - \frac{K(x)}{\rho(x)} \frac{\partial G}{\partial x_k}(x) \frac{K(x)}{\rho(x)}.
			\end{equation}
			
			Therefore, from $(\ref{1.41})$ and $(\ref{1.42})$, we obtain
			\begin{equation}\label{1.43}
			\begin{aligned}
			\dot{\xi}_k &= - \frac{1}{2} \sum_{1 \leq i,j \leq d} \frac{\partial}{\partial x_k} \left(\frac{k_{ij}(x(s)}{\rho(x)})\right) \xi_{i}(s) \cdot \xi_k(s)\\&= -\frac{1}{2} \xi^{\top} \cdot \frac{\partial}{\partial x_k} \left(\frac{K(x)}{\rho(x)}\right) \cdot \xi\\
			&= - \frac{1}{2} (G(x)\dot{x} )^{\top} \frac{\partial}{\partial x_k} \left(\frac{K(x)}{\rho(x)}\right) \cdot G(x)\dot{x} \\
			&= - \frac{1}{2} \dot{x}^{\top}  G(x)^{\top} \cdot \frac{\partial}{\partial x_k} \left(\frac{K(x)}{\rho(x)}\right) \cdot G(x) \dot{x}\\
			& = - \frac{1}{2} \dot{x}^{\top} G(x) \cdot \left( - \frac{K(x)}{\rho(x)}\cdot \frac{\partial G}{\partial x_k}(x) \cdot \frac{K(x)}{\rho(x)}\right) \cdot G(x) \dot{x}\\
			&= \frac{1}{2} \dot{x}^{\top}  G(x) \cdot \frac{K(x)}{\rho(x)}\cdot \frac{\partial G}{\partial x_k}(x) \cdot \frac{K(x)}{\rho(x)} \cdot G(x)  \dot{x}\\
			& = \frac{1}{2} \dot{x}^{\top} \cdot \frac{\partial G}{\partial x_k}(x) \cdot \dot{x}.
			\end{aligned}
			\end{equation}
			
			Thus, from (\ref{1.41}) and (\ref{1.43}), we have
			\begin{equation}\label{1.44}
			(G(x) \cdot \dot{x})^{\boldsymbol{\cdot}} = \dot{\xi} = \frac{1}{2} \dot{x}^{\top} \cdot \nabla G(x) \cdot \dot{x}.
			\end{equation}
			
			Since $\tilde{p}$ is 0 over each null bicharacteristic, it follows that
			\begin{equation*}
			\xi(s)^{\top} \cdot \frac{K(x(s))}{\rho(x(s))} \cdot \xi(s) = \tau(s)^{2}.
			\end{equation*}
			
			And since $\dot{\tau}(s) = 0,$ we have $\tau(s) = \tau,$ for some $\tau \in \mathbb{R}.$ Thus,
			\begin{equation*}
			\xi(s)^{\top} \cdot \frac{K(x(s))}{\rho(x(s))} \cdot \xi(s) = \tau^2.
			\end{equation*}
			
			On the other hand, from (\ref{1.39})
			\begin{equation}\label{1.45}
			\begin{aligned}
			\dot{x}(s)^{\top} \cdot G(x(s)) \cdot \dot{x}(s) &= \left( \frac{K(x(s))}{\rho(x(s))} \cdot \xi(s)\right)^{\top} \cdot G(x(s)) \cdot \left(\frac{K(x(s))}{\rho(x(s))} \cdot \xi(s) \right)\\
			&= \xi^{\top} \cdot \left(\frac{K(x(s))}{\rho(x(s))}\right)^{\top} \cdot G(x(s)) \cdot \frac{K(x(s))}{\rho(x(s))} \cdot \xi(s)\\
			&= \xi(s)^{\top}\cdot \frac{K(x(s))}{\rho(x(s))} \cdot \xi(s)\\
			&= \tau^2,
			\end{aligned}
			\end{equation}
			that is, $\dot{x}(s)^{\top} \cdot G(x(s)) \cdot x(s)$ over each null bicharacteristic of $\tilde{p}$. By (\ref{1.44}) and (\ref{1.45}) we can write
			\begin{equation}\label{1.46}
			\frac{d}{ds} \frac{G(x(s)) \cdot \dot{x}(s)}{\sqrt{\dot{x}(s)^{\top} \cdot G(x(s)) \cdot \dot{x}(s)}} = \frac{1}{2} \frac{ \dot{x}(s)^{\top} \cdot \nabla G(x(s)) \cdot \dot{x}(s)}{\sqrt{\dot{x}(s)^{\top} \cdot G(x(s)) \cdot \dot{x}(s)}}.
			\end{equation}
			
			Defining the arc length functional $L$ by
			\begin{equation*}
			L(x, \dot{x}) = \sqrt{\dot{x}^{\top} \cdot G(x) \cdot {x}},
			\end{equation*}
			from (\ref {1.46}) we obtain
			\begin{equation*}
			\frac{d}{ds} \frac{\partial}{\partial \dot{x}} L(x, \dot{x}) = \frac{ \partial L}{\partial x}(x, \dot{x}).
			\end{equation*}
			that is, $L$ satisfies the Euler Lagrange equations, so $x$ is a geodesic.
			
			Therefore, the bicharacteristics are curves of the form
			\begin{equation*}
			\gamma(s) = (- \tau s, x(s), \tau, G(x(s)) \dot{x}(s)).
			\end{equation*}
			
			Consider the function $\alpha(s) = - \frac{s}{\tau},$ thus $\alpha$ is the arc length parameter for $x.$ In fact, it suffices to show that for $y(s) = (x \circ \alpha)(s)$ we have
			\begin{equation*}
			\|\dot{y}(s)\|_{y(s) = 1.}
			\end{equation*}
			
			Indeed,
			\begin{equation*}
			\begin{aligned}
			\|\dot{y}(s)\|_{y(s)} &= \|(x \circ \alpha)^{\boldsymbol{\cdot}}(s)\|_{x(\alpha(s))}\\
			&= (x \circ \alpha)^{\boldsymbol{\cdot}}(s)^{\top} \cdot G(x(\alpha(s))) \cdot (x \circ \alpha)^{\boldsymbol{\cdot}}(s)\\
			&= (\dot{x}(\alpha(s))\dot{\alpha}(s))^{\top} \cdot G(x(\alpha(s))) \cdot \dot{x}(\alpha(s)) \dot{\alpha}(s)\\
			&= \left(\dot{x}(\alpha(s))\left(-\frac{1}{\tau}\right) \right)^{\top} \cdot G(x(\alpha(s))) \cdot \dot{x}(\alpha(s)) \left(- \frac{1}{\tau}\right)\\
			&= - \frac{1}{\tau} \dot{x}(\alpha(s)) \cdot G(x(\alpha(s))) \cdot \dot{x}(\alpha(s)) \left(- \frac{1}{\tau}\right)\\
			&= \frac{1}{\tau^{2}} \tau^{2}\\
			&=1.
			\end{aligned}
			\end{equation*}
			
			Note that, $(x \circ \alpha)^{\boldsymbol{\cdot}}(s) = \dot{x}(\alpha(s)) \dot{\alpha}(s) = - \frac{\dot{x}(\alpha(s))}{\tau}.$ So, $\dot{x}(\alpha(s)) = - \tau(x \circ \alpha)^{\boldsymbol{\cdot}}(s).$ Thus,
			\begin{equation*}
			\begin{aligned}
			(\gamma \circ \alpha)(s) &= \gamma(\alpha(s))\\
			&= (- \tau \alpha(s), x(\alpha(s)), \tau, G(x(\alpha(s)))\dot{x}(\alpha(s))\\
			&= (- \tau\left(- \frac{s}{\tau}\right), y(s), \tau, - \tau G(y(s)) \dot{y}(s))\\
			&= (s, y(s), \tau, - \tau G(y(s)) \dot{y}(s)).
			\end{aligned}
			\end{equation*}
			
			Therefore, less than one reparametrization, the bicharacteristics of $p$ are curves of the form
			\begin{equation*}
			s \mapsto (s, y(s), \tau, - \tau G(y(s)) \dot{y}(s))
			\end{equation*}
			where $y(s)$ is a geodesic of the $G$ metric parameterized by the arc length.
			
			Conversely, consider $s \mapsto y(s)$ a geodesic of the $G$ metric parameterized by the arc length. Let us show that, less than one reparameterization of the curve $\gamma$ given by
			\begin{equation*}
			s \mapsto (s, y(s), \tau, - \tau G(y(s)) \dot{y}(s))
			\end{equation*}
			is a bicharacteristics of $ p. $ For this, we must prove that the component functions satisfy the Hamilton-Jacobi equations.
			
			Let $ \beta(s) = - \tau s, $ then defining $x(s) = y(\beta(s)),$ we have
			\begin{equation*}
			\begin{aligned}
			\gamma(\beta(s)) &= (\beta(s), y(\beta(s)), \tau, - \tau G(y(\beta(s))) \dot{y}(\beta(s))\\
			&= (- \tau s, x(s), \tau, - \tau G(x(s)) (y \circ \beta)^{\boldsymbol{\cdot}}(s) \frac{1}{\dot{\beta}(s)} )\\
			&= (- \tau s, x(s), \tau, G(x(s)) \dot{x}(s)).
			\end{aligned}
			\end{equation*}
			
			Let $t(s) = - \tau s$ and $\xi(s) = G(x(s)) \dot{x}(s).$ Thus,
			\begin{equation}\label{1.47}
			\begin{aligned}
			\dot{t}(s) &= - \tau\\
			\dot{x}(s) &= \frac{K(x(s))}{\rho(x(s))} \cdot \xi(s)\\
			\dot{\tau}(s) &= 0.
			\end{aligned}
			\end{equation}
			
			It remains to show that
			\begin{equation*}
			\dot{\xi}(s) = - \frac{1}{2} \xi(s)^{\top} \cdot \nabla \left(\frac{K(x(s))}{\rho(x(s))} \right) \cdot \xi(s),
			\end{equation*}
			which is equivalent to
			\begin{equation*}
			(G(x(s) \dot{x}(s))^{\boldsymbol{\cdot}} = \frac{1}{2} \dot{x}(s) \cdot \nabla G(x(s)) \cdot \dot{x}(s).
			\end{equation*}
			
			Since $y(s)$ is parameterized by arc length, we have
			\begin{equation*}
			\| \dot{y}(s)\|_{y(s)} = 1,
			\end{equation*}
			that is,
			\begin{equation*}
			1 = \dot{y}(s)^{\top} \cdot G(y(s))\cdot \dot{y}(s).
			\end{equation*}
			
			From the Euler-Lagrange equations,
			\begin{equation*}
			(G(y(s)) \dot{y}(s))^{\boldsymbol{\cdot}} = \frac{1}{2} (\dot{y}(s)^{\top} \cdot \nabla G(y(s)) \cdot \dot{y}(s)).
			\end{equation*}
			
			Hence, taking $f(s) = G(y(s)) \dot{y}(s)$, we have
			\begin{equation}\label{1.48}
			(f \circ \beta)^{\boldsymbol{\cdot}}(s) = \dot{f}(\beta(s)) \dot{\beta}(s) = \frac{1}{2} \left(\dot{y} (\beta(s))^{\top} \cdot \nabla G(y(\beta(s))) \cdot \dot{y}(\beta(s))\right) \dot{\beta}(s).
			\end{equation}
			
			Thus, from \eqref{1.48} and recalling that $\dot{\beta}(s) = - \tau$, we obtain
			\begin{equation*}
			\begin{aligned}
			(G(x(s)) \dot{x}(s))^{\boldsymbol{\cdot}} &= (G(y(\beta(s))(y \circ \beta)^{\boldsymbol{\cdot}} (s))^{\boldsymbol{\cdot}}\\
			&= (G(y(\beta(s))\dot{y}(\beta(s))\dot{\beta}(s))^{\boldsymbol{\cdot}}\\
			&= - \tau (G(y (\beta(s)) \dot{y}( \beta(s)))^{\boldsymbol{\cdot}}\\
			&= - \tau \frac{1}{2} \left(\dot{y} (\beta(s))^{\top} \cdot \nabla G(y(\beta(s))) \cdot \dot{y}(\beta(s))\right) \beta^{\boldsymbol{\cdot}}(s)\\
			&= \tau^{2} \frac{1}{2} \left( \frac{\dot{x}(s)}{\tau}^{\top} \cdot \nabla G(x(s)) \cdot \frac{\dot{x}(s)}{\tau}\right)\\
			&= \frac{1}{2} (\dot{x}(s)^{\top} \cdot \nabla G(x(s)) \cdot \dot{x}(s)).
			\end{aligned}
			\end{equation*}
			
			Therefore,
			\begin{equation*}
			\dot{\xi}(s) = \frac{1}{2} (\dot{x}(s))^{\top} \cdot \nabla G(x(s)) \cdot \dot{x}(s)),
			\end{equation*}
			which ends the proof.
		\end{proof}

	\end{document}